\documentclass[12pt]{article}

\RequirePackage{amsthm,amsmath,amsfonts,amssymb}
\RequirePackage[numbers,sort&compress]{natbib}
\RequirePackage{graphicx}

\usepackage{dsfont}
\usepackage[colorlinks=false,linkcolor=darkblue]{hyperref}
\usepackage{enumerate} 
\usepackage{makecell}
\usepackage{gensymb}
\usepackage{booktabs} 
\usepackage{xcolor}   
\usepackage{pifont}   

\usepackage{epsfig}
\usepackage[numbers]{natbib}
\usepackage{color}
\usepackage{dsfont}
\usepackage{bm}
\usepackage{verbatim}
\usepackage{stmaryrd}
\usepackage{appendix}
\usepackage{bbm}			
\usepackage{enumerate} 
\usepackage{nicefrac}
\usepackage{xfrac}
\usepackage{units} 
\usepackage{mathtools}
\usepackage{pdfcomment}
\usepackage{url}
\usepackage[labelfont={bf,sf},font={small},%
  labelsep=space]{caption}
\usepackage{acronym}
\usepackage{mathrsfs}
\usepackage{fancyhdr}
\usepackage[T1]{fontenc}
\usepackage{caption, booktabs}
\usepackage{textcomp}
\usepackage{multirow}
\usepackage{lipsum}
\usepackage{makecell}

\usepackage{tikz}
\usetikzlibrary{positioning,arrows}
\usetikzlibrary{arrows.meta,decorations.markings}
\usetikzlibrary{matrix}
\usepackage{hologo}
\usepackage{xcolor}
\usetikzlibrary{matrix,fit}
\tikzset{%
  mleftdelimiter/.style={inner ysep=0pt, inner xsep=1ex,left delimiter=\{,label={[label distance=3mm]left:#1}}
}

\allowdisplaybreaks

\theoremstyle{plain}
\newtheorem{theorem}{Theorem}[section]
\newtheorem{proposition}[theorem]{Proposition}
\newtheorem{corollary}[theorem]{Corollary}
\newtheorem{lemma}[theorem]{Lemma}

\theoremstyle{remark}
\newtheorem{remark}[theorem]{Remark}

\newtheorem{assumption}[theorem]{Assumption}
\newtheorem{example}[theorem]{Example}

\newcommand{\cmark}{\textcolor{green}{\ding{51}}} 
\newcommand{\xmark}{\textcolor{red}{\ding{55}}}   

\newcommand{\E}{\mathbb{E}}

\newcommand{\N}{\mathbb{N}}

\newcommand{\R}{\mathbb{R}}

\newcommand{\UU}{\mathbf{U}}
\newcommand{\vv}{\mathbf{v}}
\newcommand{\VV}{\mathbf{V}}
\newcommand{\WW}{\mathbf{W}}

\newcommand{\XX}{{\mathbf{X}}}
\newcommand{\YY}{\mathbf{Y}}
\newcommand{\ZZ}{{\mathbf{Z}}}

\newcommand{\xx}{\mathbf{x}}
\newcommand{\yy}{\mathbf{y}}
\newcommand{\zz}{\mathbf{z}}

\newcommand{\cD}{\mathcal{D}}

\newcommand{\cL}{\mathcal{L}}

\newcommand{\cA}{\mathcal{A}}

\newcommand{\cS}{\mathcal{S}}

\newcommand{\eqd}{\stackrel{\mathrm{d}}=}

\newcommand{\1}{\mathds{1}}

\newcommand{\de}{\mathrm{\,d}}

\newcommand{\Ran}{\mathsf{Ran}}

\newcommand{\rank}{\mathsf{rank}}

\newcommand{\Cor}{\mathrm{Cor}}
\newcommand{\var}{\mathrm{Var}}

\newcommand{\Var}{\mathrm{Var}}
\newcommand{\cov}{\mathrm{Cov}}

\DeclareMathOperator{\Cov}{Cov}

\DeclareMathOperator{\v@r}{V@R}

\DeclareMathOperator{\av@r}{AV@R}

\definecolor{light-gray}{gray}{0.95}
\definecolor{darkblue}{rgb}{0,0,.5}
\definecolor{foxred}{rgb}{0.7, 0.11, 0.11}

\begin{document}

\def\spacingset#1{\renewcommand{\baselinestretch}%
{#1}\small\normalsize} \spacingset{1}


\title{\bf A direct extension of Azadkia \& Chatterjee's rank correlation to multi-response vectors}
  %
  \author{Jonathan Ansari\hspace{.2cm}
    \\
    Department for Artificial Intelligence and Human Interfaces, \\
    University of Salzburg \\
    and \\
    Sebastian Fuchs \\
    Department for Artificial Intelligence and Human Interfaces, \\
    University of Salzburg}
\maketitle

\bigskip
\begin{abstract}
Recently, \citet{chatterjee2023} recognized the lack of a direct generalization of his rank correlation $\xi$ in \citet{chatterjee2021} to a multi-dimensional response vector. As a natural solution to this problem, we here propose an extension of $\xi$ to a set of $q \geq 1$ response variables, where our approach builds upon converting the original vector-valued problem into a univariate problem and then applying the rank correlation $\xi$ to it. Our novel measure $T$ quantifies the scale-invariant extent of functional dependence of a response vector $\mathbf{Y} = (Y_1,\dots,Y_q)$ on predictor variables $\mathbf{X} = (X_1, \dots,X_p)$, characterizes independence of $\mathbf{X}$ and $\mathbf{Y}$ as well as perfect dependence of $\mathbf{Y}$ on $\mathbf{X}$ and hence fulfills all the characteristics of a measure of predictability. Aiming at maximum interpretability, we provide various invariance results for $T$ as well as a closed-form expression in multivariate normal models. Building upon the graph-based estimator for $\xi$ in \citep{chatterjee2021}, we obtain a non-parametric, strongly consistent estimator for $T$ and show---as a main contribution---its asymptotic normality. 
Based on this estimator, we develop a model-free and rank-based feature ranking and forward feature selection for multiple-outcome data that works without any tuning parameters. Simulation results and real case studies illustrate $T$'s broad applicability.
\end{abstract}

\noindent%
{\it Keywords:}  
conditional dependence,
information gain inequality,
multi-output feature selection,
nonparametric measures of association
\vfill

\newpage

\section{Introduction}

In regression analysis the main objective is to estimate the functional relationship $\YY={\bf f}(\XX,\boldsymbol{\varepsilon})$  between a set of $q \geq 1$ response variables $\YY = (Y_1, \dots, Y_q)$ and a set of $p \geq 1$ predictor variables $\XX = (X_1, \dots, X_p)$ where \(\boldsymbol{\varepsilon}\) is a model-dependent error.
In view of constructing a good model,
the question naturally arises to what extent \(\YY\) can be predicted from the information provided by the multivariate predictor variable \(\XX\), and which of the predictor variables \(X_1, \dots, X_p\) are relevant for the model at all.

We refer to $\kappa$ as a \emph{measure of predictability} 
for the random vector \(\YY\) given the random vector \(\XX\)
if it satisfies the following axioms, cf. \citep{Renyi-1959} and \citep{siburg2013}: 
\begin{enumerate}[({A}1)]
\item \label{prop1} $0 \leq \kappa(\YY,\XX) \leq 1$,
\item \label{prop2} $\kappa(\YY,\XX) = 0$ if and only if $\YY$ and $\XX$ are independent,
\item \label{prop3} $\kappa(\YY,\XX) = 1$ if and only if $\YY$ is perfectly dependent on $\XX$,
i.e., there exists some measurable function ${\bf f}: \R^p \to \R^q$ such that $\YY = {\bf f}(\XX)$ almost surely.
\end{enumerate}
In addition to the above-mentioned three axioms, it is desirable that additional information improves the predictability of \(\YY\,.\)
This yields the following two closely related properties 
which appear to be crucial for this paper
(cf., e.g.,  \citep{chatterjee2021,sfx2022phi,fgwt2021,deb2020b,strothmann2022}): 
\begin{enumerate}[({P}1)]
\item \label{prop.IGI} \emph{Information gain inequality}: 
$\kappa(\YY,\XX) \leq \kappa(\YY,(\XX,\ZZ))$ for all $\XX$, $\ZZ$ and $\YY$.
\item \label{prop.CI} \emph{Characterization of conditional independence}:
$\kappa(\YY,\XX) = \kappa(\YY,(\XX,\ZZ))$ if and only if $\YY$ and $\ZZ$ are conditionally independent given $\XX$.
\end{enumerate}

To the best of our knowledge and according to \citet{chatterjee2023},
so far the only measure of predictability applicable to a vector 
$\YY = (Y_1, \dots,Y_q)$ of $q > 1$ response variables has been introduced by \citet{deb2020b} and employs the vector-valued structure of $\YY$ for its evaluation.
In contrast to that, in the present paper, we take a different approach by converting the original vector-valued problem into a univariate problem and then applying to it a measure of predictability for a single response variable capable of characterizing conditional independence.
A particularly suitable candidate for such a single response measure has been recently introduced by \citet{chatterjee2021}:
Their so-called `simple measure of conditional dependence' $\xi$ is defined for \(q=1\) by
\begin{align} \label{Tcond}
  \xi(Y, \XX | \ZZ)
	:= \frac{\int_{\mathbb{R}} \E( {\rm var} (P(Y \geq y \, | \, \XX,\ZZ) \,|\, \ZZ) ) \; \mathrm{d} P^{Y}(y)}
					{\int_{\mathbb{R}} \E( {\rm var} (\mathds{1}_{\{Y \geq y\}} \,|\, \ZZ) ) \; \mathrm{d} P^{Y}(y)}\,,
\end{align}
with its unconditional counterpart denoted by 
\(\xi(Y,\XX) 
  := \xi(Y,\XX|\emptyset).
\) 
The functional \(\xi\) in \eqref{Tcond} has attracted a lot of attention in the past few years; see, e.g., 
\citep{auddy2021,bickel2022,bickel2020,deb2020,bernoulli2021,fan2022A,deb2020b,han2022limit,shi2021normal,shi2021,strothmann2022,wiesel2022}. 
Due to the variance decomposition, \(\xi\) can be expressed in terms of its unconditional counterpart as
\begin{align} \label{Tcond-uncond}
  \xi(Y, \XX | \ZZ) 
  & = \frac{\xi(Y, (\XX,\ZZ) \, | \, \emptyset) - \xi(Y, \ZZ \, | \, \emptyset)}{1 - \xi(Y, \ZZ \, | \, \emptyset)}\,,
\end{align} 
thus bringing the investigation of the unconditional version 
\begin{align} \label{Tuniv}
  \xi(Y,\XX) = \xi(Y,\XX|\emptyset) 
  & = \frac{\int_{\mathbb{R}} {\rm var} (P(Y \geq y \, | \, \XX)) \; \mathrm{d} P^{Y}(y)}
					{\int_{\mathbb{R}} {\rm var} (\mathds{1}_{\{Y \geq y\}}) \; \mathrm{d} P^{Y}(y)}
\end{align}
to the fore, that is based on \citep{chatterjee2020, siburg2013} and is also known as \emph{Dette-Siburg-Stoimenov's} dependence measure. $\xi$ in \eqref{Tuniv} captures the variability of the conditional distributions in various ways \citep{Ansari-LFT-2023}. It quantifies the scale-invariant extent of \emph{functional} (or \emph{monotone regression}) \emph{dependence} of the single response variable $Y$ (i.e., $q=1$) on the predictor variables $X_1, \dots, X_p$ and fulfills the above-mentioned characteristics of a measure of predictability for a single response variable.

As mentioned by \citet{chatterjee2023}, a direct generalization of \(\xi\) in \eqref{Tuniv} to higher-dimensional spaces, i.e., to arbitrary \(q \in \mathbb{N}\), has not been proposed so far:
A naive way of extending $\xi$ to 
a vector $(Y_1, \dots,Y_q)$ of $q > 1$ response variables 
would be to sum up the individual amounts of predictability, namely, to consider the quantity
\begin{equation} \label{TSigma}
  T^\Sigma (\YY,\XX) := \frac{1}{q} \sum_{i=1}^{q} \xi(Y_i,\XX).
\end{equation}
It is clear that $T^\Sigma$ satisfies axioms (A\ref{prop1}) and (A\ref{prop3}).
However, it fails to characterize independence between the vectors $\XX$ and $\YY$;
see, e.g., Example \ref{Ex.Cube} in the Supplementary Material.
A more promising approach involves combining $\xi$ in \eqref{Tcond} and 
the chain rule for conditional independence to define the quantity
\begin{align}\label{KappaAverage}
  \kappa^{\boldsymbol{\alpha}} (\YY,\XX) 
  & := \sum_{i=1}^{q} \alpha_i \, \xi(Y_i, \XX \,|\, (Y_{i-1},\dots,Y_{1}))\,, \quad \sum_{i=1}^q \alpha_i =1\,, ~\boldsymbol{\alpha} = (\alpha_1,\ldots,\alpha_q)\in (0,1)^q\,,
\end{align}
that turns out to be a proper measure of predictability in the sense of axioms (A\ref{prop1}) to (A\ref{prop3}).
As a consequence of Eq. \eqref{Tcond-uncond},
$\kappa^{\boldsymbol{\alpha}}$ can be written as
\begin{align} \label{KappaAverage2}
  \kappa^{\boldsymbol{\alpha}} (\YY,\XX) 
  & = \sum_{i=1}^{q} \alpha_i \; 
  \frac{\xi(Y_i, (\XX,Y_{i-1},\dots,Y_{1})) - \xi(Y_i, (Y_{i-1},\dots,Y_{1}))}
  {1 - \xi(Y_i, (Y_{i-1},\dots,Y_{1}))}\,.
\end{align}
However, if the weights $\alpha_i$ are static (i.e., they neither depend on $\XX$ nor $\YY$), then $\kappa^{\boldsymbol{\alpha}}$ suffers from two severe disadvantages: 
(i) Whenever there exists some $i \in \{1,\dots,d-1\}$ such that $Y_i$ is perfectly dependent on $(Y_{i-1},\dots,Y_{1})$, then $\kappa^{\boldsymbol{\alpha}}$ is not defined;
(ii) The estimator of $\kappa^{\boldsymbol{\alpha}}$ may become extremely sensitive to the dependence structure among the response variables $\YY$ due to the denominator in \eqref{KappaAverage2}.
Figure \ref{Fig.Sim} in the Supplementary Material illustrates this sensitivity for the multivariate normal distribution (i.e. for continuous data),
for which the nearest neighbor-based estimator of \(\kappa^{\boldsymbol{\alpha}}\) may attain arbitrarily large negative values. Due to the lack of interpretability of these values, any practical use of convex combinations \(\kappa^{\boldsymbol{\alpha}}\) with static weights \(\boldsymbol{\alpha}\) becomes obsolete.
Instead, choosing in \eqref{KappaAverage2} the weights 
\begin{align*}
  \alpha_i (\YY) 
  & := \frac{1 - \xi(Y_i,(Y_{i-1},\dots,Y_{1}))}
            {\sum_{i=1}^q \big[ 1 - \xi(Y_i,(Y_{i-1},\dots,Y_{1})) \big]}
\end{align*}
incorporating dependencies among the responses, yields the functional \(T\) defined by 
\begin{align}
  \label{defmdm}
  T (\YY,\XX)
  & := \sum_{i=1}^q \; \frac{\xi(Y_i,(\XX,Y_{i-1},\dots,Y_{1}))-\xi(Y_i,(Y_{i-1},\dots,Y_{1}))}
           {\sum_{k=1}^q \big[ 1 - \xi(Y_k,(Y_{k-1},\dots,Y_{1})) \big]}
  \\ 
  \label{defmdm2}	 
  & \phantom{:} = 1 - \frac{q - \sum_{i=1}^{q} \xi(Y_i, (\XX,Y_{i-1},\dots,Y_{1}))}{q - \sum_{i=1}^{q} \xi(Y_i, (Y_{i-1},\dots,Y_{1}))}\,, 
   \qquad \text{with}~~~ \xi(Y_1,\emptyset):=0\,.
\end{align}
Interestingly, \(T(\YY,\XX)\) is a measure of predictability with various outstanding properties:
 \begin{enumerate}[(1)]
    \item $T(\YY,\XX)$ is defined for any random vectors \(\YY\) and \(\XX\), in particular, for any type of dependence that may occur among the response variables. The only inevitable restriction is that the components of \(\YY\) are non-degenerate.
    \(T\) exhibits a simple expression, is merely rank-based and thus  fully non-parametric without any tuning parameters.
     For the multivariate normal distribution, we obtain a closed-form expression for \(T\) (Proposition \ref{propChatformmGau}).
    Further, \(T\) fulfills the information gain inequality (P\ref{prop.IGI}), 
    characterizes conditional independence (P\ref{prop.CI}),  
    satisfies the so-called 
    data processing inequality (see, e.g., \citep{cover2006}),
    is self-equitable (cf., e.g., \citep{kinney2014}),
    and exhibits numerous invariance properties (Subsections \ref{Main.FundProp.} and \ref{Main.Invariance}).
    \item \(T\) has a strongly consistent, nearest neighbor-based estimator \(T_n\) which can be computed in $O(n \log n)$ time and which is given by a transformation of Azadkia \& Chatterjee's graph-based estimator $\xi_n$ for $\xi$ (Theorem \ref{theT.Consistency}).
    The estimator \(T_n\) is not affected by the extreme sensitivity to the dependence structure of \(\YY\) as observed for static/general convex combinations \(\kappa^{\boldsymbol{\alpha}}\) (see again Figure \ref{Fig.Sim}).
    As a main contribution of this paper (Theorem \ref{AN.Thm:AN}), we prove asymptotic normality of \(\sqrt{n} (T_n - \E[T_n])/\sqrt{\Var(T_n)}\) noting that such a result has not been proved for related multi-output measures like the kernel partial correlation \cite{deb2020b}. Our technical proof in Section \ref{App.Sect.AN} extends the ideas in \citet{han2022limit} by using a modification of the nearest neighbor-based normal approximation in \citet[Theorem 3.4.]{chatterjee2008}. 
    In Section \ref{secroc}, we give a rate of convergence for the bias of \(T_n\,.\)
    \item As an important application, \(T\) allows for a model-free, merely rank-based feature ranking and forward feature selection without any tuning parameters for data with multiple outcomes. 
    In particular, our algorithm which we call \emph{multivariate feature ordering by conditional independence} (MFOCI) extends the variable selection algorithm FOCI in \citep{chatterjee2021} to multi-output data, noting that related model-free methods such as KFOCI \cite{deb2020b} depend on various tuning parameters.
    MFOCI is consistent in the sense that the subset of selected predictor variables is sufficient with high probability, thus facilitating the identification of the relevant predictor variables, see Proposition \ref{FS.Consistency}.
    An implementation of MFOCI is provided by the R package \texttt{didec}, available on CRAN \citep{fuchs2024R}, where also a new variable clustering method 
    is implemented.
    Several simulation studies and real-data examples for multi-response data from medicine, meteorology and finance illustrate \(T\)'s broad applicability and the superior performance of MFOCI in various settings in comparison to existing procedures, see Section \ref{secappl} and Section \ref{App.Add.secappl} in the Supplementary Material. 
 \end{enumerate}
 
A comparison of \(T\) with competing multi-output dependence measures such as the kernel partial correlation (KPC) in \citep{deb2020b} and the distance correlation (dCor) in \citep{Szekely-2007} is given in Section \ref{seccomparison} and Table \ref{tabpropkpct}.
     
Additional results including a geometric illustration of $T$'s most important properties and the construction principle underlying $T$ are available in the Supplementary Material. 
Throughout the paper, we assume that \(\XX=(X_1,\ldots,X_p)\,,\) \(\YY=(Y_1,\ldots,Y_q)\,,\) and \(\ZZ=(Z_1,\ldots,Z_r)\) are \(p\)-, \(q\)-, and \(r\)-dimensional random vectors, respectively, defined on a common probability space \((\Omega,\cA,P)\,,\) with \(p,q,r\in \N=\{1,2,\ldots\}\) being arbitrary. Variables not in bold type refer to one-dimensional real-valued quantities.
We denote \(\YY\) as the response vector which is always assumed to have non-degenerate components, i.e., for all \(i\in \{1,\ldots,q\}\,,\) the distribution of \(Y_i\) does not follow a one-point distribution. 
This equivalently means that \(\var(Y_i)>0\,,\) which ensures that \(\xi(Y_i,\,\cdot)\) in \eqref{Tuniv} and, hence, \(T(\YY,\,\cdot\,)\) in \eqref{defmdm} are well-defined. Note that the assumption of non-degeneracy is inevitable because, otherwise, \(Y_i\) is constant and thus both independent from \(\XX\) and a function of \(\XX\,.\) However, in this case, there is no measure satisfying both the axioms (A\ref{prop2}) and (A\ref{prop3}).

\section{Properties of the Extension \(T\)}\label{secmainres}   

In this section, various basic properties of the measure \(T\) in \eqref{defmdm} are established.
The first part focuses on fundamental characteristics, 
showing that \(T\) can be viewed as a natural extension of Azadkia \& Chatterjee's rank correlation \(\xi\) to a measure of predictability for a vector $\YY=(Y_1, \dots,Y_q)$ of response variables. 
Then the so-called data processing inequality as well as self-equitability of \(T\) are discussed and important invariance properties such as distribution invariance are derived.
It is further shown that \(\xi\) is invariant with respect to a dimension reduction principle preserving the key information about the extent of functional dependence of the response variables on the predictor vector---a principle that ensures a fast estimation for \(\xi\) and \(T\).

\subsection{Fundamental Properties of \(T\)} \label{Main.FundProp.}

As a first result, the following theorem states that \(T\) in \eqref{defmdm} indeed extends \(\xi\) to a measure of predictability for multi-response data with the desired additional properties.

\begin{theorem}[\(T\) as a measure of predictability] \label{theT}\label{theT4}\label{theT5}~\\
The map \(T\) defined by \eqref{defmdm} 
\begin{enumerate}[(i)]
\item satisfies the axioms (A\ref{prop1}), (A\ref{prop2}) and (A\ref{prop3}) of a measure of predictability,
\item fulfills the information gain inequality (P\ref{prop.IGI}), and
\item characterizes conditional independence (P\ref{prop.CI}).
\end{enumerate}
\end{theorem}

Due to properties (P1) and (P2) in the above theorem, it immediately follows that $T$ fulfills the so-called
\emph{data processing inequality} which states that a transformation of 
the predictor variables cannot enhance the degree of predictability;
see \citep{cover2006} for a detailed discussion and \citep{chatterjee2022estimating} for a data processing inequality with respect to \(\xi\,.\)

\begin{corollary}[Data processing inequality]\label{cor.T.DPI}~\\
The map \(T\) defined by \eqref{defmdm} fulfills the data processing inequality, 
i.e.,
\(T(\YY,\ZZ) \leq T(\YY,\XX)\)
for all $\XX,\YY,\ZZ$ such that $\YY$ and $\ZZ$ are conditionally independent given $\XX$.
In particular, 
\begin{eqnarray} \label{Eq.DPI}
  T(\YY,{\bf h}(\XX)) \leq T(\YY,\XX)
\end{eqnarray}
for all $\XX,\YY$ and all measurable functions ${\bf h}$. 
\end{corollary}

The data processing inequality implies several interesting invariance properties for \(T\).
The first one is the so-called \emph{self-equitability} introduced in \citep{kinney2014} (see also \citep{ding2017}).
According to \citep{reshef2011} self-equitability states that
``the statistic should give similar scores to equally noisy relationships of different types''.
In an additive regression setting 
 \(\YY = {\bf f} (\XX) + \boldsymbol{\varepsilon}\) (where \(\boldsymbol{\varepsilon}\) is not necessarily independent of \(\XX\)),  
it means that
the measure of predictability \(T\) ``depends only on the strength of the noise \(\boldsymbol{\varepsilon}\) and not on the specific form of \({\bf f}\)'' \citep{strothmann2022}.

\begin{corollary}[Self-equitability]\label{cor.T.SE}~\\
The map \(T\) defined by \eqref{defmdm} is \emph{self-equitable}, 
i.e., \(T(\YY,{\bf h}(\XX)) = T(\YY,\XX)\)
for all $\XX$, $\YY$ and all measurable functions ${\bf h}$ such that $\YY$ and $\XX$ are conditionally independent given ${\bf h}(\XX)$.
\end{corollary}

\subsection{Invariance Properties for \(T\)} \label{Main.Invariance}

We make use of the data processing inequality to gain insights into important invariance properties of \(T(\YY,\XX)\) concerning the distributions of \(\XX\) and \(\YY\).
The next result shows that the value \(T(\YY,\XX)\) remains unchanged when transforming the predictor variables \(X_1, \dots, X_p\) or the response variables \(Y_1, \dots, Y_q\) by their individual distribution functions, 
i.e., the predictor and response variables can be replaced by the individual ranks.

\begin{proposition}[Distribution invariance]\label{Cor.T.DistTrans}~~\\ 
The map \(T\) defined by \eqref{defmdm} fulfills
\( T(\YY,\XX) 
   = T \, ({\bf F}_{\YY}(\YY),{\bf F}_{\XX}(\XX)) \)
where \({\bf F}_{\XX}=(F_{X_1},\ldots,F_{X_p})\) and \({\bf F}_{\YY}=(F_{Y_1},\ldots,F_{Y_q})\,.\)
\end{proposition}

An extension of Proposition \ref{Cor.T.DistTrans} to invariance of \(T\) under the multivariate distributional transform (i.e. the generalized Rosenblatt transform) is examined in Section \ref{App.Add.SecMainRes} in the Supplementary Material. 
Invariance of $T$ under permutations (and bijective transformations in general) within the conditioning vector $\XX$ is immediate from the definition of \(T\) in \eqref{defmdm} (cf. Theorem \ref{theT6}).
Sufficient conditions on the underlying dependence structure for the invariance of $T$ under permutations within the response vector $\YY$ are presented in detail in Section \ref{App.Add.SecMainRes} in the Supplementary Material.

\subsection{Dimension Reduction Principle for \(T\)}
\label{Subsect.DimReduct.Princ.}

As we study in the sequel, the measure \(\xi\) is invariant with respect to a dimension reduction principle that preserves the key information about the extent of functional dependence of the response variables on the predictor vector.  The construction of \(T\) is based on this principle, which allows a fast estimation, as we discuss in Section \ref{secconasy}.

For \(\XX\) and \(Y\) consider the integral transform
\begin{align}\label{deftrafpsi}
  \psi_{Y|\XX}(y,y')
  & := \int_\Omega \E\left[\1_{\{Y\leq y\}} \mid \XX\right] \E\left[\1_{\{Y\leq y'\}}\mid \XX\right]  \de P 
     = \int_{\R^p} F_{Y|\XX=\xx}(y)F_{Y|\XX=\xx}(y')\de P^{\XX}(\xx)
\end{align}
for \(y,y'\in \R\,,\) where \(F_{Y|\XX=\xx}\) denotes the conditional distribution function of \(Y\) given \(\XX=\xx\,.\) 
Then \(\psi_{Y|\XX}\) is the distribution function of a bivariate random vector \((Y,Y^*)\) with \((\XX,Y^*)\) and \((\XX,Y)\) having the same distribution such that \(Y\) and \(Y^*\) are conditionally independent given \(\XX\,.\)
Due to the following result, Azadkia \& Chatterjee's rank correlation coefficient \(\xi(Y,\XX)\) only depends on the diagonal of the function \(\psi_{Y|\XX}\) and on the range of the distribution function of \(Y\):

\begin{proposition}[Dimension reduction principle]~~~\label{lemrepT} \\
The integral transform \(\psi_{Y|\XX}\) defined by \eqref{deftrafpsi} is a bivariate distribution function. Further, the measure \(\xi\) defined by \eqref{Tuniv} fulfills
\begin{align}\label{lemrepT1}
  \xi(Y,\XX) 
  & = a\int_{\R} \lim_{t\uparrow y} \psi_{Y|\XX}(t,t)\de P^Y(y) - b
\end{align}
for positive constants \(a := (\int_{\R} \Var(\1_{\{Y\geq y\}})\de P^Y(y))^{-1}\) and \(b := a \int_{\R} \lim_{z\uparrow y} F_Y(z)^2 \de P^Y(y)\,,\) both depending only on the range of the distribution function \(F_Y\).
\end{proposition}
Note that if \(\XX\) and \(Y\) in Proposition \ref{lemrepT} have continuous distribution functions, then 
the representation of \(\xi(Y,\XX)\) in \eqref{lemrepT1} simplifies to
\begin{align}\label{lemrepT.Cop}
  \xi(Y,\XX) 
  & = 6 \int_{\R} \psi_{Y|\XX}(y,y)\de P^Y(y) - 2 
    = 6 \int_{[0,1]} \psi_{F_Y(Y)|{\bf F}_{\XX}(\XX)}(u,u)\de \lambda(u) - 2
\end{align}
where \({\bf F}_{\XX}=(F_{X_1},\ldots,F_{X_p})\) and \(\psi_{F_Y(Y)|{\bf F}_{\XX}(\XX)}\) is the bivariate copula associated with \((Y,Y^*)\); 
see \citep[Theorem 2]{siburg2013} for \(p=1\), compare \citep[Theorem 4]{sfx2022phi} for general \(p\,.\)
Thus, in the case of continuous marginal distribution functions and in accordance with Corollary \ref{Cor.T.DistTrans}, 
\(\xi\) and \(T\) are solely copula-based and hence margin-free.

\subsection{A Permutation Invariant Version}

Since Azadkia \& Chatterjee's rank correlation \(\xi\) is a measure of directed dependence, it is not symmetric, i.e., \(\xi(Y_2,Y_1)\ne \xi(Y_1,Y_2)\) in general. 
This implies that the multivariate extension \(T\) is in general not invariant under permutations of the response variables,
i.e., \(T((Y_1,Y_2),\XX)\ne T((Y_2,Y_1),\XX)\) in general.
However, permutation invariance w.r.t. the components of \(\YY\) can be achieved by defining the map
\begin{align}\label{defmdmav}
    \overline{T}(\YY,\XX) := \frac{1}{q!} \sum_{\boldsymbol{\sigma}\in S_q} T(\YY_{\boldsymbol{\sigma}},\XX)\,,
\end{align}
where \(S_q\) denotes the set of permutations of \(\{1,\ldots,q\}\) and where \(\YY_{\boldsymbol{\sigma}}:= (Y_{\sigma_1},\ldots,Y_{\sigma_q})\) for \(\boldsymbol{\sigma}=(\sigma_1,\ldots,\sigma_q)\in S_q\,.\)
As an immediate consequence of Theorem \ref{theT}  and Corollaries \ref{cor.T.DPI} \& \ref{cor.T.SE}, 
the permutation invariant version  \(\overline{T}\) defines a measure of predictability that inherits all the aforementioned properties from \(T\,.\)

\begin{corollary}[\(\overline{T}\) as a measure of predictability]~
\label{corppivv}\\
The map \(\overline{T}\) defined by \eqref{defmdmav} 
\begin{enumerate}[(i)]
\item \label{corppivv1} 
satisfies the axioms (A\ref{prop1}), (A\ref{prop2}) and (A\ref{prop3}) of a measure of predictability.
\item \label{corppivv2} 
fulfills the information gain inequality (P\ref{prop.IGI}).
\item \label{corppivv3} 
characterizes conditional independence (P\ref{prop.CI}).
\end{enumerate}
In addition, \(\overline{T}\) fulfills the data processing inequality, is self-equitable and distribution invariant.
\end{corollary}

\subsection{\(T\) for the Multivariate Normal Distribution}

If \((\XX,\YY)\) follows a multivariate normal distribution, then \(T(\YY,\XX)\) has a closed-form expression due to the following representation of Chatterjee's rank correlation.

\begin{proposition}[Closed-form expression for the multivariate normal distribution]\label{propChatformmGau}~~\\
    Assume that \((\XX,Y)\sim N({\bf 0},\Sigma)\) has covariance matrix \(\Sigma = \left(\begin{smallmatrix} 
  \Sigma_{11} & \Sigma_{12} \\
  \Sigma_{21} & \sigma_Y^2  \end{smallmatrix}\right)\) with \(\sigma_Y>0\,.\) Then
\begin{align}\label{eqpropChatformmGau}
  \xi(Y,\XX) = \frac 3 \pi \arcsin\left(\frac{1+\rho^2}{2}\right)-\frac 1 2 \,, 
  \qquad \text{with } \rho = \sqrt{\Sigma_{21} \Sigma_{11}^{-} \Sigma_{12} / \sigma_Y^2}\,,
\end{align}
where \(\Sigma_{11}^-\) denotes a generalized inverse of \(\Sigma_{11}\) such as the Moore–Penrose inverse.
\end{proposition}

As a consequence of the above result, \(T(\YY,\XX)\) depends on all pairwise correlations, i.e., both on the correlations between the components of \(\XX\) and \(\YY\) as well as on the correlations within the vector \(\XX\) and within \(\YY\,.\) Further, since \(T\) is invariant under scaling, \(T(\YY,\XX)\) does not depend on the diagonal elements of the covariance matrix---at least if the latter is positive definite.

The next result characterizes the extreme values for \(T\) in the multivariate normal model, noting that perfect dependence of \(\YY\) on \(\XX\) in this case corresponds to perfect linear dependence. 

\begin{proposition}[Characterization of extreme cases]\label{propExtcasemn}~\\
    Let \((\XX,\YY)\) be multivariate normal with covariance matrix \(\Sigma = (\sigma_{ij}) = \left(\begin{smallmatrix} 
  \Sigma_{11} & \Sigma_{12} \\
  \Sigma_{21} & \Sigma_{22}
  \end{smallmatrix}\right)\).  Then
    \begin{enumerate}[(i)]
        \item \label{propExtcasemn1} \(T(\YY,\XX) =0 \) if and only if \(\Sigma_{12}\) is the null matrix (i.e., \(\sigma_{ij}=0\) for all \((i,j)\in \{1,\ldots,p\}\times \{p+1,\ldots,p+q\}\)),
        \item\label{propExtcasemn2} \(T(\YY,\XX)=1\) if and only if \(\rank(\Sigma)=\rank(\Sigma_{11})\,.\)
    \end{enumerate}
\end{proposition}

\section{Estimation} \label{secconasy}

We propose strongly consistent estimators \(T_n\) and \(\overline{T}_n\) for $T$ and $\overline{T}$ defined by \eqref{defmdm} and  \eqref{defmdmav}. Both estimators are consistent with the underlying construction principle and rely on Azadkia \& Chatterjee's graph-based estimator $\xi_n$ for $\xi$ that are used as plug-ins in \eqref{defmdm} and  \eqref{defmdmav}. The properties of $\xi_n$ imply strong consistency and 
a computation time of $O(n \log n)$ for the estimators $T_n$ and $\overline{T}_n$, respectively. 
As the main result of this section, we show asymptotic normality for $T_n$. Further, we provide rates of convergence for \(T_n\) and \(\overline{T}_n\,.\)
In Example \ref{exlimcasmnd} in the Supplementary Material we give evidence that the proposed estimators perform well in the multivariate normal model where closed-form expressions for \(T\) are known due to Proposition \ref{propChatformmGau}. 
The last part of this section contains a comparison of \(T\) and its estimator with related multi-output dependence measures (Table \ref{tabpropkpct}).

\subsection{Consistency}

In the following, we consider a $(p+q)$-dimensional random vector $(\XX,\YY)$
with i.i.d. copies $(\XX_1,\YY_1)$, $\dots$, $(\XX_n,\YY_n)$.
Recall that \(\YY\) is assumed to have non-degenerate components.

As an estimator for $T$ we propose the statistic $T_n$
given by
\begin{align} \label{def.estimator}
  T_n
  & =  T_n(\YY,\XX)
    := 1 - \frac{q - \sum_{i=1}^{q} \xi_n(Y_i , (\XX,Y_{i-1},\dots,Y_{1}))}{q - \sum_{i=2}^{q} \xi_n(Y_i , (Y_{i-1},\dots,Y_{1}))}     
 \end{align}
with $\xi_n$ being the estimator proposed by \citet{chatterjee2021} and given by
\begin{eqnarray} \label{def.estimator.T}
  \xi_n(Y,\XX)
  & = & \frac{\sum_{k=1}^{n} (n \, \min\{R_k,R_{N(k)}\}-L_k^2)}{\sum_{k=1}^{n} L_k \,(n - L_k)},
\end{eqnarray}
where $R_k$ denotes the rank of $Y_k$ among $Y_1, \dots, Y_n$, i.e., the number of $\ell$ such that $Y_\ell \leq Y_k$, 
and $L_k$ denotes the number of $\ell$ such that $Y_\ell \geq Y_k$. Further, for each $k$, the number $N(k)$ denotes the index $\ell$ such that $\XX_\ell$ is the nearest neighbor of $\XX_k$ with respect to the Euclidean metric on $\R^p$.
Since there may exist several nearest neighbors of $\XX_k$, ties are broken uniformly at random.
From the definition of \(\xi_n\) in \eqref{def.estimator.T}, 
it is apparent that the estimator \(T_n\) is built on the dimension reduction principle explained in Subsection \ref{Subsect.DimReduct.Princ.} (see also \citep[Section 11]{chatterjee2021} where the authors prove that \(Y_k\) and \(Y_{N(k)}\) are asymptotically conditionally independent given \(\XX\)),
which is key to a fast estimation of \(\xi\) and hence \(T\).

As an estimator for the permutation invariant version \(\overline{T}\,,\) we propose the statistic \(\overline{T}_n\) given by
\begin{align}\label{estTnqq}
    \overline{T}_n
  =  \overline{T}_n(\YY,\XX)
    := \frac{1}{q!} \sum_{\sigma\in S_q} T_n(\YY_\sigma,\XX)
\end{align}
where \(T_n\) is defined by \eqref{def.estimator}.

\citet{chatterjee2021} proved that $\xi_n$ is a strongly consistent estimator for $\xi$.
As a direct consequence, we obtain strong consistency of \(T_n\) and \(\overline{T}_n\) as follows.

\begin{theorem}[Consistency]\label{theT.Consistency}~\\
It holds that $\lim_{n \to \infty} T_n = T$ almost surely and $\lim_{n \to \infty} \overline{T}_n = \overline{T}$ 
almost surely.
\end{theorem}

\begin{remark}\label{remconst}
\begin{enumerate}[(i)]
    \item From the properties of the estimator $\xi_n$, see \citep[Remark (1)]{chatterjee2021}, it follows that also $T_n$ and $\overline{T}_n$ can be computed in $O(n \log n)$ time, using that the denominator in \eqref{defmdm} and \eqref{defmdm2} takes values only in the interval \([1,q]\,.\) 
    \item The estimators \(T_n\) and \(\overline{T}_n\) are model-free and merely rank-based estimators for \(T\) and \(\overline{T}\) without any tuning parameters and
    being consistent in full generality, 
    compare \citep{chatterjee2021}.
    \item All properties also apply to measures that are constructed and estimated as in \eqref{estTnqq} and \eqref{defmdmav} by averaging over a subset of permutations. Simulations indicate that averaging over cyclic permutations such as \((1,2,\ldots,q),(2,\ldots,q,1),\ldots\) yields overall good results for the variable selection in Section \ref{RDE.Sect.FS}. 
\end{enumerate}
\end{remark}

\subsection{Asymptotic Normality}\label{secasynor}

If the underlying distribution function of \((\XX,Y)\) is continuous and if \(Y\) is not perfectly dependent on \(\XX\,,\) then the estimator \(\xi_n\) (which coincides with \(T_n\) for \(q=1\)) behaves asymptotically normal, see \citep{han2022limit}.
Using this result and a modification of the nearest neighbor-based normal approximation in \citep[Theorem 3.4]{chatterjee2008}, it can be shown that also the linear combinations \(\Lambda_n\) and \(\alpha_n\) defined by
\begin{align}\label{AN.Def:Lambda}
    \Lambda_n
    & := \sum_{i=1}^{q} \xi_n(Y_i \, , \, (\XX,Y_{i-1},\dots,Y_{1}))
    \quad \textrm{and}
    & \alpha_n 
    & := \sum_{i=2}^q \xi_n(Y_i,(Y_{i-1},\ldots,Y_1))\,, \quad n\in \N\,,
\end{align}
are asymptotically normal.
The following result shows asymptotic normality of 
\begin{align}\label{repT_n}
  T_n = 1 - \frac{q - \Lambda_n}{q-\alpha_n}
\end{align}
for \(q>1\) response variables under some mild regularity conditions on \(\Lambda_n\) and \(\alpha_n\,.\) The detailed proof of Theorem \ref{AN.Thm:AN} is postponed to Section \ref{App.Sect.AN}.

\begin{theorem}[Asymptotic normality, $q > 1$]~~\label{AN.Thm:AN}
Assume that \((\XX,\YY)\) has a continuous distribution function,
that \(\YY\) is not perfectly dependent on \(\XX\,,\) and that there exists some \(i \in \{2,\dots,q\}\) such that \(Y_i\) is not perfectly dependent on \(\{Y_1,\dots,Y_{i-1}\}\).
If, additionally, 
\begin{align} \label{AN.Thm:AN.Cond0}
\limsup_{n \to \infty} \Cor(\Lambda_n,\alpha_n) &< 1 \quad \text{and}\\
\label{AN.Thm:AN.Cond}
    \sup_{n \in \mathbb{N}} \, \E \left( \left| \frac{\Lambda_n - \E[\Lambda_n]}{\sqrt{\var(\Lambda_n)}} \right|^{2+\delta_1} \right) &< \infty\,,    
    \qquad 
    \sup_{n \in \mathbb{N}} \, \E \left[ \left| \frac{\alpha_n - \E[\alpha_n]}{\sqrt{\var(\alpha_n)}} \right|^{2+\delta_2} \right] < \infty
\end{align}
for some \(\delta_1, \delta_2 > 0\),
then
\begin{align*}
  \frac{T_n - \E[T_n]}{\sqrt{\var(T_n)}} \stackrel{d}{\longrightarrow} N(0,1)\,,
\end{align*}
where \(\xrightarrow{ d }\) denotes convergence in distribution.
\end{theorem}

\begin{remark}
\begin{enumerate}[(i)]
\item 
For the proof of Theorem \ref{AN.Thm:AN}, we show asymptotic normality of 
\begin{align}\label{asnormi1}
    & \sqrt n (T_n-\E[T_n]) = - \sqrt{n} \big(\tfrac{q-\Lambda_n}{q-\alpha_n} - \E[\tfrac{q-\Lambda_n}{q-\alpha_n}]\big) \\
    \label{asnormi2} & =  \sqrt n \big(\Lambda_n-\E[\Lambda_n]\big) \tfrac{1}{q-\alpha_n} 
    - \sqrt n \big(\tfrac 1 {q-\alpha_n} - \E[\tfrac 1 {q-\alpha_n}]\big)\E[q-\Lambda_n] -
    \sqrt n \Cov(\Lambda_n,\tfrac 1 {q-\alpha_n}) \\
    \label{asnormi3} & = \sqrt n \big(\Lambda_n  - \kappa\, \alpha_n - \E[\Lambda_n - \kappa \alpha_n] \big)/(q-\alpha) + o_P(1)\,,
\end{align}
for some constant \(\kappa >0\,,\)
where the last expression in \eqref{asnormi2} is shown to converge to \(0\,.\) Noting that, in general, the second term in \eqref{asnormi2} does not vanish for \(n\to \infty\,,\) the idea is to use a Taylor approximation of \(1/(q-\alpha_n)\) to obtain \eqref{asnormi3}. Then, since \(\Lambda_n-\kappa \alpha_n\) is a linear combination of statistics of the form \(\xi_n\,,\) we derive (under slight regularity conditions due to the Taylor approximation) asymptotic normality of \(T_n\) from \(\xi_n\) using a modification of the local limit theorem \citep[Theorem 3.4]{chatterjee2008} for nearest neighbour statistics to several subgraphs. A key element in our proof is a sequential conditioning argument to show that $\liminf_{n \to \infty} n \var(\alpha_n) > 0$ (see the proof of Lemma \ref{AN.Lemma:AsympVar}\eqref{Lemma:AsympVar2}(ii)). We are not aware of an extension of this reasoning to the estimator \(\overline{T}_n\) of our permutation invariant version \(\overline{T}\,.\)

\item For the Taylor approximation in \eqref{asnormi3}, we use the assumption that the moments of order \(2+\delta_2\) for \(\frac{\alpha_n - \E[\alpha_n]}{\sqrt{\var(\alpha_n)}}\) are uniformly bounded, see Proposition \ref{AN.Lemma:TaylorExp}.
The moment assumptions in \eqref{AN.Thm:AN.Cond} imply uniform integrability of \(\left(\frac{\Lambda_n - \E[\Lambda_n]}{\sqrt{\var(\Lambda_n)}}\right)^2\) and \(\left(\frac{\alpha_n - \E[\alpha_n]}{\sqrt{\var(\alpha_n)}}\right)^2\), 
which is used in the proofs of Lemmas \ref{AN.Lemma:TaylorExp:Cor} and \ref{AN.Lemma:LimitsKappaBeta2}.
The correlation condition \eqref{AN.Thm:AN.Cond0} ensures that the limiting variance of \(\sqrt{n} (T_n - \E[T_n])\) does not vanish,
i.e., $\liminf_{n \to \infty} n \, \var(T_n) > 0$, see Lemma \ref{AN.Lemma:AsympVar.KappaBeta}\eqref{AN.Lemma:AsympVar.KappaBeta.Eq1}.
Simulations indicate that all these regularity assumptions are generally fulfilled. However, it is not clear how to extend the asymptotic normality for \(\xi_n\) in \citep{han2022limit} to our estimator \(T_n\) for output dimension \(q>1\) without the conditions \eqref{AN.Thm:AN.Cond0} and \eqref{AN.Thm:AN.Cond}.

\item Apart from special cases, the limiting variance of \(\sqrt{n} (T_n - \E[T_n])\) is generally unknown and must be estimated, for example via the bootstrap procedure.
Although the bootstrap method generally fails for Chatterjee's rank correlation, as demonstrated by \citep{han2023boot}, it is shown by \citep{dette2023boot} that an $m$ out of $n$ bootstrap procedure leads to a consistent estimator of the limiting variance whenever asymptotic normality for \(\sqrt{n} (\xi_n - \E[\xi_n])\) is achieved.
Simulations indicate that the latter method can also be employed for estimating the asymptotic variance of \(\sqrt{n} (T_n - \E[T_n])\), noting that an extension of \citep[Theorem 1]{dette2023boot} to our setting for output dimension \(q>1\) is not straightforward. 
\item As a consequence of \eqref{Tcond-uncond}, the conditional version of Azadkia and Chatterjee's rank correlation can be estimated through the unconditional version by
\begin{align}\label{eqestcondchatt}
    \xi_n(Y,\XX|\ZZ) = 1 - \frac{1-\xi_n(Y,(\XX,\ZZ))}{1-\xi_n(Y,\ZZ)}\,,
\end{align}
where the nearest neighbor-based estimator for the left-hand side is given in \citep[Section 2]{chatterjee2021}.
Due to \eqref{repT_n}, the estimator \(T_n\) has a similar form so that asymptotic normality for \(\xi_n(Y,\XX|\ZZ)\) might be shown in the same way as for \(T_n\,.\) However, the crucial difference is that we use in the proof of Theorem \ref{AN.Thm:AN} for the inequality in \eqref{AN.Lemma:TaylorExp:Proof.1} that the denominator in \eqref{repT_n} is bounded by positive constants, see \eqref{boundalpha_n}. In contrast, the denominator of \eqref{eqestcondchatt} is only lower bounded by \(0\) and \(\xi_n(Y,\XX|\ZZ)\) may attain arbitrarily large negative values, see also Figure \ref{Fig.Sim} in the Supplementary Material.
\end{enumerate}
\end{remark}


\subsection{Rate of convergence}\label{secroc}

Under some sensitivity assumptions on the conditional distributions, 
the asymptotic bias of \(T_n\) can be controlled adopting to the ideas in \citep[Theorem 4.1]{chatterjee2021} and \citep[Proposition 1.1]{han2022limit}.
The first assumption will comprise a local Lipschitz condition for the conditional distributions. The second is an exponential boundedness condition for tail probabilities. We refer to \citep[Chapter 4 and Proposition 4.2]{chatterjee2021} for a detailed discussion of the assumptions.
\begin{assumption}\label{Assumption.1a}
    For all \(i\in \{1,\ldots,q\}\,,\) there exist real numbers \(\gamma_i,\gamma_i',C_i,C_i'\geq 0\) such that 
    for any 
    \(\xx,\xx'\in \R^{p}\,,\) any \(\yy,\yy'\in \R^{i-1}\,,\) and any \(t\in \R\,,\)
    \begin{align*}
        &|P(Y_i\geq t \mid \XX =\xx, (Y_{i-1},\ldots,Y_1)=\yy) - P(Y_i\geq t \mid \XX =\xx', (Y_{i-1},\ldots,Y_1)=\yy') | \\
        & \leq C_i\left( 1+ \lVert (\xx,\yy) \rVert^{\gamma_i} + \lVert (\xx',\yy') \rVert^{\gamma_i}\right) \lVert (\xx,\yy)-(\xx',\yy')\rVert \quad \text{and}\\
        &|P(Y_i\geq t \mid (Y_{i-1},\ldots,Y_1)=\yy) - P(Y_i\geq t \mid (Y_{i-1},\ldots,Y_1)=\yy') | \\
        & \leq C_i'\big( 1+ \lVert \yy \rVert^{\gamma_i'} + \lVert \yy' \rVert^{\gamma_i'}\big) \lVert \yy-\yy'\rVert\,.
    \end{align*}
\end{assumption} 
\begin{assumption}\label{Assumption.2a}
    For all \(i\in \{1,\ldots,q\}\,,\) there exist \(\delta_{i},D_{i}>0\) such that 
    for any \(t>0\,,\) 
    \(P(\lVert (\XX,Y_{i-1},\dots,Y_1) \rVert\geq t)\leq D_{i} e^{- \delta_{i} t}\,.\)
\end{assumption}
Recall that \(p,q\geq 1\) are the dimensions of \(\XX\) and \(\YY\,,\) respectively. 
\begin{proposition}[Asymptotic bias]\label{propasybias}~\\
Assume that \((\XX,\YY)\) has a continuous distribution function.
    Then, under Assumptions \ref{Assumption.1a} and \ref{Assumption.2a}, we have for \(\gamma:= \max_i \{\gamma_i,\gamma_i'\}\) and \(d:= p+q\) that
    \begin{align}
        \left| \E[T_n] - T \right| = O\Big(\frac{(\log n)^{d+\gamma + \1_{d=2}}}{n^{1/(d-1)}}\Big)\,.
    \end{align}
\end{proposition}
A similar results holds true for \(\overline{T}_n\,.\)
\begin{remark}\label{RemPower}
    It is known that \(\xi_n\) and thus \(T_n\) suffer from weak power and are inefficient when testing on independence. As studied by \citep{deb2020}, \citep{shi2021} and \citep{Lin-Han-2023}, the power of \(\xi_n\) can be boosted by considering \(k\)-nearest neighbors instead of one-nearest neighbors. Simulations indicate that these results carry over to our extension \(T_n\,.\) 
    However, we emphasize that \(T\) is constructed to measure the degree of functional dependence of \(\YY\) on \(\XX\) rather than to test on (conditional) independence. We discuss various properties of \(T_n\) in comparison with other existing measures in Subsection \ref{seccomparison}.
\end{remark}


\subsection{Comparison with KPC and Distance Correlation} \label{seccomparison}

Our proposed extension of Azadkia and Chatterjee's rank correlation to multi-output measures satisfies various properties that motivate to use it for a multivariate variable selection, see Section \ref{secappl}.
Similar and competing measures from the literature are the kernel partial correlation (KPC) \(\rho^2\) studied in \citep{deb2020b} and the distance correlation (dCor) \(\cD\) in \citep{Szekely-2007}. We refer to Table  \ref{tabpropkpct} for an overview of various properties.

\begin{table}[h!]
\centering
\scalebox{0.90}{
\begin{tabular}{|l||c|c|c|c|}
\hline
\textbf{Comparison of \(\rho^2\), \(T\), and \(\overline{T}\)} & \textbf{dCor} & \textbf{KPC} & \multicolumn{2}{c|}{\textbf{didec}} \\ \hline
   Property  &  \(\mathcal{D}\) &  \(\rho^2\)  & \(T\) & \(\overline{T}\) \\ \hline\hline
   Population version in \([0,1]\) & \cmark  & \cmark   & \cmark  &   \cmark \\
Characterization of independence & \cmark   & \cmark   & \cmark  &   \cmark \\ 
Characterization of perfect dependence & \xmark  & \cmark & \cmark  &   \cmark   \\ \hline
Information gain inequality      & \xmark    & \cmark            & \cmark           & \cmark           \\
Characterization of conditional independence & \cmark & \cmark     & \cmark     & \cmark     \\ \hline
Data processing inequality & \xmark & \cmark     & \cmark     & \cmark     \\ 
Self-equitability & \xmark & \cmark     & \cmark     & \cmark     \\ \hline
Closed-form expression for multivariate normal distributions  & \cmark (\(d=2\)) & \xmark  & \cmark & \cmark \\ \hline
Invariance of \(\XX\) under bijective transformations & \xmark   & \cmark & \cmark & \cmark \\ 
Invariance of \(\YY\) under translations & \cmark & \cmark & \cmark & \cmark \\
Invariance of \(\YY\) under orthogonal transformations & \cmark & \cmark & \xmark & \xmark \\ 
Invariance of \(\YY\) under permutations & \cmark  & \cmark & \xmark & \cmark \\ 
Invariance of \(\YY\) under strictly increasing transformations & \xmark  & \xmark & \cmark & \cmark \\ \hline
Estimator computable in \(O(n \log n)\) time & \xmark  & \cmark & \cmark & \cmark \\ 
Statistically efficient estimation & \cmark  & \xmark & \xmark & \xmark \\ 
Strongly consistent estimator & \cmark & \cmark & \cmark & \cmark \\ 
Asymptotically normal estimator & ?  & ? & \cmark & ? \\ \hline
\end{tabular}}
\vspace{5mm}
\caption{Overview of properties for the (conditional) distance correlation (dCor) in \citep{Szekely-2007,Wang-2015}, the kernel partial correlation (KPC) in \citep{deb2020b} and our measures \(T\) and \(\overline{T}\) of directed dependence (didec).}
\label{tabpropkpct}
\end{table}

The KPC is defined in terms of the maximum mean discrepancy between two probability distributions, i.e., through a distance metric for distributions depending on a kernel such as the Gaussian kernel, the Laplace kernel or a linear kernel. There is both freedom in the choice of the kernel and in various tuning parameters, which may be an advantage but may also cause problems as we discuss in Example \ref{exIGinequ}.
The KPC is a measure of predictability that satisfies the information gain inequality (P\ref{prop.IGI}) and is a able to characterize conditional independence (P\ref{prop.CI}), see \citep[Theorem 1]{deb2020b}.
Hence, it describes the degree of predictability of \(\YY\) through \(\XX\) and is well-suited for a multi-output variable selection.

The dCor and its conditional version in \citep{Wang-2015} are defined through a distance between characteristic functions. They are designed primarily for testing (conditional) independence between multivariate random vectors and admit a statistically efficient estimation. However, in general, the distance correlation does not describe the strength of dependence because, as a consequence of \citep[Theorem 3]{Szekely-2007}, it does neither satisfy an information gain inequality nor it is able to characterize perfect dependence of \(\YY\) on \(\XX\,.\) 

Our proposed measure \(T\) satisfies all the important properties for a model-free multi-output variable selection method, i.e., \(T\) is a measure of predictability that additionally satisfies (P\ref{prop.IGI}) and (P\ref{prop.CI}). 
Using the construction in \eqref{defmdm2}, every single-output measure of predictability with these desired properties (P\ref{prop.IGI}) and (P\ref{prop.CI}) can in principle be extended to similar multi-output measures. The significant advantage of \(T\) (and KPC) is that it can be computed in \(O(n \log n)\) time, which is crucial for a fast variable selection. As a drawback, it is known that \(\xi_n\) and thus \(T_n\) (and also the estimator of \(\rho^2\)) are statistically inefficient. However, as our practical results show, it is worth making this trade-off in favor of a fast computation time, see also the discussion in \citep[Section 5]{chatterjee2021} and Remark \ref{RemPower}.

In contrast to KPC, our measure \(T\) is a merely rank-based quantity (hence margin-free), does not depend on any tuning parameters, and it has a simple expression without any regularity assumptions. 
In the following example, we show that a wrong choice of kernels and tuning parameters for \(\rho^2\) can lead to severe problems in variable detection.

\begin{example}[Information gain]\label{exIGinequ}~~\\
For \(\alpha>0\,,\) consider independent and standard normal predictors $X_1,X_2,X_3$ and two responses $Y_1 = X_1 + X_2 + N(0,1)$ and $Y_2 = X_1 + \alpha \, X_3 + N(0,1)$ with independent, standard normal errors.
Then both \(Y_1\) and \(Y_2\) depend on \(X_1\) but only \(Y_2\) depends on \(X_3\) with impact increasing in \(\alpha\,.\) 
We estimate the (normalized) information gain 
\begin{align}
\label{infgain}
\begin{split}
    & \frac{T(\YY , (X_1,X_2,X_3)) - T(\YY , (X_1,X_2))}{1 - T(\YY , (X_1,X_2))}\,,
  \\ 
  & \frac{\rho^2(\YY, (X_1,X_2,X_3)) - \rho^2(\YY, (X_1,X_2))}{1 - \rho^2(\YY, (X_1,X_2))} \,,
\end{split}
\end{align}
obtained by adding $X_3$. For a better comparability, we normalize the information gains with respect to the maximal possible information gain when including \(X_3\,.\) Both expressions in \eqref{infgain} are \(1\) if and only if \(\YY\) is a function of \((X_1,X_2,X_3)\), and they are \(0\) if and only if \(\YY\) and \(X_3\) are conditionally independent given \((X_1,X_2)\,,\) which both is clearly not the case in our setting.
Surprisingly, as can be seen in Figure \ref{Fig.Sim3}, $\rho^2$ with standard kernel \texttt{rbfdot(1)} has difficulties recognizing the information gain and decreases again towards zero as the influence of $X_3$ increases from $\alpha = 7$. This is not the case with our measure $T$. 
Interestingly for large values of $\alpha$, $\rho^2$ with kernels 
\texttt{rbfdot(1/(2*stats::median(stats::dist($Y$))$^2$))} and \texttt{vanilladot()} achieves values close to $1$, which appears to be a too high value in view of the only moderate influence of the variables $X_1$ and $X_2$ on $Y_1$. 
In both situations, the scale dependence of $\rho^2$ becomes visible.
\end{example}

\begin{figure}[t!]
  \centering
  \includegraphics[width=1\textwidth]{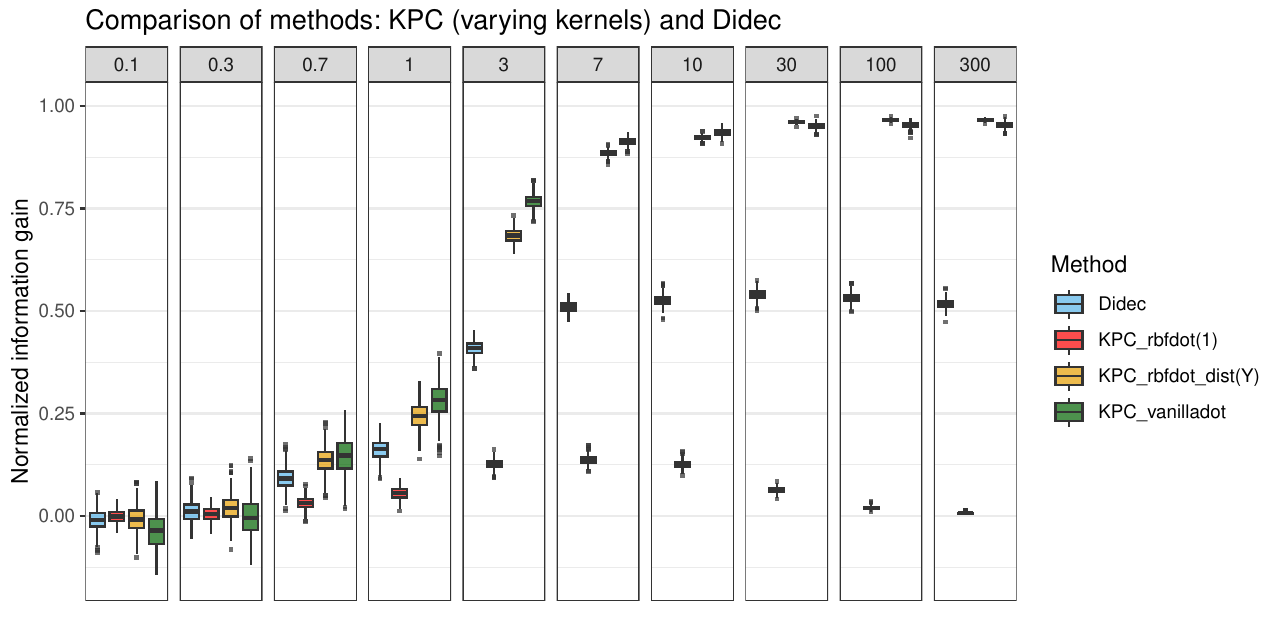}
  \caption{Boxplots comparing the $500$ obtained normalized information gains in \eqref{infgain} for $\rho^2$ estimated via R function \texttt{KMAc} (R package KPC) with those for $T$ estimated via  R function \texttt{didec} (R package didec). 
  Sample size is $1,000$ and \(\alpha\) is varying over $\alpha \in \{0.1, 0.3, 0.7, 1, 3, 7, 10, 30, 100, 300\}$ from left to right.}
  \label{Fig.Sim3}
\end{figure}


\section{MFOCI: Multivariate Feature Ordering by Conditional Independence} \label{RDE.Sect.FS}\label{secappl}

In the literature, numerous variable selection methods for a single output variable \(Y\) (i.e., \(q=1\)) are studied. Model-dependent methods are mostly based on linear or additive models, see e.g. \citep{Candes-2007,Tibshirani-2004,Fan-2001,Friedman-1991,George-1993,Miller-1990,Ravikumar-2009,
Yuan-2006,Zou-2005} and \citep{Chandrashekar-2014,Dash-1997,
Hastie-2009,Venkatesh-2019} for an overview; model-free methods rely on random forests \citep{Breiman-2001,Breiman-2017}, mutual information \citep{Battiti-1994,Vergara-2014}, or measures of predictability \citep{chatterjee2021,deb2020b}. 
However, up to our knowledge, there is rather little literature on feature selection methods that are applicable to multivariate response variables (i.e., for \(q > 1\)). In the class of linear methods, the lasso allows an extension to multiple output data \citep{Tibshirani-1996,Tibshirani-2011}, while distance correlation variable selection in \citep{Borboudakis-2019} and the kernel feature ordering by conditional independence in \citep{deb2020b} are more general or model-free methods. 

\subsection{Variable Selection Method MFOCI}

Since \(T\) and \(\overline{T}\) are measures of predictability that satisfy the information gain inequality, characterize conditional independence between multivariate random vectors (Theorem \ref{theT4} and Corollary \ref{corppivv}) and can be estimated in almost linear time (Theorem \ref{theT.Consistency} and Remark \ref{remconst}), they are suitable for a new subset selection method for predicting multivariate responses.
We adapt to \citep[Chapter 5]{chatterjee2021} and extend their feature selection method called FOCI (feature ordering by conditional independence) from \(q=1\) to arbitrary output dimension \(q\in \N\) which we denote MFOCI (multivariate FOCI) and works as follows: 
For the vector \(\YY = (Y_1,\dots,Y_q)\) of \(q \geq 1\) response variables and the vector \(\XX = (X_1,\dots,X_p)\) of \(p \geq 1\) predictor variables, consider i.i.d. copies \((\XX_1,\YY_1), \dots, (\XX_n,\YY_n)\) of $(\XX,\YY)$.
First, denote by \(j_1\) the index \(j\) maximizing \(T_n(\YY,X_j)\).
Now, assume that after \(k\) steps MFOCI has chosen the variables \(X_{j_1}, \dots, X_{j_k}\) 
and denote by \(j_{k+1}\) the index \(j \in \{1,\dots,p\}\backslash \{j_1,\dots,j_k\}\) maximizing \(T_n(\YY,(X_{j_1}, \dots, X_{j_k}, X_{j}))\).
The algorithm stops with the first index \(k\) for which 
\(T_n(\YY,(X_{j_1}, \dots, X_{j_k}, X_{j_{k+1}})) \leq T_n(\YY,(X_{j_1}, \dots, X_{j_k}))\), i.e., with the first index for which the degree of predictability no longer increases when adding further predictor variables.
Finally, denote by \(\hat{S} := \{j_1, \dots,j_k\}\) the subset selected by the above algorithm.
If the stopping criterion is not fulfilled for any \(k\) we set 
\(\hat{S} := \{1, \dots,p\}\).
If \(T_n(\YY,X_{j_1}) \leq 0\) the set \(\hat{S}\) is chosen to be empty.

A subset \(S\subseteq \{1,\ldots,p\}\) is called \emph{sufficient} if \(\YY\) and \(\XX_{S^c}:=(X_j)_{j\in S^c}\) are conditionally independent given \(\XX_S:=(X_j)_{j\in S}\), where \(S^c:=\{1,\ldots,p\}\setminus S\,.\)
By adapting to \citep[Chapter 6]{chatterjee2021}, 
we now discuss consistency of MFOCI for the case when, for every \(i \in \{1,\dots,q\}\), 
there exist \(\varepsilon_i > 0\) such that for any insufficient subset \(S\) there is some \(j \notin S\) with 
\begin{align}\label{Assumption.0}
    \lefteqn{\int_{\mathbb{R}} {\rm var} (P(Y_i \geq y \, | \, (\XX_{S \cup \{j\}},Y_{i-1},\dots,Y_1))) \; \mathrm{d} P^{Y_i}(y)}
    \\
    & \geq & \int_{\mathbb{R}} {\rm var} (P(Y_i \geq y \, | \, (\XX_{S},Y_{i-1},\dots,Y_1))) \; \mathrm{d} P^{Y_i}(y) + \varepsilon_i\,. \notag
\end{align}

We make use of the following weak regularity assumptions similar to Assumptions \ref{Assumption.1a} and \ref{Assumption.2a}.
\begin{assumption}\label{Assumption.1}
    For all \(i\in \{1,\ldots,q\}\,,\) there exist real numbers \(\gamma_i,C_i\geq 0\) such that for any \(S\subseteq \{1,\ldots,p\}\) with \(|S|\leq 1/\varepsilon_i+2\,,\) any \(\xx,\xx'\in \R^{|S|}\,,\) any \(\yy\in \R^{i-1}\,,\) and any \(t\in \R\,,\)
    \begin{align*}
        &|P(Y_i\geq t \mid \XX_S =\xx, (Y_{i-1},\ldots,Y_1)=\yy) - P(Y_i\geq t \mid \XX_S =\xx', (Y_{i-1},\ldots,Y_1)=\yy) | \\
        & \leq C^i\left( 1+ \lVert (\xx,\yy) \rVert^{\gamma_i} + \lVert (\xx',\yy) \rVert^{\gamma_i}\right) \lVert \xx-\xx'\rVert\,.
    \end{align*}
\end{assumption}

\begin{assumption}\label{Assumption.2}
    For all \(i\in \{1,\ldots,q\}\,,\) there exist \(\delta_i,D_i > 0\) such that for any subset \(S\subseteq\{1,\ldots,p\}\) with \(|S|\leq 1/\varepsilon_i+2\) and for any \(t>0\,,\) 
    \(P(\lVert (\XX_S,Y_{i-1},\dots,Y_1) \rVert\geq t)\leq D_i e^{- \delta_i t}\,.\)
\end{assumption}

As an immediate consequence of \citep[Theorem 6.1]{chatterjee2021}
and the inequality 
\begin{align*}
    P\left(\cap_{i=1}^q E_i\right) 
    & \geq \max \left\{ \sum_{i=1}^{q} P(E_i) - (q-1), 0 \right\} 
\end{align*}
for events \(E_1, \dots, E_q\),
the following result states that the subset of selected predictor variables via MFOCI is sufficient with high probability.

\begin{proposition}[Consistency] \label{FS.Consistency}~~\\
    Suppose that \(\varepsilon_1,\ldots,\varepsilon_q>0\) and that Assumptions \ref{Assumption.1} and \ref{Assumption.2} are fulfilled. Let \(\hat{S}\) be the subset selected by MFOCI with a sample size of \(n\,.\) Then, there exist \(L_1,L_2,L_3>0\) depending only on \(\gamma_i,C_i,\delta_i,D_i\) and \(\varepsilon_i\,,\) \(i\in \{1,\ldots,q\},\) such that 
    \begin{align}
    P(\hat{S} \text{ is sufficient})\geq 1 -  L_1 (p+q)^{L_2} e^{-L_3 n}\,.
    \end{align}
\end{proposition}

Appendix \ref{RDE.Subsect.FS} in the Supplementary Material verifies that MFOCI is plausible in the sense that it chooses a small number of variables that include, in particular, the most important variables for the individual feature selections.

In the sequel, we use simulated and real-world data to demonstrate the relevance of \(T\) and \(\overline{T}\) in forward feature selection extracting the most important variables for predicting a vector \(\YY\) of response variables.


\subsection{Comparison with Related Variable Selection Methods} \label{RDE.Subsect.FSII}

\begin{table}[t]
\small
\centering
\caption{Multivariate linear models (LM), generalized additive models (GAM), and non-linear models (nLM) used for the comparative simulation study in Subsection \ref{RDE.Subsect.FSII}.}
\label{Fig.Models}
\scalebox{1.00}{
\begin{tabular}{l||l|l}
LM1 
& \(Y_1 = 3X_1+2X_2-X_3 + \varepsilon_1\) 
& \(Y_2 = - \tfrac 1 3 X_1 - \tfrac 1 2 X_2 + X_3 + \varepsilon_2\)
\\
LM2 
& \(Y_1 = 3X_1 -4X_3 +\varepsilon_1\) 
& \(Y_2 = -X_1 + \tfrac 3 4 X_2 +\varepsilon_2\)
\\
LM3 
& \(Y_1 = 3X_1 +2X_2 +\varepsilon_1\) 
& \(Y_2 = X_3 +\varepsilon_2\)
\\
GAM1
& \(Y_1=\sin(X_1) + \cos(X_2) + e^{X_3} +\varepsilon_1\) 
& \(Y_2=X_1X_2+\sin(X_1 X_3)+\varepsilon_2\)
\\
GAM2 
& \(Y_1=\sin(X_1) + 2\cos(X_2) + e^{X_3} +\varepsilon_1\) 
& \(Y_2=X_1+2\sin(X_1 X_3)+\varepsilon_2\)
\\
GAM3
& \(Y_1=\sin(X_1) + 1.5\cos(X_2) + e^{X_3} +\varepsilon_1\) 
& \(Y_2=X_1+2\sin(X_1 X_3)+\varepsilon_2\)
\\
nLM1
& \(Y_1 = \frac{2 \log(X_1^2+X_2^4)}{\cos(X_1)+\sin(X_3)} + t_1\) 
& \(Y_2= |X_1+V|^{\sin(X_2-X_3)}\)
\\
nLM2
& \(Y_1 = \frac{2 \log(X_1^2+X_1^4)}{\cos(X_1)+\sin(X_3)} + t_1\) 
& \(Y_2= |X_1+V|^{\sin(X_1-X_2)}\)
\end{tabular}}

\bigskip
\scalebox{1}{
\begin{tabular}{l||l}
LM4
&  \(Y_i= \sum_{k=1}^{11-i} X_k + \varepsilon_i\) for \(i\in \{1,\ldots,10\}\)
\\
nLM3
& \(Y_1 = X_1 + \varepsilon_1\,,\) ~ \(Y_i= X_i/(1+\sin(Y_{i-1})) + \varepsilon_i\) for \(i\in \{2,\ldots,10\}\)
\end{tabular}}
\end{table}
\begin{table}[t!]
\small
\centering
\caption{Performance of the variable selection algorithms for a sample of size \(n=200\,,\) for \(q=2\) output variables (i.e., \(Y_1,Y_2\)) depending in the respective model on at most \(3\) out of \(p=10\) predictor variables. The reported numbers are: Proportion of times \(\{X_1,X_2,X_3\}\) is selected possibly with other variables // proportion of times exactly \(\{X_1,X_2,X_3\}\) is selected // average number of variables selected. 
}
\label{Fig.SimFS1}
\scalebox{0.85}{
\begin{tabular}{c||c|c|c|c|c}
\makecell[l]{\(q=2\) \\ \(p=10\)}
& \makecell[c]{MFOCI}
& \makecell[c]{KFOCI \\ default kernel}
& \makecell[c]{KFOCI \\ kernel  \texttt{rbfdot(1)} }
& \makecell[c]{dcorVS}
& \makecell[c]{Lasso}
\\
\hline
\hline
LM1
& 1.00 / 0.52 // 3.52
& 1.00 / 0.99 // 3.01
& 1.00 / 1.00 // 3.00
& 1.00 / 0.85 // 3.15
& 1.00 / 1.00 // 3.00
\\
LM2
& { 0.89 / 0.49 // 3.32}
& 0.08 / 0.08 // 2.08
& 0.00 / 0.00 // 2.00
& 0.95 / 0.84 // 3.08
& 1.00 / 1.00 // 3.00
\\
LM3
& 1.00 / 0.48 // 3.60
& 0.93 / 0.93 // 2.93
& 0.76 / 0.76 // 2.76
& 1.00 / 0.84 // 3.18
& 1.00 / 1.00 // 3.00
\\
\hline
GAM1
& 0.91 / 0.27 // 3.82
& 0.94 / 0.74 // 3.14
& 0.99 / 0.99 // 2.99
& 0.53 / 0.40 // 2.79
& 0.00 / 0.00 // 1.88
\\ 
GAM2
& { 0.92 / 0.39 // 3.63}
& 0.86 / 0.82 // 2.91
& 0.90 / 0.90 // 2.90
& 0.94 / 0.72 // 3.17
& 0.01 / 0.01 // 2.03
\\
GAM3
&  0.78 / 0.36 // 3.40
& 0.47 / 0.43 // 2.56
& 0.38 / 0.38 // 2.38
& 0.76 / 0.56 // 3.03
& 0.01 / 0.01 // 2.03
\\ 
\hline
nLM1
& 0.87 / 0.57 // 3.14
&  0.90 / 0.80 // 3.02
& 0.97 / 0.91 // 3.03
& 0.00 / 0.00 // 1.54
& 0.92 / 0.00 // 9.73
\\ 
nLM2
&  0.91 / 0.77 // 3.09 
& 0.04 / 0.04 // 1.99
& 0.00 / 0.00 // 2.00
& 0.00 / 0.00 // 1.76
& 0.91 / 0.00 // 9.70
\\
\end{tabular}}
\end{table}

\begin{table}[t!]
\small
\centering
\caption{Performance of the variable selection algorithms for a sample of size \(n=200\,,\) for \(q=10\) output variables (i.e., \(Y_1,\ldots,Y_{10}\)) depending in the respective model on at most \(10\) out of \(p=25\) predictor variables. The reported numbers are: Proportion of times \(3\) / \(5\) / \(8\) / all \(10\) variables are correctly selected possibly with other variables // average number of variables selected. 
}
\label{Fig.SimFS2}
\scalebox{0.85}{
\begin{tabular}{c||c|c|c|c|c}
\makecell[l]{\(q=10\) \\ \(p=25\)}
& \makecell[c]{MFOCI}
& \makecell[c]{KFOCI \\  default kernel}
& \makecell[c]{KFOCI \\ kernel  \texttt{rbfdot(1)} }
& \makecell[c]{dcorVS}
& \makecell[c]{Lasso}
\\
\hline
\hline
LM4
& \makecell[c]{ 0.85  / 0.65 / 0.17 / \\  0.00 // 6.91}
& \makecell[c]{0.46 / 0.24 / 0.01 / \\ 0.00 // 4.53 }
& \makecell[c]{1.00 / 1.00 / 0.83 / \\ 0.04 // 8.28 }
& \makecell[c]{1.00 / 1.00 / 1.00 / \\ 0.83 // 10.33 }
& \makecell[c]{1.00 / 1.00 / 1.00 / \\ 1.00 // 10.00 }
\\[10pt]
nLM3
& \makecell[c]{{ 0.99 / 0.93 / 0.52 /} \\ { 0.08 // 9.10 }}
& \makecell[c]{0.54 / 0.14 / 0.01 / \\ 0.00 // 5.53 }
& \makecell[c]{0.48 / 0.29 / 0.01 / \\ 0.00 // 5.26 }
& \makecell[c]{0.01 / 0.00 / 0.00 / \\ 0.00 // 1.46 }
& \makecell[c]{1.00 / 1.00 / 1.00 / \\ 1.00 // 25.00 }
\\
\end{tabular}}
\end{table}

We compare MFOCI with the feature selection algorithm \emph{kernel feature ordering by conditional independence} (KFOCI) based on KPC (see Subsection \ref{seccomparison}) in \citep{deb2020b},
with the distance correlation variable selection method dcorVS in \citep{Borboudakis-2019},
and with the Lasso (see \citep{Tibshirani-1996,Tibshirani-2011}) which are all applicable to multi-output data. 

\bigskip\noindent
\textit{Simulation Study:}
As illustrative examples, we consider the multivariate linear models (LM), generalized additive models (GAM), and non-linear models (nLM) summarized in Table \ref{Fig.Models} for independent random variables \(X_1,\ldots,X_p\sim N(0,1)\,,\) \(U_1,\ldots,U_p\sim U(0,1)\,,\) \(\varepsilon_1,\ldots,\varepsilon_q\sim N(0,1)\,,\) \(t_1\) Student-t distributed with \(\nu=1\) degree of freedom, and \(V
\sim U(0,1)\,.\) From these random variables, we generate \(n=200\) samples and determine \(Y_1,\ldots,Y_q\) depending only on a few of the predictor variables \(X_1,\ldots,X_p\) and \(U_1,\ldots,U_p\,,\) respectively, with a normally, Student-t or uniformly distributed error term.

The performances of the feature selection algorithms are presented in Table \ref{Fig.SimFS1} for \(p=10\) predictor variables and \(q=2\) response variables and in Table \ref{Fig.SimFS2} for \(p=25\) predictor variables and \(q=10\) response variables. In both cases, \((Y_1,\ldots,Y_q)\) depend only on a small number of input variables. The reported numbers show the proportion of times where the algorithm selects the correct variables or exactly the correct variables. Further, the last number in each cell is the average of variables selected by the respective algorithm when it stops. For determining the proportions and averages the algorithms run \(100\) times.\\
As to be expected, Lasso works very well in linear models but is not useful elsewhere. The variable selection method dcorVS (version \texttt{dcor.fbed}) based on the distance correlation is also useful in generalized additive models but fails to work in the non-linear settings.
In contrast, MFOCI and KFOCI generally perform quite well in all models considered where we applied for KFOCI both the default kernel and the kernel \texttt{rbfdot(1)} implemented in  \texttt{R}. However, in cases where not all output variables depend on the same predictor variables (see models LM3, GAM3 and, in particular, LM2, nLM2, LM4, nLM4) MFOCI clearly outperforms KFOCI, see the discussion in Example \ref{exIGinequ}.
For \(q=2\,,\) we used for MFOCI the estimator \(\overline{T}_n\) in \eqref{estTnqq} for the permutation invariant version \(\overline{T}\) defined in \eqref{defmdmav}. For \(q=10\,,\) we used a version where we averaged over all \(10\) decreasing resp. increasing permutations, instead of averaging over all \(10!\) permutations; cf. Remark \ref{remconst}.

\bigskip \noindent
\textit{Real-World Data Example:} 
We use real-world data to compare our forward feature selection method MFOCI 
with KFOCI \citep{deb2020b}, dcorVS \citep{Borboudakis-2019} and the Lasso \citep{Tibshirani-1996,Tibshirani-2011};
see also Section \ref{RDE.Subsect.FSIII.App} in the Supplementary Material.

\begin{table}[t!]
\small
\centering
\caption{The relevant variables to predict (AMT, AP) selected via MFOCI,
KFOCI, dcorVS (including variable ranking in brackets) and Lasso with MSPEs for each response variable; see Subsection \ref{RDE.Subsect.FSII} for details.
The variables selected via MFOCI to predict (AMT,AP) are marked in red color.}
\label{Fig.SeoulWeatherII}
\scalebox{0.95}{
\begin{tabular}{c||l|l|l|l|l}
\makecell[l]{Variables to predict \\ (AMT, AP)}
& \makecell[l]{via MFOCI}
& \makecell[l]{via KFOCI, \\ default kernel}
& \makecell[l]{via KFOCI, \\ kernel \texttt{rbfdot(1)}}
& \makecell[l]{via dcorVS}
& \makecell[l]{via Lasso}
\\
\hline
\hline
& (1) \textcolor{foxred}{MTWaQ}
& (4) \textcolor{foxred}{MTWaQ}
& (1) \textcolor{foxred}{MTWaQ}
& (3) \textcolor{foxred}{MTWaQ}
& \textcolor{foxred}{MTWaQ}
\\
& (3) \textcolor{foxred}{MTCQ}
& (3) \textcolor{foxred}{MTCQ}
& 
& 
& \textcolor{foxred}{MTCQ}
\\
&
& 
& 
& 
& MTWeQ
\\
& (2) \textcolor{foxred}{PWeQ}
& (1) \textcolor{foxred}{PWeQ}
& (2) \textcolor{foxred}{PWeQ}
& (1) \textcolor{foxred}{PWeQ}
& \textcolor{foxred}{PWeQ}
\\
& (4) \textcolor{foxred}{PDQ}
& (2) \textcolor{foxred}{PDQ}
& (5) \textcolor{foxred}{PDQ}
& (2) \textcolor{foxred}{PDQ}
& \textcolor{foxred}{PDQ}
\\
& 
& 
& (4) PWaQ
& (5) PWaQ
& PWaQ
\\ 
& 
& 
& (3) PCQ
& (4) PCQ
& PCQ
\\
\hline
\hline
\makecell[c]{MSPE for AMT}
& \multicolumn{2}{c|}{\makecell[c]{151}}
& \multicolumn{2}{c|}{\makecell[c]{616}}
& \makecell[c]{178}
\\
\hline
\makecell[c]{MSPE for AP}
& \multicolumn{2}{c|}{\makecell[c]{13265}}
& \multicolumn{2}{c|}{\makecell[c]{13725}}
& \makecell[c]{12738}
\end{tabular}}
\end{table}

We consider the data set of bioclimatic variables for $n=1862$ locations homogeneously distributed over the global landmass from CHELSEA \citep{karger2017, karger2018} 
and analyze the influence of a set of thermal and precipitation-related variables (see Table \ref{Fig.Climate1} in the Supplementary Material) on the pair \emph{Annual Mean Temperature} (AMT) and \emph{Annual Precipitation} (AP). 
Due to Table \ref{Fig.SeoulWeatherII},
the procedure through Lasso ends with 7 predictor variables,
the procedures via KFOCI (where we apply the kernel \texttt{rbfdot(1)}) and via dcorVS (where we apply the methods \texttt{dcor.fed}) end with 5 predictor variables,
and both MFOCI via \(T\) and KFOCI applying the default kernel end with the same 4 predictor variables (even though the order of selected variables differs).
Note that KFOCI is used here with the default number of nearest neighbors.
For each subset of selected variables, the (cross-validated) mean squared prediction errors (MSPE) based on a random forest are calculated using the \texttt{R}-package \emph{MultivariateRandomForest}.\footnote{When compared to the corresponding MSPE (for response AMT) of 151 for the variables selected via MFOCI and KFOCI (default kernel) in Table \ref{Fig.SeoulWeatherII}, an MSPE of 178 for the variables selected by Lasso seems a little too high. We suspect that this is due to a sensitivity of the \texttt{R}-package \emph{MultivariateRandomForest} to the order of the predictor variables.}

To conclude, MFOCI achieves similar prediction errors as the other methods, 
but with a considerably smaller number of selected variables than Lasso (and BVCQR in Section \ref{RDE.Subsect.FSIII.App}) resulting in a significant reduction in complexity.
MFOCI and KFOCI perform comparably well; however, the variable selection via KFOCI exhibits a sensitivity to the kernel used.



\section{Proofs from Section \ref{secmainres}} \label{App.Proofs}

Recall that we refer to \(\YY\) as the vector of response variables which is always assumed to have non-degenerate components, i.e., for all \(i\in \{1,\ldots,q\}\,,\) the distribution of \(Y_i\) does not follow a one-point distribution.


\begin{proof}[Proof of Theorem \ref{theT}.]
(A\ref{prop1}): For the measure $\xi$ defined by \eqref{Tuniv}, we first observe that
\begin{align*}
  0
  \leq \xi(Y_i,(Y_{i-1},\ldots,Y_1)) 
  \leq \xi(Y_i,(\XX,Y_{i-1},\ldots,Y_1))
  \leq 1 ~~~\text{for all } i\in \{1,\ldots,q\}\,,
\end{align*}
because \(\xi\) satisfies axioms (A\ref{prop1}) - (A\ref{prop3}) and the information gain inequality (P\ref{prop.IGI}); see \citep[Lemma 11.2]{chatterjee2021}. This implies \(T(\YY,\XX)\in [0,1]\). 
\\
(A\ref{prop2}): 
From \eqref{defmdm} we obtain that \(T(\YY,\XX)=0\) if and only if
\begin{align}\label{protheT1}
    \sum_{i=1}^q \underbrace{\left[\xi(Y_i,(\XX,Y_{i-1},\ldots,Y_1)) - \xi(Y_i,(Y_{i-1},\ldots,Y_1))\right]}_{\geq 0} =0 \,,
\end{align}
where each sumand is non-negative due to the information gain inequality (P\ref{prop.IGI}). 
Hence \eqref{protheT1} is equivalent to
\begin{align}\label{protheT2}
    \xi(Y_i,(\XX,Y_{i-1},\ldots,Y_1)) = \xi(Y_i,(Y_{i-1},\ldots,Y_1)) ~~~\text{for all } i \in \{1,\ldots,q\}\,.
\end{align}
Due to the characterization of conditional independence, see \citep[Lemma 11.2]{chatterjee2021}, \eqref{protheT2} is equivalent to 
\begin{align*}
\begin{cases}
&  Y_1 \text{ is independent of } \XX\,,\\
&  Y_2 \text{ is conditionally independent of } \XX \text{ given } Y_1\\
&~~~ \vdots\\
& Y_q \text{ is conditionally independent of } \XX \text{ given } (Y_1,\ldots,Y_{q-1})\,, 
\end{cases}
\end{align*}
which in turn is equivalent to \(\YY\) being independent of \(\XX\) (see, e.g., \citep[Proposition 6.8]{Kallenberg-2002}).
\\
(A\ref{prop3}): 
From \eqref{defmdm} we further obtain that \(T(\YY,\XX)=1\) if and only if
\begin{align}\label{protheT4}
    \sum_{i=1}^q \big[ \xi(Y_i,(\XX,Y_{i-1},\ldots,Y_1)) - \xi(Y_i,(Y_{i-1},\ldots,Y_1))\big]  
    = \sum_{i=1}^q \big[1 - \xi(Y_i,(Y_{i-1},\ldots,Y_1))\big] \,.
\end{align}
Since \(\xi(Y_i,(\XX,Y_{i-1},\ldots,Y_1))\in [0,1]\) for all \(i\in \{1,\ldots,q\}\,,\) \eqref{protheT4} is equivalent to
\begin{align*}
    \xi(Y_i,(\XX,Y_{i-1},\ldots,Y_1)) = 1 ~~~\text{for all } i\in \{1,\ldots,q\}\,,
\end{align*}
which exactly means by the characterization of perfect dependence that
\begin{align}\label{protheT7}
\begin{cases}
  Y_1 &= g_1(\XX)\,,\\
  Y_2 &= g_2(\XX,Y_1)=g_2(\XX,g_1(\XX))\,,\\
   & \vdots\\
  Y_q &= g_q(\XX,Y_{q-1},\ldots,Y_1) = g_q(\XX,g_{q-1}(\XX,g_{q-2}(\XX,\ldots),\ldots,g_1(\XX)),\ldots,g_1(\XX))
\end{cases}
\end{align}
for some measurable functions \(g_i\colon \R^{p+i-1}\to \R\,,\) \(i\in \{1,\ldots,q\}\,,\) 
which equivalently means that \(\YY\) is a function of \(\XX\,.\)
This proves the axioms.

We now verify properties (P\ref{prop.IGI}) and (P\ref{prop.CI}).
Since the denominator \(\sum_{i=1}^q \left[1-\xi(Y_i,(Y_{i-1},\ldots,Y_1)\right]\) is strictly positive, 
the information gain inequality (P\ref{prop.IGI}) for $T$ immediately follows from the respective property of $\xi$ (see \citep[Lemma 11.2]{chatterjee2021}).
\\
To prove that \(T\) characterizes conditional independence, recall that conditional independence of \(\YY\) and \(\ZZ\) given \(\XX\) is equivalent to
\begin{align}\label{proofCIg}
\begin{cases}
& Y_1 \text{ and } \ZZ \text{ are conditionally independent given } \XX\,,\\
& Y_2 \text{ and } \ZZ \text{ are conditionally independent given } (\XX,Y_1)\,,\\
& ~~~\vdots\\
&Y_q \text{ and } \ZZ \text{ are conditionally independent given } (\XX,Y_1,\ldots,Y_{q-1})\,,
\end{cases}
\end{align}
see, e.g., \citep[Proposition 6.8]{Kallenberg-2002}.
Due to the fact that $\xi$ characterizes conditional independence (see \citep[Lemma 11.2]{chatterjee2021}),
the system \eqref{proofCIg} is equivalent to
\begin{align*}
  \xi(Y_i,(\XX,Y_{i-1},\ldots,Y_1)) 
  & = \xi(Y_i,(\XX,\ZZ,Y_{i-1},\ldots,Y_1)) ~~~\text{for all } i\in \{1,\ldots,q\}\,.
\end{align*}
By virtue of the information gain inequality, this in turn is equivalent to \(T(\YY,\XX) = T(\YY,(\XX,\ZZ))\), which proves the assertion.
\end{proof}

\begin{proof}[Proof of Corollary \ref{cor.T.DPI}.]
Let \(\YY\) and \(\ZZ\) be conditionally independent given \(\XX\).
Then the information gain inequality and the characterization of conditional independence (see Theorem \ref{theT4}) yield
$$ 
  T(\YY,\ZZ) 
  \leq T(\YY,(\XX,\ZZ))
    =  T(\YY,\XX)
$$
which proves the data processing inequality.
The second inequality is immediate from the fact that \(\YY\) and \({\bf h}(\XX)\) are conditionally independent given \(\XX\).
\end{proof}

\begin{proof}[Proof of Corollary \ref{cor.T.SE}.]
Let \(\YY\) and \(\XX\) be conditionally independent given \({\bf h}(\XX)\).
Then the data processing inequality in Corollary \ref{cor.T.DPI} implies
$$ 
  T(\YY,{\bf h}(\XX)) 
  \leq T(\YY,\XX)
  \leq T(\YY,{\bf h}(\XX))   
$$
which proves self-equitability.
\end{proof}

\begin{proof}[Proof of Proposition \ref{Cor.T.DistTrans}.]
For \(k\in \{1,\ldots,p\}\,,\) it is straightforward to verify that \(\xi(X_k,F_{X_k}(X_k)) = 1\,.\) Hence, there exists some measurable function \(f_k\) such that 
\(X_k = f_k(F_{X_k}(X_k))\) almost surely.
From the data processing inequality in Corollary \ref{cor.T.DPI} we then conclude that 
\begin{align*}
  T \big(\YY, (F_{X_1}(X_1), \dots, F_{X_p}(X_p))\big)
  & \leq T \big(\YY,\XX\big)
  \\
  &   =  T \big(\YY, (f_1(F_{X_1}(X_1)), \dots, f_p(F_{X_p}(X_p)))\big)
  \\
  & \leq T \big(\YY, (F_{X_1}(X_1), \dots, F_{X_p}(X_p))\big)\,.
\end{align*}
This proves invariance of \(T\) with respect to the predictor variables.

For the second part, recall that every random variable \(Y\) fulfills \(Y \eqd F_{Y}^{-1} \circ F_{Y} \circ Y\), see \citep[Lemma A.1(ii)]{Ansari-2021}.
From the definition of \(\xi\) we conclude that
\begin{align*}
  \xi(F_Y(Y),\XX)
  & = \frac{\int_{[0,1]} {\rm var} (P(F_Y(Y) \geq u \, | \, \XX)) \; \mathrm{d} P^{F_Y \circ Y}(u)}
					{\int_{[0,1]} {\rm var} (\mathds{1}_{\{F_Y(Y) \geq u\}}) \; \mathrm{d} P^{F_Y \circ Y}(u)}
  \\
  & = \frac{\int_{\mathbb{R}} {\rm var} (P(Y \geq y \, | \, \XX)) \; \mathrm{d} P^{F_{Y}^{-1} \circ F_{Y} \circ Y}(y)}
					{\int_{\mathbb{R}} {\rm var} (\mathds{1}_{\{Y \geq y\}}) \; \mathrm{d} P^{F_{Y}^{-1} \circ F_{Y} \circ Y}(y)}
  \\
  & = \frac{\int_{\mathbb{R}} {\rm var} (P(Y \geq y \, | \, \XX)) \; \mathrm{d} P^{Y}(y)}
					{\int_{\mathbb{R}} {\rm var} (\mathds{1}_{\{Y \geq y\}}) \; \mathrm{d} P^{Y}(y)}
    = \xi(Y,\XX)\,.
\end{align*}
Hence, the invariance of \(T\) with respect to the vector of response variables follows from its definition \eqref{defmdm} using also the invariance of \(T\) with respect to the predictor variables. 
\end{proof}

\begin{proof}[Proof of Proposition \ref{lemrepT}.]
Since \(Y\) is non-degenerate, the constants \(a\) and \(b\) are positive. It follows that 
\begin{align*}
\int_{\R} \Var(P(Y\geq y\mid \XX))\de P^Y(y) 
&= \int_{\R} \left[\E\left( P(Y\geq y\mid \XX)^2\right) - P(Y\geq y)^2 \right] \de P^Y(y)\\
&= \int_{\R} \left[\E \left( (1-\lim_{z\uparrow y}F_{Y|\XX}(z))^2 \right) -(1-\lim_{z\uparrow y}F_Y(z))^2\right] \de P^Y(y)\\
&= \int_{\R} \left[\E\left(\lim_{z\uparrow y}F_{Y|\XX}(z)^2\right)-\lim_{z\uparrow y} F_Y(z)^2 \right]\de P^Y(y)\\
&= \int_{\R}\lim_{z\uparrow y}\psi_{Y|\XX}(z,z) \de P^Y(y) - \frac b a\,,
\end{align*}
where the last two identities follow from the monotone convergence theorem and from the definition of \(\psi_{Y|\XX}\) in \eqref{deftrafpsi}. Hence, we obtain
\begin{align*}
  \xi(Y,\XX) 
  = \frac{\int_{\R}\Var(P(Y\geq y \mid \XX)) \de P^Y(y)}{\int_{\R}\Var(\1_{\{Y\geq y\}})\de P^Y(y)} 
  = a \int_{\R} \lim_{z\uparrow y}\psi_{Y|\XX}(z,z) \de P^Y(y) - b\,,
\end{align*}
which proves the assertion.
\end{proof}

\begin{proof}[Proof of Proposition \ref{propChatformmGau}.]
First assume that \(\sigma_Y = 1\,.\)
If \(\Sigma_{21} = \Sigma_{12}^T\) is the null matrix, then \(\rho = 0\) and the formula in \eqref{eqpropChatformmGau} yields \(\xi(Y,\XX) = 0\,,\) which is the correct value since \(\Sigma_{12}\) being the null matrix characterizes independence in the multivariate normal model due to Proposition \ref{propExtcasemn}\eqref{propExtcasemn1}.\\
If \(\Sigma_{21}\) is not the null matrix (and thus also \(\Sigma_{11}\) is not the null matrix),
consider the linear transformation \(S:= A \XX\) for \(A:=\frac{\Sigma_{21}\Sigma_{11}^{-}}{\Sigma_{21}\Sigma_{11}^{-}\Sigma_{12}}\,,\) noting that the denominator is positive. Then \((S,Y)\) is bivariate normal with zero mean and covariance matrix \(\Sigma = \left(\begin{smallmatrix} 
  A\Sigma_{11}A^T & A\Sigma_{12} \\
  \Sigma_{21}A^T & 1  \end{smallmatrix}\right).\) 
  Denoting by \(\cS\) the row space of \(\Sigma_{11}\), it follows for all \(\xx\in \cS\) (and thus for \(P^\XX\)-almost all \(\xx\in \R^p\)) and for \(s := A\xx\) that
    \begin{align}\label{eqpropChatformmGau1}
    \begin{split}
        (Y\mid S = s ) &\sim N\bigg(\frac{\Sigma_{21} A^T}{A\Sigma_{11} A^T}\, s,\, 1-\frac{\Sigma_{21} A^T A \Sigma_{12}}{A \Sigma_{11} A^T}\bigg)\\
        &= N(\Sigma_{21}\Sigma_{11}^{-} \xx,\, 1-\Sigma_{21} \Sigma_{11}^{-} \Sigma_{12}) \sim (Y\mid \XX = \xx)\,,
    \end{split}
    \end{align}
    cf. \citep[Corollary 5]{Cambanis-1981}.
    Hence, for all \(y\in \R\,,\) we have \(F_{Y|S=s}(y) = F_{Y|\XX=\xx}(y)\) for \(P^\XX\)-almost all \(\xx\in \R^p\,.\) This implies 
    \begin{align}\label{eqpropChatformmGau2}
        \Var(P(Y\geq y | S)) = \Var(P(Y\geq y|\XX)) \quad \text{for all } y\in \R\,.
    \end{align}
    From \citep[Example 4]{sfx2022phi}, we know for \((S,Y)\) bivariate normal with correlation \(\rho\) that \(\xi(Y,S) = \frac 3 \pi \arcsin\left(\tfrac{1+\rho^2}{2}\right)-\tfrac 1 2 \,.\) Hence, the result follows from \eqref{eqpropChatformmGau2} using that the correlation of \((S,Y)\) is \(\rho = \sqrt{\Sigma_{21} \Sigma_{11}^{-} \Sigma_{12}}\,.\) For the case \(\sigma_Y > 0\), consider the matrix \(\Sigma' =  \Sigma / \sigma_Y^2\) and use scale invariance of \(\xi\,.\) 
\end{proof}


\begin{proof}[Proof of Proposition \ref{propExtcasemn}.]
\eqref{propExtcasemn1}: Due to Theorem \ref{theT}, \(\XX\) and \(\YY\) are independent if and only if \(T(\YY,\XX)=0\,.\) 
Hence, the assertion follows from the well-known property of multivariate normal distributions that \(\XX\) and \(\YY\) are independent if and only if \(\Sigma_{12}\) is the null matrix, see, e.g., \citep[Corollary 2 in Section 2.3]{Fang-1990}.\\
\eqref{propExtcasemn2}:
Consider the decomposition \({\bm{\mu}}=({\bm{\mu}}_1,{\bm{\mu}}_2)\,,\) \({\bm{\mu}}_1\in \R^p\,,\) \({\bm{\mu}}_2\in \R^q\), and define \(k:=\rank(\Sigma)-\rank(\Sigma_{11})\geq 0\,.\) Then, it is well-known that
\begin{align}\label{eqthenormodextcas2}
    (\YY \mid \XX=\xx) & \sim N({\bm{\mu}}_{\xx},\Sigma^*)
\end{align}
with a stochastic representation
\begin{align}\label{repprofsell}
    (\YY \mid \XX=\xx) \eqd {\bm{\mu}}_{\xx} + \ZZ
\end{align}
for \({\bm{\mu}}_{\xx}={\bm{\mu}}_2+(\xx-{\bm{\mu}}_1)\Sigma_{11}^-\Sigma_{12}\) and \(\Sigma^*:=\Sigma_{22}-\Sigma_{21}\Sigma_{11}^-\Sigma_{12}\,,\) where \(\rank(\Sigma^*)=k\) and where \(\ZZ\) is a \(q\)-dimensional \(N(\bf{0},\Sigma^*)\)-distributed random vector that does not depend on \(\xx\,.\)
Here \(A^-\) denotes a generalized inverse of a symmetric matrix \(A\) with positive rank. It follows by the characterization of perfect dependence (see Theorem \ref{theT}) that
\begin{align*}
&&T(\YY,\XX) &=1 \\
&\Longleftrightarrow  &\YY &= {\bf f}(\XX) ~~~\text{a.s.} \\
&\Longleftrightarrow &\YY &= {\bm{\mu}}_2 + (\XX-{\bm{\mu}}_1) \Sigma_{11}^-\Sigma_{12} ~~~ \text{a.s.} \\
&\Longleftrightarrow & k&=0\\
&\Longleftrightarrow &\rank(\Sigma)&=\rank(\Sigma_{11})\,,
\end{align*}
where the second equivalence follows with \eqref{repprofsell}. For the third equality, we observe that \(\ZZ=\bm{0}\) almost surely if and only if \(\rank(\Sigma^*)=0\,.\) Finally, the last equality holds true by the definition of \(k\,.\)
\end{proof}


\section{Proof of Theorem \ref{AN.Thm:AN} and Proposition \ref{propasybias}} \label{App.Sect.AN}

For the proof of our main result, Theorem \ref{AN.Thm:AN}, we need a series of intermediate steps.
To this end, consider a $(p+q)$-dimensional random vector $\ZZ := (\XX,\YY)$
with i.i.d. copies \(\ZZ_l  := (\XX_l,\YY_l)\), \(l \in \{1,\dots,n\}\).
For \(n \in \N\), recall and define 
\begin{align}\label{defkappnalphan}
    \Lambda_n
    & = \sum_{i=1}^{q} \underbrace{\xi_n(Y_i \, , \, (\XX,Y_{i-1},\dots,Y_{1}))}_{:= \xi_{n,i}}\,,
    & \alpha_n 
    & = \sum_{i=2}^q \xi_n(Y_i,(Y_{i-1},\ldots,Y_1))\,,
    \\
    \nonumber \Lambda
    & := \sum_{i=1}^{q} \xi(Y_i \, , \, (\XX,Y_{i-1},\dots,Y_{1}))\,, \notag
    & \alpha   
    &:= \sum_{i=2}^q \xi(Y_i,(Y_{i-1},\ldots,Y_1))\,,
\end{align}
noting that \(\alpha_n = \alpha = 0\) for \(q=1\,.\)
Then 
\begin{align}\label{AN.Rep:T}
  T_n 
  & = 1 - \frac{q - \Lambda_n}{q - \alpha_n}\,,
\end{align}
and, due to Theorem \ref{theT.Consistency},
\(\lim_{n\to \infty} \Lambda_n = \Lambda\) and 
\(\lim_{n\to \infty} \alpha_n = \alpha\), each convergence \(P\)-almost surely.
We prove the desired asymptotic normality 
\begin{align}\label{AN.Aim}
    \frac{T_n - \E[T_n]}{\sqrt{\var(T_n)}} \stackrel{d}{\longrightarrow} N(0,1) 
\end{align}
at the end of this section by combining the below intermediate steps:
The basic idea is to extend the asymptotic normality result for \(\xi_n\) in \citep{han2022limit} to \(T_n\) using a modification of the nearest neighbour-based normal approximation in  \citep[Theorem 3.4.]{chatterjee2008}. It is straightforward to show that both \(\Lambda_n\) and \(\alpha_n\) behave asymptotically normal. However, a direct application of these results to \(T_n\) is not an option,  
as while the first term on the right-hand side of 
\begin{align}\label{AN.Rep:T2}
  \frac{T_n - \E[T_n]}{\sqrt{\var(T_n)}}
  & = \frac{\Lambda_n - \E[\Lambda_n]}{\sqrt{\var(\Lambda_n)}} \, 
      \frac{\sqrt{n \, \var( \Lambda_n)}}{\sqrt{n \, \var(T_n)}} \, 
      \frac{1}{q-\alpha_n} 
      + \frac{\sqrt{n} \, \left( \E\left[\frac{q-\Lambda_n}{q-\alpha_n} \right] - \frac{\E[q-\Lambda_n]}{q-\alpha_n} \right)}{\sqrt{n \, \var(T_n)}}
\end{align}
converges (under some regularity conditions) to a normal distribution, the second term does not converge to $0$ in general.
For this reason, we decompose the above expression into
\begin{align}\label{AN.Rep:T3}
  \frac{T_n - \E[T_n]}{\sqrt{\var(T_n)}}
  & = \frac{\Lambda_n - \E[\Lambda_n]}{\sqrt{\var(\Lambda_n)}} \, 
      \frac{\sqrt{n \, \var( \Lambda_n)}}{\sqrt{n \, \var(T_n)}} \, 
      \frac{1}{q-\alpha_n}  
  \\
  & \qquad
    - \frac{\frac{1}{q-\alpha_n} - \E\left[ \frac{1}{q-\alpha_n}\right]}{\sqrt{\var\left(\frac{1}{q-\alpha_n}\right)}}\, 
      \frac{\sqrt{\var\left(\frac{1}{q-\alpha_n}\right)}}{\sqrt{\var(T_n)}} \,
      \E\left[ q - \Lambda_n \right]
    + \frac{\sqrt{n} \, \cov \left(\frac{1}{q-\alpha_n}, q-\Lambda_n \right)}{\sqrt{n \, \var(T_n)}}  
      \notag
\end{align}
where the last term on the right-hand side is shown to converge to $0$. 
However, the second term is not analytically tractable due to the quotient structure of $\frac{1}{q-\alpha_n}$.
We therefore replace the existing product/quotient structure visible in Eq. \eqref{AN.Rep:T} and \eqref{AN.Rep:T3} with a linear structure and define 
\begin{align} \label{AN.Def:Mu}
  \mu_n := \Lambda_n - \kappa \cdot \alpha_n
\end{align}
where $\kappa := \frac{q - \Lambda}{q - \alpha} \geq 0$.
Then,
\begin{align} \label{AN.Rep:Mu}
  \frac{\mu_n - \E [\mu_n]}{\sqrt{\var(\mu_n)}}
  & = \frac{\Lambda_n - \E[\Lambda_n]}{\sqrt{\var(\Lambda_n)}} \, 
      \frac{\sqrt{n \, \var( \Lambda_n)}}{\sqrt{n \, \var(\mu_n)}} \, 
      - \frac{\alpha_n - \E[\alpha_n]}{\sqrt{\var(\alpha_n)}}\, 
      \frac{\sqrt{\var(\alpha_n)}}{\sqrt{\var(\mu_n)}} \,
      \kappa
\end{align}
mimics Eq. \eqref{AN.Rep:T3} in that the right-hand side of Eq. \eqref{AN.Rep:Mu} and the first two terms in Eq. \eqref{AN.Rep:T3} are closely related.

\begin{enumerate}
\item 
In the first step (Proposition \ref{AN.Lemma:AN.KappaBeta}), we show
\begin{align} \label{AN.Def:Mu1}
  \frac{\mu_n - \E [\mu_n]}{\sqrt{\var(\mu_n)}} \stackrel{d}{\longrightarrow} N(0,1) 
\end{align} 
using a modification of \citep[Theorem 3.4.]{chatterjee2008} and the results in \citep{han2022limit}.

\item 
In Proposition \ref{AN.Lemma:TaylorExp}, we use a Taylor expansion to obtain a transition from $\alpha_n$, which occurs in Eq. \eqref{AN.Rep:Mu}, to $\frac{1}{q-\alpha_n}$, which occurs in Eq. \eqref{AN.Rep:T3}.

\item 
Finally, a series of intermediate results are proven that draw a path from \eqref{AN.Def:Mu1} to \eqref{AN.Aim}.
\end{enumerate}

In order to simplify the above mentioned transition we define
\begin{align}\label{defbeta_nkappa_n}
  \beta_n  
  &:= \frac{q-\alpha}{q-\alpha_n}\,,
  & \kappa_n 
  &:= \frac{q - \Lambda_n}{q - \alpha}\,.
\end{align}
Then $T_n = 1 - \beta_n \cdot \kappa_n$\,,
and, due to Theorem \ref{theT.Consistency},
\(\lim_{n\to \infty} \beta_n = 1\) and \(\lim_{n\to \infty} \kappa_n = \frac{q-\Lambda}{q-\alpha} = \kappa\), each convergence \(P\)-almost surely.
If, additionally, \(\YY\) is not perfectly dependent on \(\XX\), 
then \( \Lambda < q\) and hence \(\kappa = \lim_{n \to \infty} \kappa_n >0\).
Using the just introduced notation, Eq. \eqref{AN.Rep:T3} and Eq. \eqref{AN.Rep:Mu} can be rewritten as
\begin{align}\label{AN.Rep:T3.2}
  \lefteqn{\frac{T_n - \E[T_n]}{\sqrt{\var(T_n)}}} 
  \\
  & = - \, \frac{\kappa_n - \E[\kappa_n]}{\sqrt{\var(\kappa_n)}} \, 
      \frac{\sqrt{n \, \var( \kappa_n)}}{\sqrt{n \, \var(\beta_n \kappa_n)}} \, 
      \beta_n 
      - \frac{\beta_n - \E[\beta_n]}{\sqrt{\var(\beta_n)}}\, 
      \frac{\sqrt{n \, \var(\beta_n)}}{\sqrt{n \, \var(\beta_n \kappa_n)}} \,
      \E[\kappa_n]
      + \frac{\sqrt{n} \, \cov (\beta_n, \kappa_n)}{\sqrt{n \, \var(\beta_n \kappa_n)}}\,,
      \notag
\end{align}
and 
\begin{align} \label{AN.Rep:Mu.2}
  \lefteqn{\frac{\mu_n - \E [\mu_n]}{\sqrt{\var(\mu_n)}}}
  \\
  & = - \, \frac{\kappa_n - \E[\kappa_n]}{\sqrt{\var(\kappa_n)}} \, 
      \frac{\sqrt{n \, \var( \kappa_n)}}{\sqrt{n \, \var(\kappa_n - \kappa \cdot 1/\beta_n)}} \, \cdot 1
      + \frac{1/\beta_n - \E[1/\beta_n]}{\sqrt{\var(1/\beta_n)}}\, 
      \frac{\sqrt{\var(1/\beta_n)}}{\sqrt{\var(\kappa_n - \kappa \cdot 1/\beta_n)}} \,
      \kappa\,.
      \notag
\end{align}
As we will see, the second term in \eqref{AN.Rep:T3.2} behaves similar to the second term \eqref{AN.Rep:Mu.2}, which further elucidates the relationship between $\tfrac{T_n - \E[T_n]}{\sqrt{\var(T_n)}}$ and $\tfrac{\mu_n - \E [\mu_n]}{\sqrt{\var(\mu_n)}}$.

\bigskip\noindent
{\bf Step 1: Asymptotic normality of $\tfrac{\mu_n - \E [\mu_n]}{\sqrt{\var(\mu_n)}}$.}
Similar to the proof of \citep[Theorem 1.1]{han2022limit}, define the H{\'a}jek representations
\begin{align} \label{AN.Def:Hajek}
    \Lambda_n^\ast
    & := \sum_{i=1}^{q} \underbrace{\frac{6n}{n^2-1} \, \left( \sum_{l=1}^n F_{Y_i} (Y_{i,l} \wedge Y_{i,N_i(l)}) 
         + \sum_{l=1}^n g_{i} (Y_{i,l}) \right)}_{:= \xi^{\ast}_{n,i}}\,,
    \\
    \alpha_n^\ast
    & := \sum_{i=2}^{q} \underbrace{\frac{6n}{n^2-1} \, \left( \sum_{l=1}^n F_{Y_i} (Y_{i,l} \wedge Y_{i,M_i(l)}) 
         + \sum_{l=1}^n h_{i} (Y_{i,l}) \right)}_{=: \alpha^{\ast}_{n,i}}\,, \notag
\end{align}
and 
\begin{align}\label{AN.Def:Hajek.Mu} 
    \mu_n^\ast
    & := \Lambda_n^\ast - \kappa \, \alpha_n^\ast\,,
\end{align}
where \(N_i(l)\) represents the index of the nearest neighbor of \((\XX_l,Y_{i-1,l},\ldots,Y_{1,l})\) and where \(M_i(l)\) represents the index of the nearest neighbor of \((Y_{i-1,l},\ldots,Y_{1,l})\,.\)
This leads to the following analogues of \citep[Theorem 1.2 and Proposition 1.2]{han2022limit}.

\begin{lemma}[H{\'a}jek representations]~~\label{AN.Lemma:HajekRep}
Assume \(F_\ZZ\) to be fixed and continuous.
Then
\begin{enumerate}[(i)]
\item \(\lim_{n \to \infty} n \, \var(\Lambda_n - \Lambda_n^\ast) = 0\).
\item \(\lim_{n \to \infty} n \, \var(\alpha_n - \alpha_n^\ast) = 0\).
\item \(\lim_{n \to \infty} n \, \var(\mu_n - \mu_n^\ast) = 0\).
\end{enumerate}
\end{lemma}
\begin{proof}
Due to Cauchy-Schwarz inequality 
\begin{align*}
    n \, \var(\Lambda_n - \Lambda_n^\ast)
    &   =  \sum_{i_1=1}^{q} \sum_{i_2=1}^{q} 
           n \, \cov \left( \xi_{n,i_1} - \xi^{\ast}_{n,i_1},  \xi_{n,i_2} - \xi^{\ast}_{n,i_2} \right)
    \\
    & \leq 
          \sum_{i_1=1}^{q} \sum_{i_2=1}^{q} 
          \underbrace{\sqrt{n \, \var \big( \xi_{n,i_1} - \xi^{\ast}_{n,i_1} \big)}}_{\to \, 0} \, 
          \underbrace{\sqrt{n \, \var \big( \xi_{n,i_2} - \xi^{\ast}_{n,i_2} \big)}}_{\to \, 0}\,,
\end{align*}
where convergence follows from \citep[Theorem 1.2]{han2022limit}.
Hence, \(\lim_{n \to \infty} n \, \var(\Lambda_n - \Lambda_n^\ast) =  0\).
This proves the first assertion, and the second one follows by a similar argument. 
Combining both convergences yields
\begin{align*}
    \lefteqn{n \, \var(\mu_n - \mu_n^\ast)}
    \\
    &   =  n \, \var \left( (\Lambda_n - \Lambda_n^\ast) - \kappa (\alpha_n - \alpha_n^\ast) \right)
    \\
    &   =  n \, \var (\Lambda_n - \Lambda_n^\ast) 
           + \kappa^2 \, n \, \var (\alpha_n - \alpha_n^\ast)
           - 2 \, \kappa \, n \, \cov (\Lambda_n - \Lambda_n^\ast,\alpha_n - \alpha_n^\ast)
    \\
    & \leq \underbrace{n \, \var (\Lambda_n - \Lambda_n^\ast)}_{\to \, 0} 
           + \kappa^2 \,\underbrace{n \, \var (\alpha_n - \alpha_n^\ast)}_{\to \, 0}
           + 2 \, \kappa \, \underbrace{\sqrt{n \, \var (\Lambda_n - \Lambda_n^\ast)}}_{\to \, 0}  \, \underbrace{\sqrt{n \, \var(\alpha_n - \alpha_n^\ast)}}_{\to \, 0} \,,
\end{align*}   
which proves the assertion.
\end{proof}

\begin{lemma}[Asymptotic variance I]~~\label{AN.Lemma:AsympVar}
Assume \(F_\ZZ\) to be fixed and continuous. 
Then
\begin{enumerate}[(1)]
\item \label{Lemma:AsympVar1} 
\begin{enumerate}[(i)]
\item $\limsup_{n \to \infty} n \var(\Lambda_n) < \infty$, 
\item $\limsup_{n \to \infty} n \var(\alpha_n) < \infty$, 
\item $\limsup_{n \to \infty} n \var(\mu_n) < \infty$.
\end{enumerate}
\item \label{Lemma:AsympVar2} 
\begin{enumerate}[(i)]
\item If \(\YY\) is not perfectly dependent on \(\XX\,,\) then $\liminf_{n \to \infty} n \var(\Lambda_n) > 0$.
\item If there exists some \(i \in \{2,\dots,q\}\) such that \(Y_i\) is not perfectly dependent on \(\{Y_1,\dots,Y_{i-1}\}\), then $\liminf_{n \to \infty} n \var(\alpha_n) > 0$.
\item If \(\YY\) is not perfectly dependent on \(\XX\,,\) there exists some \(i \in \{2,\dots,q\}\) such that \(Y_i\) is not perfectly dependent on \(\{Y_1,\dots,Y_{i-1}\}\) and $\limsup_{n \to \infty} \Cor(\Lambda_n,\alpha_n) < 1$,
then $\liminf_{n \to \infty} n \, \var(\mu_n) > 0$.
\end{enumerate}
\end{enumerate}
\end{lemma}
\begin{proof}
\eqref{Lemma:AsympVar1}:
Due to Cauchy-Schwarz inequality,
\begin{align*}
    n \, \var(\Lambda_n)
    &   =  \sum_{i_1=1}^{q} \sum_{i_2=1}^{q} n \, \cov \big( \xi_{n,i_1}, \xi_{n,i_2} \big)
      \leq \sum_{i_1=1}^{q} \sum_{i_2=1}^{q} 
           \sqrt{n \, \var \big( \xi_{n,i_1} \big)} \, 
           \sqrt{n \, \var \big( \xi_{n,i_2} \big)}\,.
\end{align*}
Now, \citep[Proposition 1.2]{han2022limit} gives $\limsup_{n \to \infty} n \, \var(\xi_{n,i}) < \infty$ for all $i \in \{1,\dots,q\}$ and thus
$\limsup_{n \to \infty} n \, \var(\Lambda_n) < \infty$.
The result for $\alpha_n$ follows by a similar argument, 
and the result for $\mu_n$ then is immediate from Eq. \eqref{AN.Def:Mu} and Cauchy-Schwarz inequality.  
\\
\eqref{Lemma:AsympVar2}:
Since \(\YY\) is not perfectly dependent on \(\XX\,,\) 
there exists some \(i \in \{1\,\ldots,q\}\) such that \(Y_i\) is not perfectly dependent on \((\XX,Y_{i-1},\ldots,Y_{1})\) (see Eq. \eqref{protheT7}).
Now, choose \[i^*:= \max\{i\mid Y_i \text{ is not perfectly dependent on } (\XX,Y_{i-1},\ldots,Y_{1})\}\,.\] Then we obtain
\begin{align*}
    n \, \var(\Lambda_n)
    & \geq n \, \E \, \big[ \var(\Lambda_n \, | \, (\XX,Y_{i^*-1},\dots,Y_{1})) \big] \notag
    \\
    &   =  n \, \E \left[ \var \left( \left( \sum_{k = i^*+1}^q \xi_{n,k} + \xi_{n,i^*} + \sum_{\ell=1}^{i^*-1} \xi_{n,\ell} \right) \mid (\XX,Y_{i^*-1},\dots,Y_{1}) \right) \right] \notag
    \\
    &   =  \E \big[ n \, \var \big( \xi_{n,i^*} \mid (\XX,Y_{i^*-1},\dots,Y_{1}) \big) \big] \,, \notag
\end{align*}
where we use for the last equality on the one hand that \(\xi_{n,\ell}\) is conditionally on \((\XX,Y_{i^*-1},\ldots,Y_{1})\) constant for all \(\ell < i^*\,.\)
On the other hand, \(\xi_{n,k}\) is almost surely constant for all \(k>i^*\) because \(Y_k\) is a measurable function of \((\XX,Y_{k-1},\ldots,Y_{1})\), see \citep[Remark 1.2]{han2022limit}.
Proceeding as in the proof of \citep[Proposition 1.2]{han2022limit} gives \(\liminf_{n \to \infty} n \, \var(\Lambda_n) > 0\).
The result for $\alpha_n$ follows by a similar argument.
To finally show the result for \(\mu_n\), we first calculate
\begin{align*}
    \var(\mu_n)
    & = \var(\Lambda_n - \kappa \, \alpha_n)
      = \var(\Lambda_n) + \var(\kappa \, \alpha_n) - 2 \, \Cov(\Lambda_n, \kappa \, \alpha_n)
    \\
    & = \var(\Lambda_n) + \kappa^2 \, \var(\alpha_n) - 2 \, \kappa \, \Cor(\Lambda_n, \alpha_n) \, \sqrt{\var(\Lambda_n)} \, \sqrt{\var(\alpha_n)}
    \\
    & = \left( \sqrt{\var(\Lambda_n)} - \kappa \, \sqrt{\var(\alpha_n)} \right)^2 
        + 2 \sqrt{\var(\Lambda_n)} \, \kappa \, \sqrt{\var(\alpha_n)} \,  (1-\Cor(\Lambda_n, \alpha_n))\,,
\end{align*}
which gives
\begin{align*}
    \liminf_{n \to \infty} n \,\var(\mu_n)
    & \geq 2 \, \underbrace{\liminf_{n \to \infty} \sqrt{n \, \var(\Lambda_n)}}_{> 0} \, \kappa \, \underbrace{\liminf_{n \to \infty} \sqrt{n \, \var(\alpha_n)}}_{> 0}  \, \left( 1 - \underbrace{\limsup_{n \to \infty} \Cor(\Lambda_n, \alpha_n)}_{< 1}  \right) > 0\,.
\end{align*}
This proves the result.
\end{proof}

For vectors \(\xx \in \R^p\) and \(\yy \in \R^q\) consider the combined vector \(\zz = (\xx,\yy)\),
and for \(\xx_1, \dots, \xx_n \in \R^p\) and \(\yy_1, \dots, \yy_n \in \R^q\) define 
\([\zz]_n := (\zz_1,\dots,\zz_n)\).
For every \(i \in \{1,\dots,q\}\), further define \(\zz^{(i)} := (\xx,y_i,\dots,y_1)\) with \(\zz^{(0)} := \xx\) and \(\yy^{(i)} := (y_i,\dots,y_1)\).
Now, for \(n \in \N\) such that \(n \geq 4\) define in correspondence to Eq. \eqref{AN.Def:Hajek}
\begin{align} \label{AN.Def:Wn}
    W_n ([\zz]_n)
    &  := \frac{1}{\sqrt{n}} \sum_{i=1}^{q} \sum_{l=1}^n \underbrace{\big[ F_{Y_i} (y_{i,l} \wedge y_{i,N_i(l)}) 
           + g_{i} (y_{i,l}) \big]}_{=: e_{i,l}([\zz^{(i)}]_n)} 
    \\
    & \qquad
         - \frac{\kappa}{\sqrt{n}} \sum_{i=2}^{q} \sum_{l=1}^n \underbrace{\big[ F_{Y_i} (y_{i,l} \wedge y_{i,M_i(l)}) 
         - h_{i} (y_{i,l}) \big]}_{=: f_{i,l}([\yy^{(i)}]_n)}\,. \notag
\end{align}
Then \( e_{i,l} ([\zz^{(i)}]_n) = e_{i,l} (\zz^{(i)}_1, \dots, \zz^{(i)}_n)\) is a function of only \(\zz^{(i)}_l\) and \(\zz^{(i)}_{N_i(l)}\) 
where \(N_i(l)\) represents the index of the nearest neighbor of \(\zz_l^{(i-1)} = (\xx_l,y_{i-1,l},\dots,y_{1,l})\) in the nearest neighbor graph constructed by \([\zz^{(i-1)}]_n\,.\)
Analogously, \( f_{i,l} ([\yy^{(i)}]_n) = f_{i,l} (\yy^{(i)}_1, \dots, \yy^{(i)}_n)\) is a function of only \(\yy^{(i)}_l\) and \(\yy^{(i)}_{M_i(l)}\) 
where \(M_i(l)\) represents the index of the nearest neighbor of \(\yy_l^{(i-1)} = (y_{i-1,l},\dots,y_{1,l})\) in the nearest neighbor graph constructed by \([\yy^{(i-1)}]_n\,.\)
Notice that \(\sqrt{n} \, \mu_n^\ast = \frac{6n^2}{n^2-1} W_n ([\ZZ]_n)\),
and thus
\begin{align*}
    \frac{\mu_n^\ast - \E [\mu_n^\ast]}{\sqrt{\var(\mu_n^\ast)}}
    & = \frac{W_n ([\ZZ]_n) - \E[W_n ([\ZZ]_n)]}{\sqrt{\var(W_n ([\ZZ]_n))}}\,.
\end{align*}
Corollary \ref{AN.Cor:AsymptNorm} below shows asymptotic normality of $\tfrac{\mu_n^\ast - \E [\mu_n^\ast]}{\sqrt{\var(\mu_n^\ast)}}$. 
For its proof, we modify \citep[Theorem 3.4]{chatterjee2008} so that it is applicable to the function \(W_n\) defined in Eq. \eqref{AN.Def:Wn} and hence to \(\mu_n^\ast\).

\begin{theorem} [Modification of Theorem 3.4 in \citep{chatterjee2008}] 
\label{Thm.Modify.Chatterjee}~
Fix \(n\geq 4\,,\) \(d\geq 1\,,\) and \(k\geq 1\,.\) Suppose \(\VV_1,\ldots,\VV_n\) are i.i.d. \(\R^d\)-valued random vectors with the property that \(\lVert \VV_1-\VV_2\rVert\) is a continuous random variable. Let \(f\colon (\R^d)^n \to \R\) be a function of the form
\begin{align*}
  f(\vv_1,\ldots,\vv_n) 
  & = \frac{1}{\sqrt{n}} \sum_{I\subseteq \{1,\ldots,d\} \atop I \ne \emptyset} \sum_{\ell = 1}^n f_{I,\ell}(\vv_1,\ldots,\vv_n)\,,
\end{align*}
where, for each \(I\) and \(\ell\,,\) the function \(f_{I,\ell}\) depends only on \(\vv_\ell\) and its \(k\) nearest neighbors built by the components \(I\subseteq \{1,\ldots,d\}\,,\) \(I\ne \emptyset\,.\) Suppose, for some \(r\geq 8\,,\) 
that \(\gamma_r:= \max_{I,\ell} \E|f_{I,\ell}(\VV_1,\ldots,\VV_n)|^r\) is finite. Let \(W = f(\VV_1,\ldots,\VV_n)\) and \(\sigma^2 = \Var(W)\,.\) Then
\begin{align}
  \delta_W \leq C \frac{\alpha(d)^3k^4\gamma_r^{2/r}}{\sigma^2 n^{(r-8)/2r}}+C\frac{\alpha(d)^3 k^3 \gamma_r^{3/r}}{\sigma^3 n^{(r-6)/2r}}\,,
\end{align}
where 
\(\delta_W\) denotes the Kantorovich-Wasserstein distance between \(W\) and the standard Gaussian distribution, 
\(\alpha(d)\) is the minimum number of \(60 \degree\) cones at the origin required to cover \(\R^d\), and \(C\) is a universal constant.
\end{theorem}

\begin{remark}
Theorem \ref{Thm.Modify.Chatterjee} is a slight modification of \citep[Theorem 3.4]{chatterjee2008} noting that the nearest neighbor graph in the reference can also be constructed with respect to subcomponents of the underlying vectors.
The interaction rule in the proof has to be replaced by a modified interaction rule for \(f\) where the graph \(G([\vv]_n)\) on \(\{1,\ldots,n\}\times \{1,\ldots,n\}\) puts an edge between the nodes \(m_1\) and \(m_2\) if, for some \(I\subseteq \{1,\ldots,d\}\), the graph \(G_I([\vv]_n)\) of 
\begin{align*}
  f_I(\vv_1,\ldots,\vv_n) 
  & = \frac{1}{\sqrt{n}} \sum_{\ell = 1}^n f_{I,\ell}(\vv_1,\ldots,\vv_n)
\end{align*}
puts an edge between \(m_1\) and \(m_2\).
\end{remark}

Theorem \ref{Thm.Modify.Chatterjee} leads to the following result.

\begin{corollary}[Asymptotic normality of $\Lambda_n^\ast$, $\alpha_n^\ast$ and $\mu_n^\ast$]\label{AN.Cor:AsymptNorm}
Assume \(F_\ZZ\) to be fixed and continuous.
Then 
\begin{align*}
  \frac{\Lambda_n^\ast - \E[\Lambda_n^\ast]}{\sqrt{\var(\Lambda_n^\ast)}}
  \stackrel{d}{\longrightarrow} N(0,1)\,,
  \qquad 
  \frac{\alpha_n^\ast - \E[\alpha_n^\ast]}{\sqrt{\var(\alpha_n^\ast)}}
  \stackrel{d}{\longrightarrow} N(0,1)\,,
  \qquad 
  \frac{\mu_n^\ast - \E[\mu_n^\ast]}{\sqrt{\var(\mu_n^\ast)}}
  \stackrel{d}{\longrightarrow} N(0,1)\,.
\end{align*}
\end{corollary}
\begin{proof}
The result for $\mu_n^\ast$ is a direct consequence of Theorem \ref{Thm.Modify.Chatterjee} and the definition of $W_n$ in \eqref{AN.Def:Wn}, which implies, in particular, the results for $\Lambda_n^\ast$ and $\alpha_n^\ast$.
\end{proof}

Combining the previously obtained results yields asymptotic normality of $\tfrac{\mu_n - \E [\mu_n]}{\sqrt{\var(\mu_n)}}$ as follows.

\begin{proposition}~~\label{AN.Lemma:AN.KappaBeta}
Assume \(F_\ZZ\) to be fixed and continuous.
\begin{enumerate}[(1)]
\item 
If \(\YY\) is not perfectly dependent on \(\XX\,,\) then   
\begin{align}\label{AN.Lemma.Eq:AN.Kappa}
  \frac{\kappa_n - \E[\kappa_n]}{\sqrt{\var(\kappa_n)}} 
  & =  - \frac{\Lambda_n - \E[\Lambda_n]}{\sqrt{\var(\Lambda_n)}} 
    \stackrel{d}{\longrightarrow} N(0,1)\,.
\end{align}

\item 
If there exists some \(i \in \{2,\dots,q\}\) such that \(Y_i\) is not perfectly dependent on \(\{Y_1,\dots,Y_{i-1}\}\), then
\begin{align}\label{AN.Lemma.Eq:AN.Beta}
  \frac{1/\beta_n - \E[1/\beta_n]}{\sqrt{\var(1/\beta_n)}} 
  & = - \frac{\alpha_n - \E[\alpha_n]}{\sqrt{\var(\alpha_n)}} 
    \stackrel{d}{\longrightarrow} N(0,1)\,. 
\end{align}

\item 
If \(\YY\) is not perfectly dependent on \(\XX\), if
there exists some \(i \in \{2,\dots,q\}\) such that \(Y_i\) is not perfectly dependent on \(\{Y_1,\dots,Y_{i-1}\}\) and if
$\limsup_{n \to \infty} \Cor(\Lambda_n,\alpha_n) < 1$, 
then
\begin{align}\label{AN.Lemma.Eq:AN.Mu}
  \frac{\mu_n - \E[\mu_n]}{\sqrt{\var(\mu_n)}} 
  & \stackrel{d}{\longrightarrow} N(0,1)\,. 
\end{align}
\end{enumerate}    
\end{proposition}
\begin{proof}
We first show that 
\begin{align*}
  \frac{\Lambda_n - \E[\Lambda_n]}{\sqrt{\var(\Lambda_n)}} \stackrel{d}{\longrightarrow} N(0,1)\,.
\end{align*}
By Lemma \ref{AN.Lemma:AsympVar} and Lemma \ref{AN.Lemma:HajekRep} we have
\begin{align*}
  \limsup_{n \to \infty} \; \E \left[ \left(\frac{\Lambda_n^\ast - \E(\Lambda_n^\ast)}{\sqrt{\var(\Lambda_n)}} - \frac{\Lambda_n - \E(\Lambda_n)}{\sqrt{\var(\Lambda_n)}} \right)^2 \right]
  &   =  \limsup_{n \to \infty} \; \frac{\var(\Lambda_n^\ast - \Lambda_n)}{\var(\Lambda_n)}
  \\
  & \leq \frac{\limsup_{n \to \infty} \var(\Lambda_n^\ast - \Lambda_n)}{\liminf_{n \to \infty} \var(\Lambda_n)}
         = 0
\end{align*}
and 
\begin{align*}
  \limsup_{n \to \infty} \; \left| \frac{\cov(\Lambda_n, \Lambda_n^\ast - \Lambda_n)}{\var(\Lambda_n)} \right|
  & \leq \limsup_{n \to \infty} \left( \frac{\var(\Lambda_n^\ast - \Lambda_n)}{\var(\Lambda_n)} \right)^{1/2}
          = 0\,,
\end{align*}
and thus
\begin{align*}
    \frac{\Lambda_n^\ast - \E[\Lambda_n^\ast]}{\sqrt{\var(\Lambda_n)}} - \frac{\Lambda_n - \E[\Lambda_n]}{\sqrt{\var(\Lambda_n)}} 
    & \stackrel{P}{\longrightarrow} 0 \qquad \textrm{ and }
    & \frac{\var(\Lambda_n^\ast)}{\var(\Lambda_n)}
    & \longrightarrow 1\,.
\end{align*}
Together with Corollary \ref{AN.Cor:AsymptNorm} and Slutsky's theorem, this proves \eqref{AN.Lemma.Eq:AN.Kappa}. 
The convergences in \eqref{AN.Lemma.Eq:AN.Beta} and \eqref{AN.Lemma.Eq:AN.Mu} follow by similar arguments. 
\end{proof}

\bigskip\noindent
{\bf Step 2: Taylor expansion.}
We use a Taylor expansion and a series of intermediate results that deal with the transition from $1/\beta_n = \tfrac{q-\alpha_n}{q-\alpha}$ to $\beta_n = \tfrac{q-\alpha}{q-\alpha_n}$.
We first show that \(\beta_n\) and \(1/\beta_n\) share a similar asymptotic behaviour, indicating that in many situations one can be replaced with the other.

\begin{lemma}[Uniform boundedness of \(\beta_n\) and \(T_n\)]~~\label{AN.Lemma:Finite.Beta}
Assume \(F_\ZZ\) to be fixed and continuous.
Then, it holds that \(1/(3q) \leq \beta_n \leq q\) for all \(n\in \N\,.\)
In particular, \(-1-\frac{1}{3q-2} \leq T_n \leq 1\).
\end{lemma}
\begin{proof}
For \(\xi_n\) one has the trivial bounds
\begin{align*}
  \xi_n 
  &   =  \frac{6}{n^2-1} \sum_{i=1}^n \underbrace{\min\{R_i,R_{N(i)}\}}_{\leq R_i} - \frac{2n+1}{n-1} 
    \leq \frac{6 (n+1)n}{2(n+1)(n-1)} - \frac{2n+1}{n-1} = 1
\end{align*}
and
\begin{align*}
  \xi_n 
  &   =  \frac{6}{n^2-1} \sum_{i=1}^n \underbrace{\min\{R_i,R_{N(i)}\}}_{\geq 1} - \frac{2n+1}{n-1}
    \geq \frac{6n}{(n+1)(n-1)}-\frac{2n+1}{n-1} = -\frac{2 n-1}{n+1}\,.
\end{align*}
Hence, since \(\alpha\in [0,q-1]\) and \(\xi_n(Y_1,\emptyset)=0\,,\) it follows that
\begin{align}\label{boundalpha_n}
  \frac 1 {3q} \leq \frac{q-\alpha}{3q} \leq \frac{q-\alpha}{q+(q-1)\frac{2 n-1}{n+1}} 
  \leq \frac{q - \alpha}{q-\alpha_n} = \beta_n
  \leq \frac{q-\alpha}{q-(q-1)} 
  \leq q 
\end{align}
\(P\)-almost surely for all \(n\in \N\,.\) 
The remaining assertion results from straightforward calculation incorporating the proven bounds for $\xi_n$.
\end{proof}

Since \(1/\beta_n = \tfrac{q-\alpha_n}{q-\alpha}\) is a linear function of \(\alpha_n\), the following statement is immediate from Lemma \ref{AN.Lemma:AsympVar}.

\begin{lemma}[Asymptotic variance II]~~\label{AN.Lemma:AsympVar.1/beta}
Assume \(F_\ZZ\) to be fixed and continuous. Then
\begin{enumerate}[(i)]
\item 
$\limsup_{n \to \infty} n \, \var(1/\beta_n) < \infty$.
\item \label{AN.Lemma:AsympVar.1/betab}
If there exists some \(i \in \{2,\dots,q\}\) such that \(Y_i\) is not perfectly dependent on \(\{Y_1,\dots,Y_{i-1}\}\), 
        then $\liminf_{n \to \infty} n \, \var(1/\beta_n) > 0$.
\end{enumerate}
\end{lemma}

By applying the Taylor expansion in Eq. \eqref{AN.Eq:TaylorExp} below, Proposition \ref{AN.Lemma:TaylorExp} verifies that the statements concerning $1/\beta_n$ in Proposition \ref{AN.Lemma:AN.KappaBeta} and Lemma \ref{AN.Lemma:AsympVar.1/beta} can also be formulated in terms of $\beta_n$.

\begin{proposition}[Taylor expansion]~~\label{AN.Lemma:TaylorExp}
Assume \(F_\ZZ\) to be fixed and continuous.
If 
\begin{align} \label{AN.Lemma:TaylorExp.UI}
    \sup_{n \in \mathbb{N}} \, \E \left[ \left| \frac{1/\beta_n - \E[1/\beta_n]}{\sqrt{\var(1/\beta_n)}} \right|^{2+\delta} \right] < \infty
\end{align}
for some \(\delta > 0\)
and if there exists some \(i \in \{1,\dots,q\}\) such that \(Y_i\) is not perfectly dependent on \(\{Y_1,\dots,Y_{i-1}\}\), then 
\begin{enumerate}[(i)]
\item \label{AN.Lemma:TaylorExp.UI.3} 
\(\lim_{n \to \infty} n \, \left| \var(1/\beta_n) - \var(\beta_n) \right| = 0\).
\item \label{AN.Lemma:TaylorExp.UI.4} 
$\limsup_{n \to \infty} n \, \var(\beta_n) < \infty$ and $\liminf_{n \to \infty} n \, \var(\beta_n) > 0$.
\item \label{AN.Lemma:TaylorExp.UI.5} 
\(\frac{\beta_n - \E(\beta_n)}{\sqrt{\var(\beta_n)}} \xrightarrow{d} N(0,1)\).

\item \label{AN.Lemma:TaylorExp.UI.6}  
$\tfrac{\beta_n - \E[\beta_n]}{\sqrt{\var(\beta_n)}} - \left( - \tfrac{1/\beta_n - \E[1/\beta_n]}{\sqrt{\var(1/\beta_n)}} \right) \stackrel{P}{\longrightarrow} 0$.
\end{enumerate}
\end{proposition}

\begin{proof}
According to Lemma \ref{AN.Lemma:Finite.Beta}, it holds that \(0 < \frac{1}{q} \leq \E[1/\beta_n] \leq 3q\).
Now, 
define
\begin{align*}
  A_n := \left\{ \omega \in \Omega \, : \, 
          \left| \frac{1/\beta_n(\omega) - \E[1/\beta_n]}{\E[1/\beta_n]} \right| < 1 \right\}\,.    
\end{align*}
Then, for every $\omega \in A_n$, the Taylor expansion of \(f(x) = 1/x\) at point \(\E[1/\beta_n]\) gives
\begin{align}\label{AN.Eq:TaylorExp}
  \beta_n (\omega) 
    = f(1/\beta_n(\omega)) 
  &  = \frac{1}{\E[1/\beta_n]} - \frac{1/\beta_n(\omega) - \E[1/\beta_n]}{\E[1/\beta_n]^2} + R_1(1/\beta_n(\omega)) 
  \\
  & = \frac{2}{\E[1/\beta_n]} - \frac{1/\beta_n(\omega)}{\E[1/\beta_n]^2} + R_1(1/\beta_n(\omega))\,,
  \notag
\end{align}
where
\begin{align*}
  R_1(x) 
  & := \frac{1}{\E[1/\beta_n]} \,  
      \sum_{k=2}^\infty (-1)^k \, \left( \frac{x - \E[1/\beta_n]}{\E[1/\beta_n]} \right)^k \,. 
\end{align*}
This yields
\begin{align}\label{AN.Eq:TaylorExp2}
  \beta_n 
  &   =  \beta_n \, \mathds{1}_{A_n} + \beta_n \, \mathds{1}_{A_n^c} 
  \\
  &   =  \left( \frac{2}{\E[1/\beta_n]} - \frac{1/\beta_n}{\E[1/\beta_n]^2} + R_1(1/\beta_n) \right) \, \mathds{1}_{A_n} 
         + \beta_n \, \mathds{1}_{A_n^c} 
         \notag
  \\
  &   =  \left( \frac{2}{\E[1/\beta_n]} - \frac{1/\beta_n}{\E[1/\beta_n]^2}  \right)  
         + R_1(1/\beta_n) \, \mathds{1}_{A_n}
         + \left( \beta_n - \frac{2}{\E[1/\beta_n]} + \frac{1/\beta_n}{\E[1/\beta_n]^2} \right) \, \mathds{1}_{A_n^c}\,.
         \notag
\end{align}
We will make use of Eq. \eqref{AN.Eq:TaylorExp2} on several occasions.

We first show that the remainder of the Taylor expansion in Eq. \eqref{AN.Eq:TaylorExp} vanishes under the bounded moment assumption \eqref{AN.Lemma:TaylorExp.UI}:
\begin{align} \label{AN.Lemma:TaylorExp.UI.1} 
  \lim_{n \to \infty} n \, \E \left[ R_1(1/\beta_n)^2 \, \mathds{1}_{A_n} \right] 
 & = 0 \,,\\
  \label{AN.Lemma:TaylorExp.UI.1b}
  \lim_{n \to \infty} n \, \E \left[ \left( \beta_n - \frac{2}{\E[1/\beta_n]} + \frac{1/\beta_n}{\E[1/\beta_n]^2} \right)^2 \, \mathds{1}_{A_n^c} \right] 
  &= 0 \,.
\end{align}
This implies
\begin{align} \label{AN.Lemma:TaylorExp.UI.2} 
  \lim_{n \to \infty} n \, \var \left( R_1(1/\beta_n) \, \mathds{1}_{A_n} \right) 
  = 0
  \quad \textrm{and} \quad
  \lim_{n \to \infty} n \, \var \left( \left( \beta_n - \frac{2}{\E[1/\beta_n]} + \frac{1/\beta_n}{\E[1/\beta_n]^2} \right) \, \mathds{1}_{A_n^c} \right) 
  = 0\,.
\end{align}
To show \eqref{AN.Lemma:TaylorExp.UI.1}, we have by the definition of $A_n$ that
\begin{align*}
  R_1(1/\beta_n) \, \mathds{1}_{A_n}
  & =  \mathds{1}_{A_n} \, \frac{1}{\E[1/\beta_n]} \, 
      \sum_{k=2}^\infty (-1)^k \, \left( \frac{1/\beta_n - \E[1/\beta_n]}{\E[1/\beta_n]} \right)^k 
  \\
  & =  \mathds{1}_{A_n} \, \frac{1}{\E[1/\beta_n]} \, \left( \frac{1/\beta_n - \E[1/\beta_n]}{\E[1/\beta_n]} \right)^2  \,
      \sum_{k=2}^\infty (-1)^{k-2} \, \left( \frac{1/\beta_n - \E[1/\beta_n]}{\E[1/\beta_n]} \right)^{k-2}
  \\
  & = \mathds{1}_{A_n}\, \left( \frac{1/\beta_n - \E[1/\beta_n]}{\E[1/\beta_n]} \right)^2 \, \beta_n\,.
\end{align*}
Due to Lemma \ref{AN.Lemma:Finite.Beta}, the definition of $A_n$ and the bounded moment assumption \eqref{AN.Lemma:TaylorExp.UI}, for some $\delta > 0$, we obtain
\begin{align} 
  n \, \E \left[ R_1(1/\beta_n)^2 \, \mathds{1}_{A_n} \right]
  &   =  n \, \E \left[ \mathds{1}_{A_n} \, 
        \left| \frac{1/\beta_n - \E[1/\beta_n]}{\E[1/\beta_n]} \right|^4  
        \beta_n^2 \right] \notag
  \\
  \label{AN.Lemma:TaylorExp:Proof.1}& \leq q^2 \, n \,  
         \E \left[ \mathds{1}_{A_n} \, \left| \frac{1/\beta_n - \E[1/\beta_n]}{\E[1/\beta_n]} \right|^{2+\delta} \right] 
  \\
  & \leq \frac{q^2}{\E[1/\beta_n]^{2+\delta}} \,  
         \underbrace{\E \left[ \left| \frac{1/\beta_n - \E[1/\beta_n]}{\sqrt{\var(1/\beta_n)}} \right|^{2+\delta} \right]}_{< \infty} \, (n \, \var(1/\beta_n)) \, \underbrace{(\sqrt{\var(1/\beta_n)})^\delta}_{\longrightarrow 0}\,, \notag
\end{align}
where we use \(\limsup_{n\to \infty} n \Var(1/\beta_n)< \infty\) by Lemma \ref{AN.Lemma:AsympVar.1/beta}. This proves \eqref{AN.Lemma:TaylorExp.UI.1}.
Applying Markov's inequality together with Lemma \ref{AN.Lemma:AsympVar.1/beta} yields
\begin{align*}
    n \, \E [\mathds{1}_{A_n^c}]
    & = n \, P (A_n^c)
      = n \, P \left( \left| \frac{1/\beta_n - \E[1/\beta_n]}{\E[1/\beta_n]} \right| \geq 1 \right)
      \leq \frac{n \, \var(1/\beta_n)}{\E[1/\beta_n]^2}
      \leq M < \infty
\end{align*}
for some \(M \in (0,\infty)\). Hence, the convergence in \eqref{AN.Lemma:TaylorExp.UI.1b} follows from
\begin{align*}
  n \, \E \left[ \left( \beta_n - \frac{2}{\E[1/\beta_n]} + \frac{1/\beta_n}{\E[1/\beta_n]^2} \right)^2 \, \mathds{1}_{A_n^c} \right] 
  & \leq  M \, \E \left[ \left(  \underbrace{\beta_n - \frac{2}{\E[1/\beta_n]} + \frac{1/\beta_n}{\E[1/\beta_n]^2}}_{\longrightarrow 0} \right)^4 \right]^{1/2}
          \longrightarrow 0\,,
\end{align*}
where we use Hölder's inequality and the boundedness of \(\beta_n\) due to Lemma \ref{AN.Lemma:Finite.Beta}.

We now prove \eqref{AN.Lemma:TaylorExp.UI.3}. 
Because of Eq. \eqref{AN.Eq:TaylorExp2} and Cauchy-Schwarz inequality, we obtain
\begin{align*}
  \lefteqn{n \, \left| \var(1/\beta_n) 
           - \var(\beta_n) \right|}
  \\
  &   =  n \, \left| \var(1/\beta_n) 
         - \var\left( \left( \frac{2}{\E[1/\beta_n]} - \frac{1/\beta_n}{\E[1/\beta_n]^2} \right)  + R_1(1/\beta_n) \, \mathds{1}_{A_n} + \left( \beta_n - \frac{2}{\E[1/\beta_n]} + \frac{1/\beta_n}{\E[1/\beta_n]^2} \right) \, \mathds{1}_{A_n^c} \right) \right|
  \\
  & \leq n \, \left| \var(1/\beta_n) 
         - \var\left( \frac{2}{\E[1/\beta_n]} - \frac{1/\beta_n}{\E[1/\beta_n]^2} \right) \right|
         + n \, \var\left( R_1(1/\beta_n) \, \mathds{1}_{A_n} \right)
  \\
  & \qquad + n \, \var\left( \left( \beta_n - \frac{2}{\E[1/\beta_n]} + \frac{1/\beta_n}{\E[1/\beta_n]^2} \right) \, \mathds{1}_{A_n^c} \right)
           + 2n \, \left| \cov \left( \left( \frac{2}{\E[1/\beta_n]} - \frac{1/\beta_n}{\E[1/\beta_n]^2} \right), R_1(1/\beta_n) \, \mathds{1}_{A_n} \right) \right|
  \\
  & \qquad + 2n \, \left| \cov \left( \left( \frac{2}{\E[1/\beta_n]} - \frac{1/\beta_n}{\E[1/\beta_n]^2} \right), \left( \beta_n - \frac{2}{\E[1/\beta_n]} + \frac{1/\beta_n}{\E[1/\beta_n]^2} \right) \, \mathds{1}_{A_n^c} \right) \right|
  \\
  & \qquad + 2n \, \left| \cov \left( R_1(1/\beta_n) \, \mathds{1}_{A_n}, \left( \beta_n - \frac{2}{\E[1/\beta_n]} + \frac{1/\beta_n}{\E[1/\beta_n]^2} \right) \, \mathds{1}_{A_n^c} \right) \right|
  \\
  & \leq n \, \var(1/\beta_n) \left| \underbrace{1 - \frac{1}{\E[1/\beta_n]^4}}_{\longrightarrow 0} \right|
         + \underbrace{n \, \var\left( R_1(1/\beta_n) \, \mathds{1}_{A_n} \right)}_{\longrightarrow 0; \, \textrm{by } \eqref{AN.Lemma:TaylorExp.UI.2}}
  \\
  & \qquad + \underbrace{n \, \var\left( \left( \beta_n - \frac{2}{\E[1/\beta_n]} + \frac{1/\beta_n}{\E[1/\beta_n]^2} \right) \, \mathds{1}_{A_n^c} \right)}_{\longrightarrow 0; \, \textrm{by } \eqref{AN.Lemma:TaylorExp.UI.2}}
           + 2 \, \underbrace{\frac{\sqrt{n \, \var \left( 1/\beta_n \right)} \, \sqrt{n \, \var \left(R_1(1/\beta_n) \, \mathds{1}_{A_n}\right)}}{\E[1/\beta_n]^2}}_{\longrightarrow 0; \, \textrm{by } \eqref{AN.Lemma:TaylorExp.UI.2}, \textrm{ Lemma \ref{AN.Lemma:AsympVar.1/beta}}}
  \\
  & \qquad + 2 \, \underbrace{\frac{\sqrt{n \, \var \left( 1/\beta_n \right)} \, \sqrt{n \, \var \left(\left( \beta_n - \frac{2}{\E[1/\beta_n]} + \frac{1/\beta_n}{\E[1/\beta_n]^2} \right) \, \mathds{1}_{A_n^c}\right)}}{\E[1/\beta_n]^2}}_{\longrightarrow 0; \, \textrm{by } \eqref{AN.Lemma:TaylorExp.UI.2}, \textrm{ Lemma \ref{AN.Lemma:AsympVar.1/beta}}}
  \\
  & \qquad + 2 \, \underbrace{\sqrt{n \, \var \left(R_1(1/\beta_n) \, \mathds{1}_{A_n}\right)} \, \sqrt{n \, \var \left(\left( \beta_n - \frac{2}{\E[1/\beta_n]} + \frac{1/\beta_n}{\E[1/\beta_n]^2} \right) \, \mathds{1}_{A_n^c}\right)}}_{\longrightarrow 0; \textrm{by } \eqref{AN.Lemma:TaylorExp.UI.2}} \longrightarrow 0
\end{align*}
where we use \(\limsup_{n\to \infty} n \Var(1/\beta_n)< \infty\) by Lemma \ref{AN.Lemma:AsympVar.1/beta}.

Lemma \ref{AN.Lemma:AsympVar.1/beta} together with \eqref{AN.Lemma:TaylorExp.UI.3} implies \eqref{AN.Lemma:TaylorExp.UI.4}.

We now prove \eqref{AN.Lemma:TaylorExp.UI.5}. Due to
Eq. \eqref{AN.Eq:TaylorExp2} we obtain 
\begin{align} \label{AN.Lemma:TaylorExp.ID2}
  \lefteqn{\beta_n - \E[\beta_n]}
  \\
  & = \left( \frac{2}{\E[1/\beta_n]} - \frac{1/\beta_n}{\E[1/\beta_n]^2}  \right)  
         + R_1(1/\beta_n) \, \mathds{1}_{A_n}
         + \left( \beta_n - \frac{2}{\E[1/\beta_n]} + \frac{1/\beta_n}{\E[1/\beta_n]^2} \right) \, \mathds{1}_{A_n^c} \notag
  \\
  & \qquad 
         - \E \left[\left( \frac{2}{\E[1/\beta_n]} - \frac{1/\beta_n}{\E[1/\beta_n]^2}  \right)  
         + R_1(1/\beta_n) \, \mathds{1}_{A_n}
         + \left( \beta_n - \frac{2}{\E[1/\beta_n]} + \frac{1/\beta_n}{\E[1/\beta_n]^2} \right) \, \mathds{1}_{A_n^c}\right] \notag
  \\
  & =  - \left( \frac{1/\beta_n - \E[1/\beta_n]}{\E[1/\beta_n]^2}  \right)  
         + R_1(1/\beta_n) \, \mathds{1}_{A_n} - \E\left[ R_1(1/\beta_n) \, \mathds{1}_{A_n} \right] \notag
  \\
  & \qquad  
         + \left( \beta_n - \frac{2}{\E[1/\beta_n]} + \frac{1/\beta_n}{\E[1/\beta_n]^2} \right) \, \mathds{1}_{A_n^c}
         - \E\left[ \left( \beta_n - \frac{2}{\E[1/\beta_n]} + \frac{1/\beta_n}{\E[1/\beta_n]^2} \right) \, \mathds{1}_{A_n^c}\right]\,. \notag
\end{align}
Hence, Slutzky's theorem yields
\begin{align*} 
  \lefteqn{\frac{\beta_n - \E[\beta_n]}{\sqrt{\var(\beta_n)}}
   = \frac{\beta_n - \E[\beta_n]}{\sqrt{\var(1/\beta_n)}} \, \frac{\sqrt{\var(1/\beta_n)}}{\sqrt{\var(\beta_n)}}} 
  \\
  & =  \left( 
       \underbrace{- \, \frac{1}{\E[1/\beta_n]^2}}_{\longrightarrow -1} \, 
       \underbrace{\frac{1/\beta_n - \E[1/\beta_n]}{\sqrt{\var(1/\beta_n)}}}_{\stackrel{d}{\longrightarrow} N(0,1) \;, \text{ by } \eqref{AN.Lemma.Eq:AN.Beta}} 
       + \underbrace{\frac{\sqrt{n} R_1(1/\beta_n)\, \mathds{1}_{A_n}}{\sqrt{n \, \var(1/\beta_n)}}}_{\stackrel{L^2}{\longrightarrow} 0 \;, \text{ by } \eqref{AN.Lemma:TaylorExp.UI.1}, \textrm{ L. \ref{AN.Lemma:AsympVar.1/beta}}}  
       - \underbrace{\frac{\sqrt{n} \, \E \left[R_1(1/\beta_n)\, \mathds{1}_{A_n}\right]}{\sqrt{n \, \var(1/\beta_n)}}}_{\longrightarrow 0 \;, \text{ by } \eqref{AN.Lemma:TaylorExp.UI.1}, \textrm{ L. \ref{AN.Lemma:AsympVar.1/beta}}} \right) \, 
       \underbrace{\frac{\sqrt{\var(1/\beta_n)}}{\sqrt{\var(\beta_n)}}}_{\longrightarrow 1 \;, \text{ by } \eqref{AN.Lemma:TaylorExp.UI.3}} \notag
  \\
  & \qquad + \left(
       \underbrace{\frac{\sqrt{n} \, \left( \beta_n - \frac{2}{\E[1/\beta_n]} + \frac{1/\beta_n}{\E[1/\beta_n]^2} \right) \, \mathds{1}_{A_n^c}}{\sqrt{n \, \var(1/\beta_n)}}}_{\stackrel{L^2}{\longrightarrow} 0 \;, \text{ by } \eqref{AN.Lemma:TaylorExp.UI.1}, \textrm{ L. \ref{AN.Lemma:AsympVar.1/beta}}}
       - \underbrace{\frac{\sqrt{n} \, \E\left[ \left( \beta_n - \frac{2}{\E[1/\beta_n]} + \frac{1/\beta_n}{\E[1/\beta_n]^2} \right) \, \mathds{1}_{A_n^c}\right]}{\sqrt{n \, \var(1/\beta_n)}}}_{\longrightarrow 0 \;, \text{ by } \eqref{AN.Lemma:TaylorExp.UI.1}, \textrm{ L. \ref{AN.Lemma:AsympVar.1/beta}}}
             \right) \, 
       \underbrace{\frac{\sqrt{\var(1/\beta_n)}}{\sqrt{\var(\beta_n)}}}_{\longrightarrow 1 \;, \text{ by } \eqref{AN.Lemma:TaylorExp.UI.3}}
  \\
  & \stackrel{d}{\longrightarrow} N(0,1)\,.
\end{align*}
This proves \eqref{AN.Lemma:TaylorExp.UI.5}.

Finally, we prove \eqref{AN.Lemma:TaylorExp.UI.6} by showing  
\begin{align*}
  \frac{\beta_n - \E[\beta_n]}{\sqrt{\var(1/\beta_n)}} - \left( - \frac{1/\beta_n - \E[1/\beta_n]}{\sqrt{\var(1/\beta_n)}} \right)
  \stackrel{P}{\longrightarrow} 0\,,
\end{align*}
which, in combination with Eq. \eqref{AN.Lemma:TaylorExp.UI.3}, yields the assertion.
Applying Eq. \eqref{AN.Lemma:TaylorExp.ID2} first gives
\begin{align*}
  & (\beta_n - \E[\beta_n]) + (1/\beta_n - \E[1/\beta_n])
  \\
  &   =  \left( 1/\beta_n - \E[1/\beta_n] \right) \, \left( 1 - \frac{1}{\E[1/\beta_n]^2} \right) 
         + R_1(1/\beta_n) \, \mathds{1}_{A_n} 
         - \E\left[ R_1(1/\beta_n) \, \mathds{1}_{A_n} \right] 
  \\
  & \qquad  
         + \left( \beta_n - \frac{2}{\E[1/\beta_n]} + \frac{1/\beta_n}{\E[1/\beta_n]^2} \right) \, \mathds{1}_{A_n^c}
         - \E\left[ \left( \beta_n - \frac{2}{\E[1/\beta_n]} + \frac{1/\beta_n}{\E[1/\beta_n]^2} \right) \, \mathds{1}_{A_n^c}\right]\,,
\end{align*}
Hence, Slutzky’s theorem implies  
\begin{align*}
  & \frac{\beta_n - \E[\beta_n]}{\sqrt{\var(1/\beta_n)}} - \left( - \frac{1/\beta_n - \E[1/\beta_n]}{\sqrt{\var(1/\beta_n)}} \right) 
  \\
  & =  \underbrace{\frac{1/\beta_n - \E[1/\beta_n]}{\sqrt{\var(1/\beta_n)}}}_{\stackrel{d}{\longrightarrow} N(0,1) \;, \text{ by } \eqref{AN.Lemma.Eq:AN.Beta}}  \, \left( \underbrace{1 - \frac{1}{\E[1/\beta_n]^2}}_{\longrightarrow 0} \right)    
       + \underbrace{\frac{\sqrt{n} \, R_1(1/\beta_n) \, \mathds{1}_{A_n}}{\sqrt{n \, \var(1/\beta_n)}}}_{\stackrel{L^2}{\longrightarrow} 0 \;, \text{ by } \eqref{AN.Lemma:TaylorExp.UI.1}, \textrm{ L. \ref{AN.Lemma:AsympVar.1/beta}}} 
       - \underbrace{\frac{\sqrt{n} \, \E \left[R_1(1/\beta_n)\, \mathds{1}_{A_n}\right]}{\sqrt{n \, \var(1/\beta_n)}}}_{\longrightarrow 0 \;, \text{ by } \eqref{AN.Lemma:TaylorExp.UI.1}, \textrm{ L. \ref{AN.Lemma:AsympVar.1/beta}}}
  \\
  & \qquad 
       \underbrace{\frac{\sqrt{n} \, \left( \beta_n - \frac{2}{\E[1/\beta_n]} + \frac{1/\beta_n}{\E[1/\beta_n]^2} \right) \, \mathds{1}_{A_n^c}}{\sqrt{n \, \var(1/\beta_n)}}}_{\stackrel{L^2}{\longrightarrow} 0 \;, \text{ by } \eqref{AN.Lemma:TaylorExp.UI.1}, \textrm{ L. \ref{AN.Lemma:AsympVar.1/beta}}}
       - \underbrace{\frac{\sqrt{n} \, \E\left[ \left( \beta_n - \frac{2}{\E[1/\beta_n]} + \frac{1/\beta_n}{\E[1/\beta_n]^2} \right) \, \mathds{1}_{A_n^c}\right]}{\sqrt{n \, \var(1/\beta_n)}}}_{\longrightarrow 0 \;, \text{ by } \eqref{AN.Lemma:TaylorExp.UI.1}, \textrm{ L. \ref{AN.Lemma:AsympVar.1/beta}}}
       \stackrel{P}{\longrightarrow} 0,
\end{align*}
This proves the result.
\end{proof}

\bigskip\noindent
{\bf Step 3: Final path to asymptotic normality of $\tfrac{T_n - \E[T_n]}{\sqrt{\var(T_n)}}$.}
Finally, a series of intermediate results are proven that draw a path from $\tfrac{\mu_n - \E[\mu_n]}{\sqrt{\var(\mu_n)}}$ to $\tfrac{T_n - \E[T_n]}{\sqrt{\var(T_n)}}$.

Since \(\kappa_n = \frac{q - \Lambda_n}{q - \alpha}\) is a linear function of $\Lambda_n$, the following statement is immediate from Lemma \ref{AN.Lemma:AsympVar}.

\begin{lemma}[Asymptotic variance III]~~\label{AN.Lemma:AsympVar.kappa}
Assume \(F_\ZZ\) to be fixed and continuous. Then
\begin{enumerate}[(i)]
\item \label{AN.Lemma:AsympVar.kappa1} $\limsup_{n \to \infty} n \, \var(\kappa_n) < \infty$.

\item \label{AN.Lemma:AsympVar.kappa2} If \(\YY\) is not perfectly dependent on \(\XX\,,\) then $\liminf_{n \to \infty} n \, \var(\kappa_n) > 0$.
\end{enumerate}
\end{lemma}

Even when interacting with \(\kappa_n\), the terms \(\beta_n\) and \(1/\beta_n\) can be replaced by each other as follows.
\begin{lemma} \label{AN.Lemma:TaylorExp:Cor}
Under the assumptions of Proposition \ref{AN.Lemma:TaylorExp}, if 
\begin{align} \label{AN.Cond.Kappa.UI}
    \sup_{n \in \mathbb{N}} \, \E \left[ \left| \frac{\kappa_n - \E[\kappa_n]}{\sqrt{\var(\kappa_n)}} \right|^{2+\delta} \right] < \infty
\end{align}
for some \(\delta > 0\) and \(\YY\) is not perfectly dependent on \(\XX\), then
\begin{enumerate}[(i)]
\item \label{AN.Eq.Var.BetaKappa}
$\lim_{n \to \infty} \left| \frac{\var(\beta_n)}{\var(\kappa_n)} - \frac{\var(1/\beta_n)}{\var(\kappa_n)} \right| = 0$.
\item \label{AN.Eq.Var.BetaKappab}
$\lim_{n \to \infty} \left| \frac{\var(\kappa_n)}{\var(\beta_n)} -  \frac{\var(\kappa_n)}{\var(1/\beta_n)} \right| = 0$.
\item \label{AN.Eq.Cov.BetaKappa}
$\lim_{n \to \infty} n \left| \cov(\kappa_n,\beta_n) - \cov(\kappa_n,-1/\beta_n) \right| = 0$.
\item \label{AN.Eq.Cov.BetaKappa2}
$\lim_{n \to \infty} \left| \cov \left( \beta_n \, \frac{\kappa_n - \E[\kappa_n]}{\sqrt{\var(\kappa_n)}}, \frac{\beta_n - \E[\beta_n]}{\sqrt{\var(\beta_n)}} \right)
    - \Cor \left(\kappa_n, \beta_n \right) \right| = 0$.
\item \label{AN.Eq.Cov.BetaKappa3}
$\lim_{n \to \infty} \left| \cov \left( \frac{\kappa_n}{\E[\kappa_n]} \, \frac{\beta_n - \E[\beta_n]}{\sqrt{\var(\beta_n)}}, 
                          \frac{\kappa_n - \E[\kappa_n]}{\sqrt{\var(\kappa_n)}} \right) \, 
               - \Cor \left( \kappa_n, \beta_n \right) \right| = 0$.
\end{enumerate}
\end{lemma}

\begin{proof}
\eqref{AN.Eq.Var.BetaKappa} is immediate from Lemma \ref{AN.Lemma:AsympVar.kappa} and Proposition \ref{AN.Lemma:TaylorExp}. 
Statement \eqref{AN.Eq.Var.BetaKappab} follows from
\begin{align*}
    \left| \frac{\var(\kappa_n)}{\var(\beta_n)} -  \frac{\var(\kappa_n)}{\var(1/\beta_n)} \right| 
    = \frac{n^2\var(\kappa_n)^2}{n\var(\beta_n) \, n \var(1/\beta_n)} \left| \frac{\var(\beta_n)}{\var(\kappa_n)} - \frac{\var(1/\beta_n)}{\var(\kappa_n)} \right| \to 0\,,
\end{align*}
using \eqref{AN.Eq.Var.BetaKappa} as well as Lemma \ref{AN.Lemma:AsympVar.kappa}\eqref{AN.Lemma:AsympVar.kappa1}, Lemma \ref{AN.Lemma:AsympVar.1/beta}\eqref{AN.Lemma:AsympVar.1/betab}, and Proposition \ref{AN.Lemma:TaylorExp}\eqref{AN.Lemma:TaylorExp.UI.4}.

Now, applying Eq. \eqref{AN.Eq:TaylorExp2} together with Cauchy-Schwarz inequality yields
\begin{align*}
  \lefteqn{n \, \left| \cov \left(\kappa_n,\beta_n\right) - \cov\left(\kappa_n,-1/\beta_n\right) \right|
      =  n \, \left| \cov \left(\kappa_n, \beta_n + 1/\beta_n \right) \right|}
  \\
  &   =  n \, \left| \cov \left(\kappa_n, 
              \left( \frac{2}{\E[1/\beta_n]} - \frac{1/\beta_n}{\E[1/\beta_n]^2}  \right)  
              + R_1(1/\beta_n) \, \mathds{1}_{A_n}
              + \left( \beta_n - \frac{2}{\E[1/\beta_n]} + \frac{1/\beta_n}{\E[1/\beta_n]^2} \right) \, \mathds{1}_{A_n^c} + 1/\beta_n
              \right) \right|
  \\
  &   =  n \, \biggl| \cov \left(\kappa_n, \frac{1}{\beta_n} \, \left(1 - \frac{1}{\E[1/\beta_n]^2}\right) \right) 
         + \cov \left(\kappa_n, R_1(1/\beta_n) \, \mathds{1}_{A_n} \right) 
  \\
  & \qquad 
         + \cov \left(\kappa_n,\left( \beta_n - \frac{2}{\E[1/\beta_n]} + \frac{1/\beta_n}{\E[1/\beta_n]^2} \right) \, \mathds{1}_{A_n^c} \right) \biggr|
  \\
  & \leq \sqrt{n \, \var \left(\kappa_n \right)} \sqrt{n \, \var \left(1/\beta_n\right)} \, \left|\underbrace{1 - \frac{1}{\E[1/\beta_n]^2}}_{\longrightarrow 0}\right|  
         + \sqrt{n \, \var \left(\kappa_n \right)}  \underbrace{\sqrt{n \, \var \left(R_1(1/\beta_n) \, \mathds{1}_{A_n}\right)}}_{\longrightarrow 0\; \textrm{ by } \eqref{AN.Lemma:TaylorExp.UI.2}}
  \\
  & \qquad + \sqrt{n \, \var \left(\kappa_n \right)}  \underbrace{\sqrt{n \, \var \left(\left( \beta_n - \frac{2}{\E[1/\beta_n]} + \frac{1/\beta_n}{\E[1/\beta_n]^2} \right) \, \mathds{1}_{A_n^c}\right)}}_{\longrightarrow 0\; \textrm{ by } \eqref{AN.Lemma:TaylorExp.UI.2}} \longrightarrow 0\,,
\end{align*}
where we use \(\limsup_{n\to \infty} n \Var(\kappa_n) < \infty\) and \(\limsup_{n\to \infty} n \Var(1/\beta_n) < \infty\) by Lemma \ref{AN.Lemma:AsympVar.kappa} and Lemma \ref{AN.Lemma:AsympVar.1/beta}.
This proves \eqref{AN.Eq.Cov.BetaKappa}.

A further use of the Cauchy-Schwarz inequality gives
\begin{align*}
  \lefteqn{\left| \cov \left( \beta_n \, \frac{\kappa_n - \E[\kappa_n]}{\sqrt{\var(\kappa_n)}}, \frac{\beta_n - \E[\beta_n]}{\sqrt{\var(\beta_n)}} \right)
    - \cov \left( \frac{\kappa_n - \E[\kappa_n]}{\sqrt{\var(\kappa_n)}}, \frac{\beta_n - \E[\beta_n]}{\sqrt{\var(\beta_n)}} \right) \right|}
  \\
  &   =  \left| \cov \left( (\beta_n-1) \, \frac{\kappa_n - \E[\kappa_n]}{\sqrt{\var(\kappa_n)}}, \frac{\beta_n - \E[\beta_n]}{\sqrt{\var(\beta_n)}} \right) \right|
  \\
  & \leq \sqrt{\var\left( \underbrace{(\beta_n-1)}_{\longrightarrow 0} \, \underbrace{\frac{\kappa_n - \E[\kappa_n]}{\sqrt{\var(\kappa_n)}}}_{\stackrel{d}{\longrightarrow} N(0,1) \; \textrm{ by } \eqref{AN.Lemma.Eq:AN.Kappa}}  \right)} \cdot \underbrace{\sqrt{\var\left( \frac{\beta_n - \E[\beta_n]}{\sqrt{\var(\beta_n)}} \right)}}_{= 1} \longrightarrow 0\,,
\end{align*}
where convergence follows from uniform integrability due to \eqref{AN.Cond.Kappa.UI} in combination with Billingsley [1999, Theorem 3.5].
This proves \eqref{AN.Eq.Cov.BetaKappa2}. Statement \eqref{AN.Eq.Cov.BetaKappa3} follows by a similar reasoning using uniform integrability as a consequence of \eqref{AN.Lemma:TaylorExp.UI}.
\end{proof}

\begin{lemma}\label{AN.Lemma:LimitsKappaBeta2}
Under the assumptions of Lemma \ref{AN.Lemma:TaylorExp:Cor}, 
\begin{enumerate}[(i)] 
\item \label{AN.Lemma:LimitsKappaBeta2.Eq1}
\(\liminf_{n \to \infty} \E[\kappa_n] > 0\).
\item \label{AN.Lemma:LimitsKappaBeta2.Eq2}
$\lim_{n \to \infty} \left| \frac{\var(\beta_n \, \kappa_n)}{\var(\kappa_n)} 
    - \frac{\var(\mu_n)}{\var(\Lambda_n)} \right|
    = \lim_{n \to \infty} \left| \frac{\var(\beta_n \, \kappa_n)}{\var(\kappa_n)} 
    - \frac{\var(\kappa_n - \kappa \cdot 1/\beta_n)}{\var(\kappa_n)} \right|
    = 0$.
\item \label{AN.Lemma:LimitsKappaBeta2.Eq3}
$\lim_{n \to \infty} \left| \frac{\var(\beta_n \, \kappa_n)}{\var(\beta_n) \, \E[\kappa_n]^2} 
    - \frac{\var(\mu_n)}{\var(\kappa \, \alpha_n)} \right|
    = \lim_{n \to \infty} \left| \frac{\var(\beta_n \, \kappa_n)}{\var(\beta_n) \, \E[\kappa_n]^2} 
    - \frac{\var(\kappa_n - \kappa \cdot 1/\beta_n)}{\var(\kappa \cdot 1/\beta_n)} \right|
    = 0$.
\end{enumerate}
\end{lemma}

\begin{proof}
It is immediate from Fatou's Lemma that 
\(\liminf_{n \to \infty} \E[\kappa_n] \geq \E [\liminf_{n \to \infty} \kappa_n] = \kappa > 0\).
This proves \eqref{AN.Lemma:LimitsKappaBeta2.Eq1}.

To show \eqref{AN.Lemma:LimitsKappaBeta2.Eq2}, straightforward calculation first yields
\begin{align} \label{AN.Lemma:LimitsKappaBeta2.Eq10}
  & \frac{\var(\beta_n \, \kappa_n)}{\var(\kappa_n)} 
    =  \var \left( \beta_n \, \frac{\kappa_n - \E[\kappa_n]}{\sqrt{\var(\kappa_n)}} + \beta_n \, \frac{\E[\kappa_n]}{\sqrt{\var(\kappa_n)}} \right)
  \\
  & =  \var \left( \beta_n \, \frac{\kappa_n - \E[\kappa_n]}{\sqrt{\var(\kappa_n)}} 
                   + \frac{\beta_n - \E[\beta_n]}{\sqrt{\var(\beta_n)}}  \, \frac{\E[\kappa_n] \, \sqrt{\var(\beta_n)}}{\sqrt{\var(\kappa_n)}} 
                   + \frac{\E[\beta_n] \, \E[\kappa_n]}{\sqrt{\var(\kappa_n)}} \right) \notag
  \\
  & =  \var \left( \beta_n \, \frac{\kappa_n - \E[\kappa_n]}{\sqrt{\var(\kappa_n)}} \right)
       + \var \left( \frac{\beta_n - \E[\beta_n]}{\sqrt{\var(\beta_n)}}  \, \frac{\E[\kappa_n] \, \sqrt{\var(\beta_n)}}{\sqrt{\var(\kappa_n)}} \right) \notag
  \\
  & \qquad  + 2 \, \cov \left( \beta_n \, \frac{\kappa_n - \E[\kappa_n]}{\sqrt{\var(\kappa_n)}}, 
                          \frac{\beta_n - \E[\beta_n]}{\sqrt{\var(\beta_n)}}  \, \frac{\E[\kappa_n] \, \sqrt{\var(\beta_n)}}{\sqrt{\var(\kappa_n)}} \right) \notag
  \\
  & =  \var \left( \beta_n \, \frac{\kappa_n - \E[\kappa_n]}{\sqrt{\var(\kappa_n)}} \right)
          + \left(\frac{\E[\kappa_n] \, \sqrt{\var(\beta_n)}}{\sqrt{\var(\kappa_n)}}\right)^2 \notag
  \\
  & \qquad  + 2 \, \cov \left( \beta_n \, \frac{\kappa_n - \E[\kappa_n]}{\sqrt{\var(\kappa_n)}}, 
                          \frac{\beta_n - \E[\beta_n]}{\sqrt{\var(\beta_n)}} \right) \, \frac{\E[\kappa_n] \, \sqrt{\var(\beta_n)}}{\sqrt{\var(\kappa_n)}}\,. \notag
\end{align}
By Eq. \eqref{AN.Lemma:LimitsKappaBeta2.Eq10} we then obtain
\begin{align*}
  & \left| \frac{\var(\beta_n \, \kappa_n)}{\var(\kappa_n)} 
    - \frac{\var(\kappa_n - \kappa \cdot 1/\beta_n)}{\var(\kappa_n)} \right|
  \\
  & =  \Biggl| \Biggl( \var \left( \beta_n \, \frac{\kappa_n - \E[\kappa_n]}{\sqrt{\var(\kappa_n)}} \right)
             + \frac{\E[\kappa_n]^2 \, \var(\beta_n)}{\var(\kappa_n)}
             + 2 \, \cov \left( \beta_n \, \frac{\kappa_n - \E[\kappa_n]}{\sqrt{\var(\kappa_n)}}, 
                          \frac{\beta_n - \E[\beta_n]}{\sqrt{\var(\beta_n)}} \right) \, \frac{\E[\kappa_n] \, \sqrt{\var(\beta_n)}}{\sqrt{\var(\kappa_n)}}
                          \Biggr)
  \\
  & \qquad  - \left( 1 
               + \frac{\kappa^2 \, \var(1/\beta_n)}{\var(\kappa_n)} 
               + 2 \, \Cor \left( \kappa_n, -1/\beta_n \right) \, \frac{\kappa \, \sqrt{\var(1/\beta_n)}}{\sqrt{\var(\kappa_n)}} \right) \Biggr|    
  \\
  & \leq \left| \var \left( \underbrace{\beta_n}_{\longrightarrow 1} \, \underbrace{\frac{\kappa_n - \E[\kappa_n]}{\sqrt{\var(\kappa_n)}}}_{\stackrel{d}{\longrightarrow} N(0,1) \; \textrm{ by }\eqref{AN.Lemma.Eq:AN.Kappa}} \right) - 1 \right|
         +  \underbrace{\left| \E[\kappa_n]^2 - \kappa^2 \right|}_{\longrightarrow 0} \, \frac{n \, \var(\beta_n)}{n \, \var(\kappa_n)}
         + \kappa^2 \, \underbrace{\left| \frac{\var(\beta_n)}{\var(\kappa_n)} -  \frac{\var(1/\beta_n)}{\var(\kappa_n)} \right|}_{\longrightarrow 0\,, \textrm{ by Lemma \ref{AN.Lemma:TaylorExp:Cor}\eqref{AN.Eq.Var.BetaKappa}}}
  \\
  & \qquad   + 2 \, \underbrace{\left| \cov \left( \beta_n \, \frac{\kappa_n - \E[\kappa_n]}{\sqrt{\var(\kappa_n)}}, \frac{\beta_n - \E[\beta_n]}{\sqrt{\var(\beta_n)}} \right) \, 
               - \Cor \left( \kappa_n, \beta_n \right) \right|}_{\longrightarrow 0\,, \textrm{ by Lemma \ref{AN.Lemma:TaylorExp:Cor}}} \, \underbrace{\E[\kappa_n]}_{\longrightarrow \kappa} \frac{\sqrt{n \, \var(\beta_n)}}{\sqrt{n \, \var(\kappa_n)}}
  \\
  & \qquad   + 2 \, \left| \Cor \left( \kappa_n, \beta_n \right) \right| \, \underbrace{\left| \E[\kappa_n] 
               - \kappa \right|}_{\longrightarrow 0} \, \frac{\sqrt{n \, \var(\beta_n)}}{\sqrt{n \, \var(\kappa_n)}}
             + 2 \, \kappa \,
               \underbrace{\left| \Cor \left( \kappa_n, \beta_n \right) - \Cor \left( \kappa_n, -1/\beta_n \right) \right|}_{\longrightarrow 0; \textrm{ by } \eqref{AN.Lemma:LimitsKappaBeta2.Eq4}}
               \, \frac{\sqrt{n \, \var(\beta_n)}}{\sqrt{n \, \var(\kappa_n)}}
  \\
  & \qquad   + 2 \, \kappa \, \left| \Cor \left( \kappa_n, -1/\beta_n \right) \right|
               \underbrace{\left| \frac{\sqrt{\var(\beta_n)}}{\sqrt{\var(\kappa_n)}} - \frac{\sqrt{\var(1/\beta_n)}}{\sqrt{\var(\kappa_n)}} \right|}_{\longrightarrow 0\,,  \textrm{ by Lemma \ref{AN.Lemma:TaylorExp:Cor} \eqref{AN.Eq.Var.BetaKappa}}}
            \qquad \longrightarrow 0\,,
\end{align*}
where we use \(\liminf_{n\to \infty} n \Var(\kappa_n) > 0\) and \(\limsup_{n\to \infty} n \Var(\beta_n) < \infty\) by Lemma \ref{AN.Lemma:AsympVar.kappa} and Proposition \ref{AN.Lemma:TaylorExp}, 
and where convergence of the first term follows from Slutzky's theorem and 
uniform integrability due to \eqref{AN.Cond.Kappa.UI} in combination with \citep[Theorem 3.5]{Billingsley-1999}.
For the convergence we have also used that
\begin{align} \label{AN.Lemma:LimitsKappaBeta2.Eq4}
  & \left| \Cor \left( \kappa_n, \beta_n \right) - \Cor \left( \kappa_n, -1/\beta_n \right) \right|
    =   \left| \frac{\cov \left( \kappa_n, \beta_n \right)}{\sqrt{\var(\kappa_n)}\sqrt{\var(\beta_n)}} - \Cor \left( \kappa_n, -1/\beta_n \right) \right| 
  \\
  &  \leq \underbrace{\frac{n \, \left| \cov \left( \kappa_n, \beta_n \right) - \cov \left( \kappa_n, -1/\beta_n \right) \right|}{\sqrt{n \, \var(\kappa_n)}\sqrt{n \, \var(\beta_n)}}}_{\longrightarrow 0\,,  \textrm{ by Proposition \ref{AN.Lemma:TaylorExp} and Lemmas  \ref{AN.Lemma:AsympVar.kappa},\ref{AN.Lemma:TaylorExp:Cor}}}
          + \left| \Cor \left( \kappa_n, -1/\beta_n \right) \right| \, 
          \underbrace{\left| \frac{\sqrt{n \, \var(1/\beta_n)}}{\sqrt{n \, \var(\beta_n)}} - 1 \right|}_{\longrightarrow 0\,, \textrm{ by Proposition \ref{AN.Lemma:TaylorExp}}}
  \quad \longrightarrow 0\,. \notag
\end{align}

To prove the remaining assertion \eqref{AN.Lemma:LimitsKappaBeta2.Eq3}, 
first define
\begin{align*}
    I_{1,n} & := \var(\kappa_n \, R_1(1/\beta_n) \, \mathds{1}_{A_n})
    \\
    I_{2,n} & := \var \left( \kappa_n \, \E\left[ R_1(1/\beta_n) \, \mathds{1}_{A_n} \right] \right)
    \\
    I_{3,n} & := \var \left(\kappa_n \, \left( \beta_n - \frac{2}{\E[1/\beta_n]} + \frac{1/\beta_n}{\E[1/\beta_n]^2} \right) \, \mathds{1}_{A_n^c} \right)
    \\
    I_{4,n} & := \var \left( \kappa_n \,  \E\left[ \left( \beta_n - \frac{2}{\E[1/\beta_n]} + \frac{1/\beta_n}{\E[1/\beta_n]^2} \right) \, \mathds{1}_{A_n^c}\right] \right)
\end{align*}
Since \(|\kappa_n| = \frac{|q-\Lambda_n|}{q-\alpha} \leq \frac{2q}{q-\alpha} \leq C < \infty\) due to Lemma \ref{AN.Lemma:Finite.Beta}, it is straightforward to verify that 
$\lim_{n \to \infty} n \, I_1 = 0$, 
$\lim_{n \to \infty} n \, I_2 = 0$, 
$\lim_{n \to \infty} n \, I_3 = 0$ and $\lim_{n \to \infty} n \, I_4 = 0$,
where convergence follows from Eq. \eqref{AN.Lemma:TaylorExp.UI.1}. 
From Eq. \eqref{AN.Lemma:TaylorExp.ID2} and Cauchy-Schwarz inequality we then obtain
\begin{align*}
    \lefteqn{\left| \var \left( \frac{\kappa_n}{\E[\kappa_n]} \, \frac{\beta_n - \E[\beta_n]}{\sqrt{\var(\beta_n)}} \right) - 1 \right|}
    \\
    &   =  \Biggl| \var \Biggl( \frac{\kappa_n}{\E[\kappa_n]} \,
                               \frac{\sqrt{\var(1/\beta_n)}}{\sqrt{\var(\beta_n)}} \,
                               \frac{-1}{\E[1/\beta_n]^2} \;
                               \frac{1/\beta_n - \E[1/\beta_n]}{\sqrt{\var(1/\beta_n)}}
           + \frac{\kappa_n}{\E[\kappa_n]} \, 
             \frac{R_1(1/\beta_n) \, \mathds{1}_{A_n} - \E\left[ R_1(1/\beta_n) \, \mathds{1}_{A_n} \right]}{\sqrt{\var(\beta_n)}} 
    \\
    & \quad
           + \frac{\kappa_n}{\E[\kappa_n]} \, 
             \frac{\left( \beta_n - \frac{2}{\E[1/\beta_n]} + \frac{1/\beta_n}{\E[1/\beta_n]^2} \right) \, \mathds{1}_{A_n^c}
                 - \E\left[ \left( \beta_n - \frac{2}{\E[1/\beta_n]} + \frac{1/\beta_n}{\E[1/\beta_n]^2} \right) \, \mathds{1}_{A_n^c}\right]}{\sqrt{\var(\beta_n)}} \Biggr) - 1 \Biggr|
    \\
    & \leq \left| \underbrace{\frac{n \, \var(1/\beta_n)}{n \, \var(\beta_n)}}_{\longrightarrow 1 \,, \textrm{ by Proposition \ref{AN.Lemma:TaylorExp}}} \, 
                  \underbrace{\frac{1}{\E[1/\beta_n]^4 \, \E[\kappa_n]^2}}_{\longrightarrow 1/\kappa^2} \,
                  \var \left( \underbrace{\kappa_n}_{\longrightarrow \kappa} \, 
                              \underbrace{\frac{1/\beta_n - \E[1/\beta_n]}{\sqrt{\var(1/\beta_n)}}}_{\stackrel{d}{\longrightarrow} N(0,1)\,, \; \textrm{ by }\eqref{AN.Lemma.Eq:AN.Beta}}
           \right) - 1 \right|
           + \frac{\overbrace{\sum_{m=1}^{4} n \, I_{m,n}}^{\longrightarrow 0}}{\E[\kappa_n]^2 \, n \, \var(\beta_n)}
    \\
    & \quad + 2 \, \frac{\sqrt{\var \left( \kappa_n \, \frac{1/\beta_n - \E[1/\beta_n]}{\sqrt{\var(1/\beta_n)}} \right)} \overbrace{\sum_{m=1}^{4} \sqrt{n \, I_{m,n}}}^{\longrightarrow 0}}
                    {\E[1/\beta_n]^2 \, \E[\kappa_n]^2 \, \sqrt{n \, \var(\beta_n)}}
                    \; \underbrace{\frac{\sqrt{\var(1/\beta_n)}}{\sqrt{\var(\beta_n)}}}_{\longrightarrow 1\,, \textrm{ by Proposition \ref{AN.Lemma:TaylorExp}}}
            + 2 \, \frac{\overbrace{\sum_{\ell=1}^{4} \sum_{m=\ell+1}^{4} \sqrt{n \, I_\ell} \, \sqrt{n \, I_{m,n}}}^{\longrightarrow 0}}
                    {\E[\kappa_n]^2 \, n \, \var(\beta_n)}
      \longrightarrow 0\,,
\end{align*}
where we use \(\liminf_{n \to \infty} \E[\kappa_n] > 0\) by \eqref{AN.Lemma:LimitsKappaBeta2.Eq1} and \(\liminf_{n\to \infty} n \Var(\beta_n) > 0\) by Proposition \ref{AN.Lemma:TaylorExp}, 
and where convergence follows from Slutzky's theorem and uniform integrability due to \eqref{AN.Lemma:TaylorExp.UI} in combination with \citep[Theorem 3.5]{Billingsley-1999}.

In a similar fashion as in the proof of \eqref{AN.Lemma:LimitsKappaBeta2.Eq2}, we finally obtain 
\begin{align*}
  & \left| \frac{\var(\beta_n \, \kappa_n)}{\var(\beta_n) \, \E[\kappa_n]^2} 
    - \frac{\var(\kappa_n - \kappa \cdot 1/\beta_n)}{\var(\kappa \cdot 1/\beta_n)} \right|
  \\
  & = \left| \var \left( \frac{\kappa_n}{\E[\kappa_n]} \, \frac{\beta_n - \E[\beta_n]}{\sqrt{\var(\beta_n)}} + \frac{\kappa_n}{\E[\kappa_n]} \, \frac{\E[\beta_n]}{\sqrt{\var(\beta_n)}} \right)
    - \frac{\var(\kappa_n) + \var(\kappa \cdot 1/\beta_n) - 2 \, \cov(\kappa_n, \kappa \cdot 1/\beta_n)}{\var(\kappa \cdot 1/\beta_n)} \right|
  \\
  & =  \Biggl| \Biggl( \var \left( \frac{\kappa_n}{\E[\kappa_n]} \, \frac{\beta_n - \E[\beta_n]}{\sqrt{\var(\beta_n)}} \right)
             + \frac{\E[\beta_n]^2}{\E[\kappa_n]^2} \, \frac{\var(\kappa_n)}{\var(\beta_n)}
  \\
  & \qquad   + 2 \, \cov \left( \frac{\kappa_n}{\E[\kappa_n]} \, \frac{\beta_n - \E[\beta_n]}{\sqrt{\var(\beta_n)}}, 
                          \frac{\kappa_n - \E[\kappa_n]}{\sqrt{\var(\kappa_n)}} \right) \, \frac{\E[\beta_n]}{\E[\kappa_n]} \, \frac{\sqrt{\var(\kappa_n)}}{\sqrt{\var(\beta_n)}} \Biggr)
  \\
  & \qquad  - \left( 1 
            + \frac{\var(\kappa_n)}{\kappa^2 \, \var(1/\beta_n)} 
            + 2 \, \Cor \left( \kappa_n, -1/\beta_n \right) \, \frac{\sqrt{\var(\kappa_n)}}{ \kappa \,\sqrt{\var(1/\beta_n)}} \right) \Biggr|    
  \\
  & \leq \underbrace{\left| \var \left( \frac{\kappa_n}{\E[\kappa_n]} \, \frac{\beta_n - \E[\beta_n]}{\sqrt{\var(\beta_n)}} \right) - 1 \right|}_{\longrightarrow 0}
         + \underbrace{\left| \frac{\E[\beta_n]^2}{\E[\kappa_n]^2} - \frac{1}{\kappa^2} \right|}_{\longrightarrow 0} \, \frac{n \, \var(\kappa_n)}{n \, \var(\beta_n)}
         + \frac{1}{\kappa^2} \, \underbrace{\left| \frac{\var(\kappa_n)}{\var(\beta_n)} - \frac{\var(\kappa_n)}{\var(1/\beta_n)} \right|}_{\longrightarrow 0\,,  \textrm{ by Lemma \ref{AN.Lemma:TaylorExp:Cor}}}
  \\
  & \qquad   + 2 \, \underbrace{\left| \cov \left( \frac{\kappa_n}{\E[\kappa_n]} \, \frac{\beta_n - \E[\beta_n]}{\sqrt{\var(\beta_n)}}, 
                          \frac{\kappa_n - \E[\kappa_n]}{\sqrt{\var(\kappa_n)}} \right) \, 
               - \Cor \left( \kappa_n, \beta_n \right) \right|}_{\longrightarrow 0\,,  \textrm{ by Lemma \ref{AN.Lemma:TaylorExp:Cor}}}
               \, \frac{\E[\beta_n]}{\E[\kappa_n]} \, \frac{\sqrt{n \, \var(\kappa_n)}}{\sqrt{n \, \var(\beta_n)}}
  \\
  & \qquad   + 2 \, \underbrace{\left| \Cor \left( \kappa_n, \beta_n \right) \, 
               - \Cor \left( \kappa_n, -1/\beta_n \right) \right|}_{\longrightarrow 0\,, \textrm{ by } \eqref{AN.Lemma:LimitsKappaBeta2.Eq4}} 
               \, \frac{\E[\beta_n]}{\E[\kappa_n]} \, \frac{\sqrt{\var(\kappa_n)}}{\sqrt{\var(\beta_n)}}
  \\
  & \qquad   + 2 \, \left| \Cor \left( \kappa_n, -1/\beta_n \right) \right|
                 \, \underbrace{\left| \frac{\E[\beta_n]}{\E[\kappa_n]} \, \frac{\sqrt{\var(\kappa_n)}}{\sqrt{\var(\beta_n)}} - \frac{\sqrt{\var(\kappa_n)}}{ \kappa \,\sqrt{\var(1/\beta_n)}} \right|}_{\longrightarrow 0\,,  \textrm{ by Lemma \ref{AN.Lemma:TaylorExp:Cor}}}
  \qquad \longrightarrow 0\,,
\end{align*}
where we use \(\liminf_{n \to \infty} \E[\kappa_n] > 0\) by \eqref{AN.Lemma:LimitsKappaBeta2.Eq1}, \(\limsup_{n\to \infty} n \Var(\kappa_n) < \infty\) and \(\liminf_{n\to \infty} n \Var(\beta_n) > 0\) by Lemma \ref{AN.Lemma:AsympVar.kappa} and Proposition \ref{AN.Lemma:TaylorExp}.
\end{proof}

\begin{lemma}~~\label{AN.Lemma:AsympVar.KappaBeta}
Under the assumptions of Lemma \ref{AN.Lemma:TaylorExp:Cor} and if, additionally, \linebreak $\limsup_{n \to \infty} \Cor(\Lambda_n,\alpha_n) < 1$, then
\begin{enumerate}[(i)]
\item \label{AN.Lemma:AsympVar.KappaBeta.Eq1} 
$\liminf_{n \to \infty} n \, \var(T_n) 
 = \liminf_{n \to \infty} n \, \var(\beta_n \, \kappa_n) > 0$.

\item \label{AN.Lemma:AsympVar.KappaBeta.Eq2}
$\lim_{n \to \infty} \left| \frac{\var(\kappa_n)}{\var(\beta_n \, \kappa_n)} 
    - \frac{\var(\Lambda_n)}{\var(\mu_n)} \right|
    = \lim_{n \to \infty} \left| \frac{\var(\kappa_n)}{\var(\beta_n \, \kappa_n)} 
    - \frac{\var(\kappa_n)}{\var(\kappa_n - \kappa \cdot 1/\beta_n)} \right|
    = 0$.
\item \label{AN.Lemma:AsympVar.KappaBeta.Eq3}
$\lim_{n \to \infty} \left| \frac{\var(\beta_n) \, \E[\kappa_n]^2}{\var(\beta_n \, \kappa_n)} 
    - \frac{\var(\kappa \, \alpha_n)}{\var(\mu_n)} \right|
    = \lim_{n \to \infty} \left| \frac{\var(\beta_n) \, \E[\kappa_n]^2}{\var(\beta_n \, \kappa_n)}
    - \frac{\var(\kappa \cdot 1/\beta_n)}{\var(\kappa_n - \kappa \cdot 1/\beta_n)} \right|
    = 0$.
\end{enumerate}
\end{lemma}
\begin{proof}
Assertion \eqref{AN.Lemma:AsympVar.KappaBeta.Eq1} is immediate from 
$\liminf_{n \to \infty} n \, \var(\mu_n) > 0$ due to Lemma \ref{AN.Lemma:AsympVar}\eqref{Lemma:AsympVar2} and 
\begin{align*}
    n \, \var(\beta_n \, \kappa_n)
    & = n \, \var(\kappa_n - \kappa \cdot 1/\beta_n) + n \, \left(\var(\beta_n \, \kappa_n) - \var(\kappa_n - \kappa \cdot 1/\beta_n) \right)
    \\
    & = \frac{n \, \var(\mu_n)}{(q-\alpha)^2} + n \, \left(\var(\beta_n \, \kappa_n) - \var(\kappa_n - \kappa \cdot 1/\beta_n) \right)
    \\
    & = \frac{n \, \var(\mu_n)}{(q-\alpha)^2} + \frac{\var(\beta_n \, \kappa_n) - \var(\kappa_n - \kappa \cdot 1/\beta_n)}{\var(\kappa_n)} \; (n \, \var(\kappa_n))
\end{align*}
in combination with Lemma \ref{AN.Lemma:LimitsKappaBeta2} \eqref{AN.Lemma:LimitsKappaBeta2.Eq2} and Lemma \ref{AN.Lemma:AsympVar.kappa} \eqref{AN.Lemma:AsympVar.kappa1}.

Statement \eqref{AN.Lemma:AsympVar.KappaBeta.Eq2} now follows from 
\begin{align*}
  \lefteqn{\left| \frac{\var(\kappa_n)}{\var(\beta_n \, \kappa_n)} 
    - \frac{\var(\kappa_n)}{\var(\kappa_n - \kappa \cdot 1/\beta_n)} \right|}
  \\
  & =  \frac{n^2 \, \var(\kappa_n)^2}{n \, \var(\beta_n \, \kappa_n) \, n \, \var(\kappa_n - \kappa \cdot 1/\beta_n)} \; 
       \left| \frac{\var(\beta_n \, \kappa_n)}{\var(\kappa_n)} 
    - \frac{\var(\kappa_n - \kappa \cdot 1/\beta_n)}{\var(\kappa_n)} \right|
  \\
  & =  \frac{n \, \var(\kappa_n) \, n \, \var(\Lambda_n)}{n \, \var(\beta_n \, \kappa_n) \, n \, \var(\mu_n)} \; 
       \left| \frac{\var(\beta_n \, \kappa_n)}{\var(\kappa_n)} 
    - \frac{\var(\kappa_n - \kappa \cdot 1/\beta_n)}{\var(\kappa_n)} \right|
    \longrightarrow 0\,,
\end{align*}
using \eqref{AN.Lemma:AsympVar.KappaBeta.Eq1} as well as Lemma \ref{AN.Lemma:AsympVar}, Lemma \ref{AN.Lemma:AsympVar.kappa}\eqref{AN.Lemma:AsympVar.kappa1}, and Lemma \ref{AN.Lemma:LimitsKappaBeta2}\eqref{AN.Lemma:LimitsKappaBeta2.Eq2}.
Statement \eqref{AN.Lemma:AsympVar.KappaBeta.Eq3} follows by a similar argument. 
\end{proof}

We are now in the position to prove Theorem \ref{AN.Thm:AN} and hence asymptotic normality in \eqref{AN.Aim}.
We use the approach presented at the beginning of this section and show how Eq. \eqref{AN.Rep:Mu.2} can be converted into Eq. \eqref{AN.Rep:T3.2} using the results presented above.

\begin{proof}[Proof of Theorem \ref{AN.Thm:AN}:]
Asymptotic normality of \((\mu_n - \E[\mu_n])/\sqrt{\var(\mu_n)}\) due to \eqref{AN.Lemma.Eq:AN.Mu} implies that the right hand side of Eq. \eqref{AN.Rep:Mu.2}, i.e., 
\begin{align*}
  - \, \frac{\kappa_n - \E\left[\kappa_n\right]}{\sqrt{\var\left(\kappa_n\right)}} \, 
           \frac{\sqrt{\var\left(\kappa_n\right)}}{\sqrt{\var\left(\kappa_n - \kappa \cdot 1/\beta_n\right)}}
      + \, \frac{1/\beta_n - \E[1/\beta_n]}{\sqrt{\var(1/\beta_n)}}  \, \frac{\sqrt{\var\left(1/\beta_n\right)}}{\sqrt{\var\left(\kappa_n - \kappa \cdot 1/\beta_n\right)}} \, \kappa \,,
\end{align*}
weakly converges to a standard normal distribution. Due to Lemma \ref{AN.Lemma:AsympVar.KappaBeta}, also 
\begin{align*}
  \frac{\kappa_n - \E\left[\kappa_n\right]}{\sqrt{\var\left(\kappa_n\right)}} \, 
  \frac{\sqrt{\var\left(\kappa_n\right)}}{\sqrt{\var\left(\beta_n \, \kappa_n\right)}}
   + \, \left( - \frac{1/\beta_n - \E[1/\beta_n]}{\sqrt{\var(1/\beta_n)}} \right) \, 
        \frac{\E[\kappa_n] \, \sqrt{\var\left(\beta_n\right)}}{\sqrt{\var\left(\beta_n \, \kappa_n\right)}}
\end{align*}
weakly converges to a standard normal distribution. 
Finally, since 
$$
  \limsup_{n \to \infty} \frac{\sqrt{\var\left(\beta_n\right)}}{\sqrt{\var\left(\beta_n \, \kappa_n\right)}} < \infty
  \qquad \textrm{and} \qquad 
  \limsup_{n \to \infty} \frac{\sqrt{\var\left(\kappa_n \right)}}{\sqrt{\var\left(\beta_n \, \kappa_n \right)}} < \infty
$$
due to Proposition \ref{AN.Lemma:TaylorExp}\eqref{AN.Lemma:TaylorExp.UI.4} and Lemmas \ref{AN.Lemma:AsympVar.kappa}\eqref{AN.Lemma:AsympVar.kappa1} and \ref{AN.Lemma:AsympVar.KappaBeta}\eqref{AN.Lemma:AsympVar.KappaBeta.Eq1}  we obtain from
Proposition \ref{AN.Lemma:TaylorExp}\eqref{AN.Lemma:TaylorExp.UI.6} that also
\begin{align*}
  \frac{\kappa_n - \E\left[\kappa_n\right]}{\sqrt{\var\left(\kappa_n\right)}} \, 
  \frac{\sqrt{\var\left(\kappa_n\right)}}{\sqrt{\var\left(\beta_n \, \kappa_n\right)}}
   + \, \frac{\beta_n - \E[\beta_n]}{\sqrt{\var(\beta_n)}} \, 
        \frac{\E[\kappa_n] \, \sqrt{\var\left(\beta_n\right)}}{\sqrt{\var\left(\beta_n \, \kappa_n\right)}}
\end{align*}
is asymptotically standard normal. It follows that
\begin{align*}
  & \frac{\kappa_n - \E\left[\kappa_n\right]}{\sqrt{\var\left(\kappa_n\right)}} \, \beta_n \, 
    \frac{\sqrt{\var\left(\kappa_n\right)}}{\sqrt{\var\left(\beta_n \kappa_n\right)}}
   + \, \frac{\beta_n - \E[\beta_n]}{\sqrt{\var(\beta_n)}} \, 
        \frac{\E[\kappa_n] \, \sqrt{\var\left(\beta_n\right)}}{\sqrt{\var\left(\beta_n \kappa_n\right)}}
  \\
  & = - \, \frac{T_n - \E[T_n]}{\sqrt{\var(T_n)}} 
      + \frac{\Cov\left(\beta_n, \kappa_n\right)}{\sqrt{\var\left(\beta_n \kappa_n \right)}}
  \\
  & = - \, \frac{T_n - \E[T_n]}{\sqrt{\var(T_n)}} 
      + \Cor\left(\beta_n, \kappa_n\right) \, 
        \frac{\sqrt{\var\left(\kappa_n \right)}}{\sqrt{\var\left(\beta_n \kappa_n \right)}} \, \underbrace{\sqrt{\var\left(\beta_n\right)}}_{\longrightarrow 0}
\end{align*}
weakly converges to a standard normal distribution,
where the first equality is given by Eq. \eqref{AN.Rep:T3.2}.
This proves the assertion.
\end{proof}


\begin{proof}[Proof of Proposition \ref{propasybias}.]
    By applying the notation used for proving Theorem \ref{AN.Thm:AN},
    for \(\beta_n = (q-\alpha)/(q-\alpha_n)\) and  \(\kappa_n = (q - \Lambda_n)/(q - \alpha)\) defined in \eqref{defbeta_nkappa_n} and \(\Lambda_n\) and \(\alpha_n\) given by \eqref{defkappnalphan}, and for \(\beta := 1\) and \(\kappa = (q-\Lambda)/(q-\alpha)\,,\) we have
    \begin{align}\label{eqpropasybias1}
         \E[T_n] - T = \beta \kappa - \E[\beta_n \kappa_n] = (\beta-\E[\beta_n]) \kappa + (\kappa-\E[\kappa_n]) \E[\beta_n] - \Cov(\beta_n,\kappa_n)\,.
    \end{align}
    It holds that \(\kappa\in [0,q]\) and, due to Lemma \ref{AN.Lemma:Finite.Beta}, \(\beta_n \in [1/(3q),q]\,.\) 
     Since \(\beta_n \to \beta = 1\) almost surely and, due to  Lemma \ref{AN.Lemma:AsympVar.kappa}, \(\Var(\kappa_n) = O(1/n)\,,\) we have for \(q\geq 2\) that
    \begin{align}\label{eqpropasybias2}
        \Cov(\beta_n,\kappa_n) = \sqrt{\Var(\beta_n)}\sqrt{\Var(\kappa_n)} \Cor(\beta_n,\kappa_n) = O(n^{-1/2})\,.
    \end{align}
    For \(q=1\,,\) note that \(\beta_n = 1\) and thus \(\Cov(\beta_n,\kappa_n) = 0\) for all \(n\,.\)
    Using Assumptions \ref{Assumption.1a} and \ref{Assumption.2a}, we obtain from \citep[Proposition 1.1]{han2022limit} for \(i\in \{1,\ldots,q-1\}\) that
    \begin{align}\label{eqpropasybias3}
        \left|\,\E[\xi_n(Y_i,(\XX,Y_{i-1},\ldots,Y_1))] - \xi(Y_i,(\XX,Y_{i-1},\ldots,Y_1))\, \right| &= O\Big(\frac{(\log n)^{p+i+\gamma_i+\1_{p+i=2}}}{n^{1/(p+i-1)}}\Big)\,,\\
        \label{eqpropasybias4}
        \left|\,\E[\xi_n(Y_i,(Y_{i-1},\ldots,Y_1))] - \xi(Y_i,(Y_{i-1},\ldots,Y_1)) \,\right| &= O\Big(\frac{(\log n)^{i+\gamma_i'+\1_{i=2}}}{n^{1/(i-1)}}\Big)\,.
    \end{align}
    For the second term on the right hand side of \eqref{eqpropasybias1}, it follows that 
    \begin{align}\label{eqpropasybias5}
       \left| \kappa - \E[\kappa_n] \right| 
       = O\Big( \frac{(\log n)^{d+\gamma + \1_{d=2}}}{n^{1/(d-1)}}\Big)
    \end{align}
    For the first term on the right hand side of \eqref{eqpropasybias1}, we have
    \begin{align}\label{eqpropasybias5b}
        \left|\beta - \E[\beta_n] \right| = (q-\alpha) \left| \E\Big[\frac{1}{q-\alpha} - \frac{1}{q-\alpha_n}\Big] \right| = \left| \E\Big[ \frac{\alpha_n - \alpha}{q-\alpha_n}\Big]\right| \,.
    \end{align}
    Due to \eqref{boundalpha_n}, we know that \(1/(3q) \leq 1/(q-\alpha_n) \leq 1\) \(P\)-almost surely. Hence, noting that the case \(q=1\) is trivial, \eqref{eqpropasybias5b} implies for \(q\geq 2\) and \(\gamma'= \max_i\{\gamma_i'\}\) that
    \begin{align}\label{eqpropasybias6}
        |\beta - \E[\beta_n]| \approx  \left|\alpha - \E[\alpha_n]\right| = O\Big(\frac{(\log n)^{q+\gamma' }}{n^{1/(q-1)}} \Big)\,,
    \end{align}
    where \(\approx\) indicates the same rate of convergence.
    Combining \eqref{eqpropasybias1}, \eqref{eqpropasybias2}, \eqref{eqpropasybias5} and \eqref{eqpropasybias6} yields the assertion.
\end{proof}




\section{Discussion}\label{secdisc}

As a direct extension of Azadkia and Chatterjee's rank correlation to multi-response variables, we have introduced the quantity \(T\) that satisfies all axioms of a measure of predictability and, additionally, fulfills an information gain inequality, characterizes conditional independence, is self-equitable, and satisfies a data-processing inequality. Further, we have proposed a model-free, strongly consistent, merely rank-based estimator for \(T\) that can be computed in almost linear time and we proved its asymptotically normality. As a powerful application of the measure \(T\) and its estimator, we have obtained a new model-free variable selection method for multi-outcome data, which works without any tuning parameters and outperforms competing methods in various scenarios.
Further, the measure \(T\) supports a wide range of applications concerning the strength of dependence among groups of random variables; see \citep{sfx2024cluster} for a variable clustering of random variables and see Section \ref{secnetworks} in the Supplementary Material for identifying networks based on $T$ via graphs.

\section*{Acknowledgement}
Both authors gratefully acknowledge the support of the Austrian Science Fund (FWF) project {P 36155-N} \emph{ReDim: Quantifying Dependence via Dimension Reduction}
and the support of the WISS 2025 project 'IDA-lab Salzburg' (20204-WISS/225/197-2019 and 20102-F1901166-KZP).


\addcontentsline{toc}{section}{Bibliography}


\newpage

\begin{appendices}
\section*{Supplementary Material}

\section{Additional Material for Section \ref{secmainres}}
\label{App.Add.SecMainRes}

\subsection{Invariance Properties of \(T\)}

As shown in Corollary \ref{Cor.T.DistTrans}, \(T(\YY,\XX)\) is invariant with respect to the transformation of the random variables by their individual distribution functions.
We now make use of the data processing inequality and highlight further important invariance properties of \(T(\YY,\XX)\) concerning the distributions of \(\XX\) and \(\YY\).
Therefore, let \(\VV=(V_1,\ldots,V_p)\) be a random vector (on \((\Omega,\cA,P)\) which is assumed to be non-atomic) that is independent of \(\XX = (X_1,\ldots,X_p)\) and uniformly on \((0,1)^p\) distributed. Then
the \emph{multivariate distributional transform} (also known as \emph{generalized Rosenblatt transform}) \(\tau_{\XX} (\XX,\VV)\) of \(\XX\) is defined by  
\begin{align*}
    \tau_{\XX} (\xx,\boldsymbol{\lambda})
    := \big( F_1(x_1,\lambda_1), F_2(x_2,\lambda_2|x_1) \dots, F_p(x_p,\lambda_p|x_1,\dots,x_{p-1}) \big)
\end{align*}
for all \(\xx=(x_1,\ldots,x_p) \in \mathbb{R}^p\) and all \(\boldsymbol{\lambda}=(\lambda_1,\ldots,\lambda_p) \in [0,1]^p\), where
\begin{eqnarray*}
    F_1(x_1,\lambda_1) 
    & := & P(X_1 < x_1) + \lambda_1 \, P(X_1 = x_1)
    \\
    F_k(x_k,\lambda_k|x_1,\dots,x_{k-1}) 
    & := & P(X_k < x_k | X_1=x_1, \dots, X_{k-1}=x_{k-1}) 
    \\
    &    & + \lambda_k \, P(X_k = x_k | X_1=x_1, \dots, X_{k-1}=x_{k-1})\,, \quad k\in \{2,\ldots,p\}\,. 
\end{eqnarray*}
For \(p=1\) and for a random variable \(X\) with continuous distribution function \(F_X\,,\) the distributional transform \(\tau_X(X,V)\) simplifies to \(F_X(X)\,,\) which is uniformly on \((0,1)\) distributed.
\\
As an inverse transformation of \(\tau_\XX\,,\) for a random vector \(\UU=(U_1,\ldots,U_p)\) uniformly on \((0,1)^p\) distributed, the \emph{multivariate quantile transform} \(q_\XX (\UU):=(\xi_1,\ldots,\xi_p)\) is defined by 
\begin{align}
\begin{split}
    \xi_1 &:= F_{X_1}^{-1}(U_1)\,, \label{MQT1}  \\
    \xi_k &:= F^{-1}_{X_k|X_{k-1}=\xi_{k-1},\ldots,X_1=\xi_1}(U_k)~~~ \text{for all } k\in \{2,\ldots,p\}
\end{split}
\end{align}
where \(F_{W|\ZZ=\zz}\) denotes the conditional distribution function of \(W\) given \(\ZZ=\zz\), and
\(F_{Z}^{-1}\) denotes the generalized inverse function of \(F_{Z}\), i.e., 
\(F_{Z}^{-1}(u) := \inf\{z \in \mathbb{R} \, : \, F_{Z}(z) \geq u \}\).

According to \citep{Arjas-1978, OBrien-1975, Ru-1981, Ru-2013},
\begin{align}
\tau_\XX(\XX,\VV) &\text{ is a random vector that is uniformly on } (0,1)^p \text{ distributed,}\\
 \label{eqpropqt} q_\XX(\UU) &\text{ is a random vector with distribution function } F_\XX,
\end{align}
and the multivariate quantile transform is inverse to the multivariate distributional transform,
i.e., 
\begin{align}\label{eqpropqt2}
  \XX = q_\XX \big( \tau_{\XX} (\XX,\VV) \big) \qquad P\text{-almost surely.}
\end{align}

The following result extends the distributional invariance of \(T\) presented in Corollary \ref{Cor.T.DistTrans} and shows that the value \(T(\YY,\XX)\) also remains unchanged when replacing the predictor variables \(X_1, \dots, X_p\) by their individual distributional transforms, i.e., the predictor variables can be replaced by their individual ranks (with ties broken at random). 
Further, the value \(T(\YY,\XX)\) even remains unchanged when replacing the vector of predictor variables \(\XX\) by its multivariate distributional transform and hence by a vector of independent and identically distributed random variables.
Interestingly, and in contrast to the situation described above, it turns out that \(T(\YY,\XX)\) is not invariant with respect to the (individual) distributional transforms of \(\YY\).

\begin{corollary}[Distributional invariance II]\label{Cor.T.PIT}~~
The map \(T\) defined by \eqref{defmdm} fulfills
\begin{enumerate}[(i)]
\item 
\(T(\YY,\XX) = T (\YY,\tau_{\XX} (\XX,\VV))\),
\item
\(T(\YY,\XX) = T \big(\YY, (\tau_{X_1} (X_1,V_1), \dots, \tau_{X_p} (X_p,V_p))\big)\).
\end{enumerate}
However,  
\(T(\YY,\XX) \neq T (\tau_{\YY} (\YY,\VV),\XX)\) and \(T(\YY,\XX) \neq T ((\tau_{Y_1}(Y_1,V_1),\ldots,\tau_{Y_q}(Y_q,V_q)),\XX)\), in general.
\end{corollary}
\begin{proof}
From Eq. \eqref{eqpropqt2} and the data processing inequality in Corollary \ref{cor.T.DPI} we conclude that 
\begin{align*}
  T(\YY,\XX)
    =  T\big(\YY \,, \, q_\XX \big( \tau_{\XX} (\XX,\VV) \big)\big)
  \leq T\big(\YY \,, \, \tau_{\XX} (\XX,\VV)\big)
  \leq T(\YY,\XX)\,,
\end{align*}
hence \(T(\YY,\XX) = T (\YY,\tau_{\XX} (\XX,\VV))\).
The same reasoning yields
\begin{align*}
  T(\YY,\XX) 
  = T \big(\YY, (\tau_{{X_1}} (X_1,V_1), \dots, \tau_{{X_p}} (X_p,V_p))\big)\,,
\end{align*}
where \(\tau_{{X_k}} (X_k,V_k)\), \(k \in \{1,\ldots,p\}\), denote the (individual) univariate distributional transforms.
This proves the first part of Corollary \ref{Cor.T.PIT}.
\\
We now show \(T(\YY,\XX) \neq T (\tau_{\YY} (\YY,\VV),\XX)\), in general.
To this end, consider the random variable \(X\) with \(P(X=0) = 0.5 = P(X=1)\)
and assume that \(Y=X\) which gives \(T(Y,X)=1\).
Since \(\tau_{Y} (Y,V)\) is a uniformly on \((0,1)\) distributed random variable, 
it can not be a measurable function of \(X\) implying \(T(\tau_{Y} (Y,V),X) < 1 = T(Y,X)\).
\end{proof}

Another important property of \(T(\YY,\XX)\) is its invariance under strictly increasing and bijective transformations of \(Y_i\) and under bijective transformations of \(\XX\) as follows.

\begin{theorem}[Invariance under (strictly increasing) bijective transformations]\label{theT6}~~\\
Let \(\bm{g}=(g_1,\ldots,g_q)\) be a vector of strictly increasing and bijective transformations \linebreak \(g_i\colon \R\to \R\), \(i\in \{1,\ldots,q\}\),
and let \(\bm{h}\) be a bijective transformation \(\R^p \to \R^p\). 
Then \(T(\bm{g}(\YY),\bm{h}(\XX)) = T(\YY,\XX)\,.\)
\end{theorem}
\begin{proof}
The invariance of $\xi(\,\cdot\,,\XX)$ and $T(\,\cdot\,,\XX)$ w.r.t. a bijective function $\bm{h}$ of $\XX$ follows from the definitions of $\xi$ and $T$ because the \(\sigma\)-algebras generated by \(\XX\) and \(\bm{h}(\XX)\) coincide.
The invariance of $\xi(Y_i,\,\cdot\,)$ and $T(\YY,\,\cdot\,)$ w.r.t. strictly increasing and bijective functions $g_i$ of $Y_i$ now follows from the invariance of the ranges under strictly increasing transformations and the above mentioned invariance w.r.t. to bijective transformations of the response variables. 
\end{proof}

Theorem \ref{theT6} implies invariance of $T$ under permutations within the conditioning vector $\XX$.
The following Proposition \ref{ppropinvper} and Corollary \ref{theT7ax} also provide sufficient conditions on the underlying dependence structure for the invariance of $T$ under permutations within the response vector $\YY$ using the notion of copulas:
A \(d\)-copula is a \(d\)-variate distribution function on the unit cube \([0,1]^d\) with uniform univariate marginal distributions. 
Denote by \(\Ran(F):=\{F(x) \,,\, x\in \R\}\) the range of a univariate distribution function \(F\) 
and denote by \(\overline{A}\) the closure of a set \(A\subset \R\,.\)
Due to Sklar's Theorem, the joint distribution function of \((\XX,\YY)\) can be decomposed into its univariate marginal distribution functions and 
a \((p+q)\)-copula $C_{\XX,\YY}$ such that
\begin{align}\label{eqsklar}
    F_{\XX,\YY}(\xx,\yy) = C_{\XX,\YY}(F_{X_1}(x_1),\ldots,F_{X_p}(x_p),F_{Y_1}(y_1),\ldots,F_{Y_q}(y_q))
\end{align}
for all \((\xx,\yy)=(x_1,\ldots,x_p,y_1,\ldots,y_q)\in \R^{p+q}\,,\) where \(C_{\XX,\YY}\) is uniquely determined on \(\Ran(F_{X_1})\times \cdots \times \Ran(F_{X_p})\times \Ran(F_{Y_1})\times \cdots \times  \Ran(F_{Y_q})\);
for more background on copulas and Sklar's Theorem we refer to \citep{fdsempi2016,Nelsen-2006}.

\begin{proposition}[Invariance under permutations]\label{ppropinvper}~~
Assume that 
\(C_{\XX,\YY_{\boldsymbol{\sigma}}}=C_{\XX,\YY}\) for all \(\boldsymbol{\sigma}\in S_q\) and
\(\overline{\Ran(F_{Y_1})} = \cdots = \overline{\Ran(F_{Y_q})}\,.\)
Then \(T(\YY_\sigma,\XX)=T(\YY,\XX)\) for all \(\boldsymbol{\sigma}\in S_q\,.\)
\end{proposition}
\begin{proof}
Due to Proposition \ref{lemrepT}, \(\xi(Y_i,(\XX,Y_{i-1},\ldots,Y_1))\) only depends on the conditional distribution function \(F_{Y_i|(\XX,Y_{i-1},\ldots,Y_1)}\,,\) on the distribution \(P^{\XX,Y_{i-1},\ldots,Y_1}\) and on \(\Ran(F_{Y_i})\). 
As a consequence of Sklar's Theorem (see Eq. \eqref{eqsklar}), the conditional distribution depends only on the copula \(C_{\XX,Y_i,\ldots,Y_1}\) and on the marginal distribution functions \(F_{Y_i},\ldots,F_{Y_1}\) as well as \(F_{X_1},\ldots,F_{X_p}\).
By the invariance properties of \(T\) and \(\xi\) (see Theorem \ref{theT6}), it follows that
\(\xi(Y_i,(\XX,Y_{i-1},\ldots,Y_1))\) only depends on \(C_{\XX,Y_i,\ldots,Y_1}\) and \(\overline{\Ran(F_{Y_i})},\ldots,\overline{\Ran(F_{Y_1})}\) as well as on \(\overline{\Ran(F_{X_1})},\ldots,\overline{\Ran(F_{X_p})}\,.\) 
The assertion now follows from the assumptions.
\end{proof}

A \(d\)-variate random vector \(\WW=(W_1,\ldots,W_d)\) is said to be \emph{exchangeable} if \(\WW \eqd \WW_{\boldsymbol{\sigma}}\) for all \({\boldsymbol{\sigma}}\in S_d\,,\) 
where '\,\(\eqd\)\,' denotes equality in distribution.
The following corollary is an immediate consequence of the previous result.

\begin{corollary}[Invariance under exchangeability]\label{theT7ax}~~
Assume that the random vector \((\XX,\YY)\) is exchangeable. 
Then \(T(\YY_{\boldsymbol{\sigma}},\XX)=T(\YY,\XX)\) for all \({\boldsymbol{\sigma}}\in S_q\,.\)
\end{corollary}

\subsection{Special Cases and Closed-Form Expressions}

\bigskip
For some special cases concerning independence, conditional independence, and perfect dependence of the response variables, the measures \(T\) and its permutation invariant version $\overline{T}$ given by \eqref{defmdmav} attain a simplified form as follows.

\begin{remark}[Special cases regarding the dependence structure of $\YY$]
\label{theT7} \leavevmode
\begin{enumerate}[(i)]
    \item \label{theT71}  If \(Y_1,\ldots,Y_q\) are independent, then 
    \begin{align}\label{eqtheT71}
        T(\YY,\XX) 
        = \frac{1}{q} \sum_{i=1}^q \xi(Y_i , (\XX,Y_{i-1},\dots,Y_{1}))\,.
    \end{align}

    \item \label{theT71b}  If \(Y_1,\ldots,Y_q\) are independent and conditionally independent given $\XX$, then 
    \begin{align}\label{eqtheT71b}
        T(\YY,\XX) 
        = \frac{1}{q} \sum_{i=1}^q \xi(Y_i , \XX) 
        = \overline{T}(\YY,\XX) 
        = T^\Sigma(\YY,\XX)\,.
    \end{align}
    
    \item \label{theT72} 
    If, for \(j > i\), \(Y_j\) is perfectly dependent on \(Y_i\)
    then
    \begin{align}\label{eqtheT72}
        T(\YY,\XX) = T(\YY_{-j},\XX)\,.
    \end{align}
    where \(\YY_{-j}\) denotes the vector of response variables excluding variable \(Y_j\).
\end{enumerate}
\end{remark}
Example \ref{Ex.Cube} below illustrates the difference between Statements \eqref{theT71} and \eqref{theT71b} in the above Remark \ref{theT7} .  
Also note that Eq. \eqref{eqtheT72} is not invariant under permutations of the components of \(\YY\,,\) i.e., if \(Y_1\) is perfectly dependent on \(Y_2\,,\) then \(T((Y_1,Y_2),\XX) \neq T(Y_2,\XX)\) in general.

The next example shows that \(T^\Sigma\) defined in \eqref{TSigma} is not a measure of predictability.

\begin{example} \label{Ex.Cube}
Consider the random vector $(X,Y_1,Y_2)$ whose mass is distributed uniformly within the four cubes
\begin{eqnarray*}
  \big(0,\tfrac{1}{2}\big) \times \big(0,\tfrac{1}{2}\big) \times \big(0,\tfrac{1}{2}\big)
	& \quad & \big(\tfrac{1}{2},1\big) \times \big(\tfrac{1}{2},1\big) \times \big(0,\tfrac{1}{2}\big)
	\\
	\big(0,\tfrac{1}{2}\big) \times \big(\tfrac{1}{2},1\big) \times \big(\tfrac{1}{2},1\big)
	& \quad & \big(\tfrac{1}{2},1\big) \times \big(0,\tfrac{1}{2}\big) \times \big(\tfrac{1}{2},1\big)
\end{eqnarray*}
and has no mass outside these cubes; cf. \citep[Example 3.4.]{sfx2021vine}. 
Then $Y_1$ and $Y_2$ are independent but not conditionally independent given \(X\).
Additionally, \(X\) and \(Y_1\) as well as \(X\) and \(Y_2\) are independent, and hence 
\begin{equation*}
  T((Y_1,Y_2),X)
  = \frac{T(Y_1,X) + T(Y_2 , (X,Y_{1}))}{2} 
  = \frac{T(Y_2 , (X,Y_{1}))}{2}  
  = \frac{1}{4}
  > 0 
  = T^\Sigma((Y_1,Y_2),X).
\end{equation*}
 Consequently, since \(T\) satisfies axiom (A\ref{prop2}), \(T^\Sigma\) does not satisfy axiom (A\ref{prop2}) and thus it is not a measure of predictability.
\end{example}


In the special case the random vector \((\XX,\YY) \sim N({\bf 0},\Sigma)\) exhibits an equicorrelated structure, the results presented in Proposition \ref{propChatformmGau} become more explicit.

\begin{example}\label{exlimcasmnd}
Consider the specific case where the covariance matrix \(\Sigma\) has the decomposition given by
    \begin{align}\label{eqdecsig5}
        \Sigma_{11}=\begin{pmatrix}
            1 & \rho_X & \cdots & \rho_X \\
            \rho_X & 1 & \ddots & \vdots\\
            \vdots & \ddots & \ddots & \rho_X \\
            \rho_X & \cdots & \rho_X & 1
        \end{pmatrix}\,, ~
        \Sigma_{21}=\begin{pmatrix}
            \rho_{XY} & \cdots & \rho_{XY} \\
            \vdots & \ddots & \vdots \\
            \rho_{XY} & \cdots & \rho_{XY} 
        \end{pmatrix}\,,~
         \Sigma_{22}=\begin{pmatrix}
            1 & \rho_Y & \cdots & \rho_Y \\
            \rho_Y & 1 & \ddots & \vdots\\
            \vdots & \ddots & \ddots & \rho_Y \\
            \rho_Y & \cdots & \rho_Y & 1
        \end{pmatrix}
    \end{align}
for some correlation parameters \(\rho_X,\rho_{XY},\rho_Y\in [-1,1]\,.\) For \(p,q>1\,,\) elementary calculations show that \(\Sigma\) is positive semi-definite if and only if 
\(\rho_X \in [-1/(p-1),1]\,,\) \(\rho_Y \in [-1/(q-1),1]\) and
\begin{align}\label{assposdefsigg}  
  \rho_{XY}^2 
  & \leq \frac{1+(p-1)\rho_X}{p}\frac{1+(q-1)\rho_Y}{q}\,. 
\end{align}
Due to Proposition \ref{propExtcasemn}\eqref{propExtcasemn2} we know that \(T(\YY,\XX)=1\) if and only if \(\rank(\Sigma)=\rank(\Sigma_{11})\,.\) The latter is by \eqref{assposdefsigg} equivalent to 
\begin{align}\label{eqsiggh}
   \rho_{XY}^2 = \frac{1+(p-1)\rho_X}{p}\frac{1+(q-1)\rho_Y}{q} ~~~ \text{and} ~~~ \rho_Y=1\,,
\end{align}
noting that, since \(\Sigma_{21}\) is assumed to be constant, there is no choice for \(\rho_Y\) other than \(1\) such that \(\rank(\Sigma)=\rank(\Sigma_{11})\,.\)
\\
Due to Proposition \ref{propChatformmGau}, straightforward calculations yield
    \begin{align}
        \xi(Y_i,(\XX,Y_{i-1},\ldots,Y_1)) 
        &\phantom{:}= \frac{3}{\pi} \arcsin\left(\frac{1+\rho^*(p,i)}{2}\right)-\frac 1 2\,,\\
        \xi(Y_i,(Y_{i-1},\ldots,Y_1)) &\phantom{:}= \frac{3}{\pi} \arcsin\left(\frac{1+\rho^*(i)}{2}\right) - \frac 1 2 ~~~~~~~~~~~~~~ \text{for } i\in \{2,\ldots,q\}\,,\\
        \label{exnnn}   \text{where } \rho^*(p,i) &:= \begin{cases}
        \frac{p\rho_{XY}^2}{1+(p-1)\rho_X} & \text{for } i=1\,,\\
            \frac{(1+(p-1)\rho_X)(i-1)\rho_Y^2-p(i\rho_Y-1)\rho_{XY}^2}{(1+(p-1)\rho_X)(1+(i-2)\rho_Y)-p(i-1)\rho_{XY}^2}
            &\text{for } i\in \{2,\ldots,q\}\,,
        \end{cases} \\
        \label{ennn2}\text{and } \rho^*(i)&:= \frac{(i-1)\rho_Y^2}{1+(i-2)\rho_Y} \,.
    \end{align}
We further observe that 
\begin{enumerate}[(i)]
\item 
\(T(\YY,\XX)\in [0,1]\) because \(\rho^*(p,1)\in [0,1]\) and \(0\leq \rho^*(i)\leq \rho^*(p,i)\leq 1\) for all \(i\in \{2,\ldots,q\}\) and thus \(T(Y_1,\XX)\in [0,1]\) and \(0 \leq T(Y_i,(Y_{i-1},\ldots,Y_i)) \leq\) 
\linebreak \(T(Y_i,(\XX,Y_{i-1},\ldots,Y_i))\leq 1\) for all \(i\in \{2,\ldots,q\}\,.\) 
\item 
\(T(\YY,\XX)=0\) 
if and only if \(\rho^*(p,1)=0\) and \(\rho^*(p,i)=\rho^*(i)\) for all \(i\in \{2,\ldots,q\}\) if and only if \(\rho_{XY}=0\), i.e., \(\Sigma_{12}\) is the null matrix. 
\item \(T(\YY,\XX)=1\) 
if and only if \(\rho^*(p,i)=1\) for all \(i\in \{1,\ldots,q\}\) if and only if \(\rho_{XY}^2 = \frac{1+(p-1)\rho_X}{p}\) for \(i=1\) and 
        \(\rho_{XY}^2 = \frac{1+(p-1)\rho_X}{p}\frac{1+(i-1)\rho_Y}{i}
                      = \rho_{XY}^2 \frac{1+(i-1)\rho_Y}{i}\) for all \(i\in \{2,\ldots,q\}\)
if and only if, due to \eqref{eqsiggh}, \(\rank(\Sigma)=\rank(\Sigma_{11})\,.\)  
\end{enumerate}
\end{example}



\section{Additional Material for Section \ref{secconasy}}
\label{App.Add.Secconasy}

\subsection{Simulation Study in Multivariate Normal Models}

We illustrate the small and moderate sample performance of our estimator $T_n$ 
in the case where the random vector $(\XX,\YY)$ follows a multivariate normal distribution according to Example \ref{exlimcasmnd} with
\begin{enumerate}[(i)]
    \item $p=5$ predictor and $q=2$ response variables and with correlation parameters $\rho_X=0.5$, $\rho_Y=0.2$, and $\rho_{XY}=0.5$, where \(T(\YY,\XX) \approx 0.2712\), and 
    \item $p=2$ predictor and $q=4$ response variables and with correlation parameters $\rho_X=0.25$, $\rho_Y=0.75$, and $\rho_{XY}=0.5$, where \(T(\YY,\XX) \approx 0.1054\),
\end{enumerate}
respectively.
To test the performance of the estimator \(T_n\) in different settings, samples of size $n \in \{20; 50; 100; 200; 500; 1,000; 5,000; 10,000; 50,000\}$ are generated and then $T_n$ is calculated.
These steps are repeated $R=1,000$ times.
Fig. \ref{Fig.Gauss.Sim} depicts the estimates of $T_n$ for different samples sizes and relates it to the true dependence value (dashed red line).
\\
As can be observed from Figure \ref{Fig.Gauss.Sim} (and as expected), the estimates converge rather fast to the true values.
Notice that the estimator \(\overline{T}_n\) performs comparably to \(T_n\).

\begin{figure}[t!]
		\centering
  \caption{Boxplots summarizing the $1,000$ obtained estimates for $T_n$.
		Samples of size $n$ are drawn from a multivariate normal distribution 
        with $5$ predictor and $2$ response variables (left panel) and with $2$ predictor and $4$ response variables (right panel).}
		\label{Fig.Gauss.Sim}
		\includegraphics[width=1.0\textwidth]{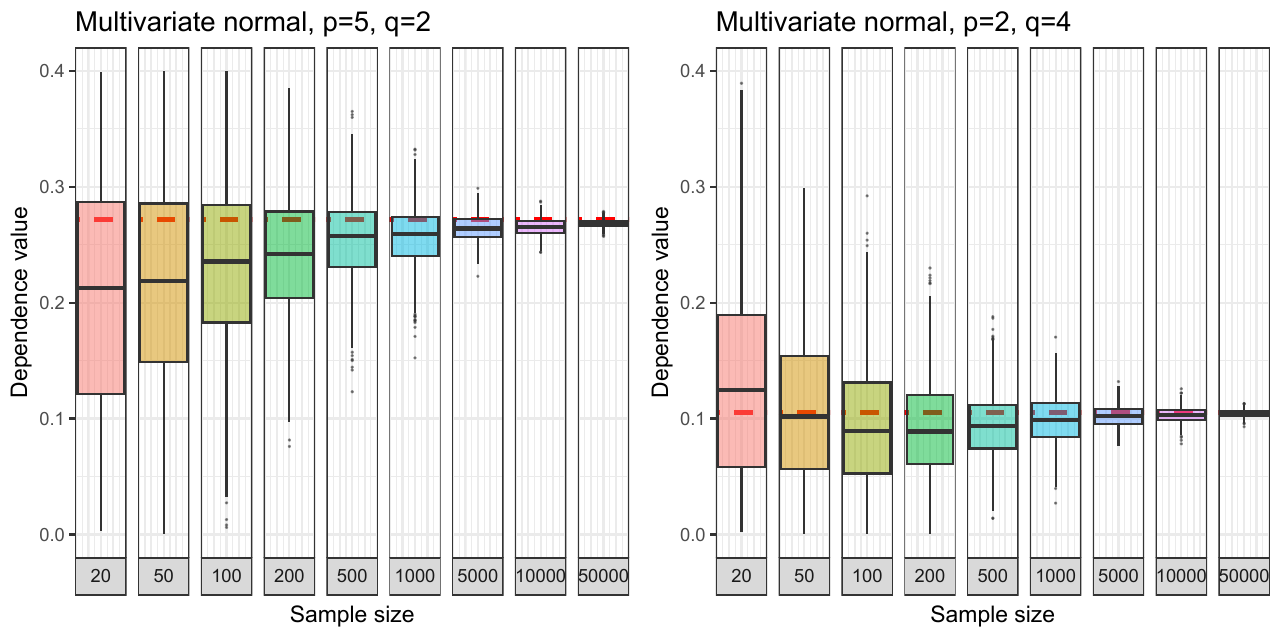}
\end{figure}

\subsection{Comparison of \(T\) with static convex combinations \(\kappa^\alpha\)}

We consider a sample drawn from random variables $X \sim N(0,1)$, $Y_1 \sim N(0,1)$ and $Y_2 = Y_1 + N(0,\sigma^2)$ with sample size $10,000$ and \(\sigma >0\,.\) By construction, the response vector \(\YY=(Y_1,Y_2)\) is independent from \(X\) and thus \(T(\YY,X) = 0 = \kappa^{\boldsymbol{\alpha}}(\YY,X)\) for all weights \({\boldsymbol{\alpha}}\,,\) where \(\kappa^{\boldsymbol{\alpha}}\) is defined in \eqref{KappaAverage2}. The dependence structure among \(\YY\) increases with decreasing \({\boldsymbol{\alpha}}\,.\) As Figure \ref{Fig.Sim} illustrates, the nearest neighbor-based plug-in estimator fails to be useful for \(\kappa^{\boldsymbol{\alpha}}\) with deterministic \({\boldsymbol{\alpha}}\,.\)

\begin{figure}[thb]
  \centering
  \includegraphics[width=0.9\textwidth]{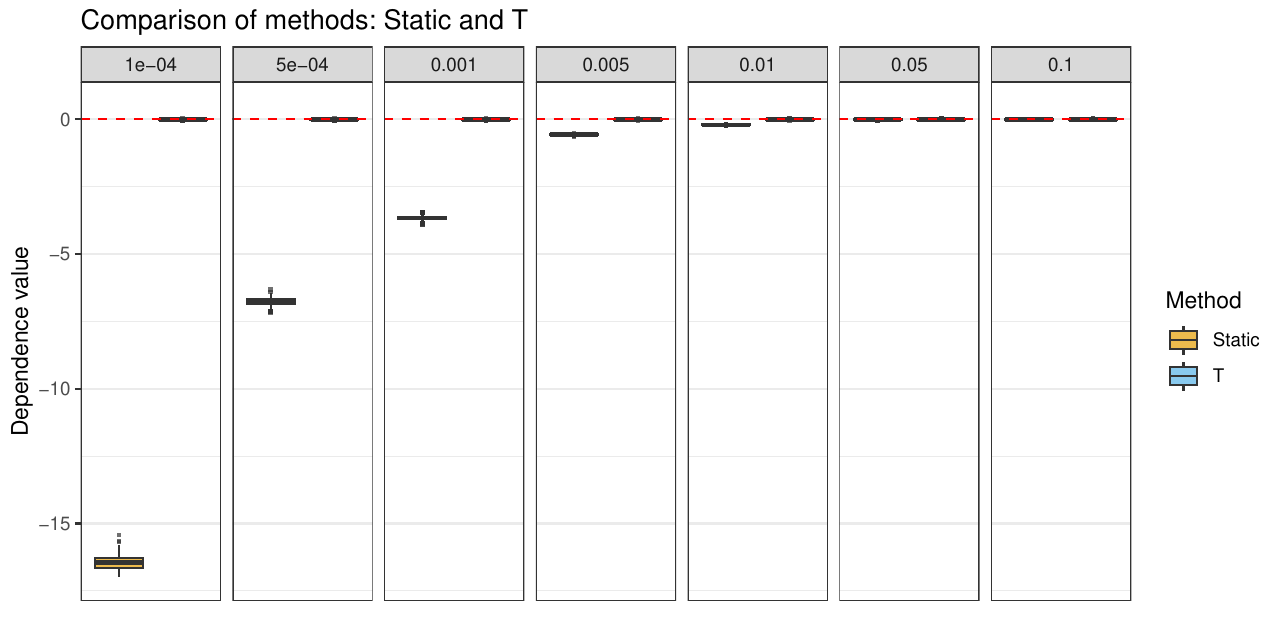}
  \caption{Boxplots for varying $\sigma \in \{0.0001, 0.0005, 0.001, 0.005, 0.01, 0.05, 0.1\}$ from left to right comparing the $1,000$ obtained dependence values of the static convex combination \(\kappa^{\boldsymbol{\alpha}}((Y_1,Y_2),X)\) in \eqref{KappaAverage2} with fixed weights $(\alpha_1,\alpha_2)=(0.5,0.5)$ estimated via R function \texttt{codec} (R package FOCI) with those of $T((Y_1,Y_2),X)$ in \eqref{defmdm} estimated via  R function \texttt{didec} (R package didec). 
   Since $(Y_1,Y_2)$ is independent of $X$, the true dependence value equals $0$ (depicted by the red dashed line).}
  \label{Fig.Sim}
\end{figure}

\section{Additional Material for Section \ref{secappl}}\label{App.Add.secappl}

\subsection{Plausibility of Multivariate Feature Selection} \label{RDE.Subsect.FS}

Exemplarily, we evaluate bioclimatic data to illustrate that MFOCI is plausible in the sense that 
it chooses a small number of variables that include, in particular, the most important variables for the individual feature selections.

\bigskip\noindent
\textit{Analysis of global climate data}\label{secglcldata}
\\
We revisit the the global climate data set from Subsection \ref{RDE.Subsect.FSII} and analyze the influence of a set of  thermal and precipitation-related variables (see Table \ref{Fig.Climate1})
on the pair \emph{Annual Mean Temperature} (AMT) and \emph{Annual Precipitation} (AP).
By applying the coefficient $T$ we first perform a forward feature selection and identify those variables that best predict the response vector (AMT,AP) (= variables $(Y_1,Y_2)$). Then, we compare the outcome with the forward feature selections that refer to the individual response variables.
First, note that the output variables AMT and AP exhibit some positive dependence (their Pearson correlation is $0.52$ and
their Spearman's rank correlation equals $0.61$).
Second, recall that our variable selection method requires neither knowledge of the marginal distributions nor knowledge of the dependence structure between or among the predictor and response variables.

\begin{table}[t!]
\small
\centering
\caption{Thermal and precipitation-related variables used as predictor variables;
see Subsection \ref{RDE.Subsect.FS} for details.}
\label{Fig.Climate1}
\begin{tabular}{l|l||l|l}
MTWeQ & Mean Temperature of Wettest Quarter
& PTWeQ & Precipitation of Wettest Quarter
\\
MTDQ & Mean Temperature of Driest Quarter
& PTDQ & Precipitation of Driest Quarter
\\
MTWaQ & Mean Temperature of Warmest Quarter 
& PTWaQ & Precipitation of Warmest Quarter 
\\
MTCQ & Mean Temperature of Coldest Quarter
& PTCQ & Precipitation of Coldest Quarter
\end{tabular}
\end{table}

Table \ref{Fig.Climate2} depicts the order of the via MFOCI selected variables based on the estimated values for $T$.
There, the values in line $k$ indicate the estimated values for $T(\YY,(X_1,\dots,X_k))$ where $X_1,\dots,X_k$ are the variables in lines $1$ to $k$. 
For the prediction of the response vector \((Y_1,Y_2)=\) (AMT, AP),
MFOCI selects the four variables \{MTWaQ, PWeQ, MTCQ, PDQ\} 
(at this point it is worth mentioning that both the feature selection referring to the permuted vector \((Y_2,Y_1)=\) (AP, AMT) and the feature selection based on \(\overline{T}\) identify the same four relevant variables.).
For the individual prediction of the response variable AMT and AP, respectively, 
MFOCI selects the variables \{MTWaQ, MTCQ, MTWeQ\} and \{PWeQ, PDQ, PCQ\}, respectively.
Remarkably, from this perspective, 
the chosen predictor variables of the multivariate feature selection for \((Y_1,Y_2)=\) (AMT,AP) are a proper subset of the union of the relevant predictor variables of the respective individual feature selections.
In this regard, our multi-output variable selection method is plausible.
\begin{table}[t!]
\small
\centering
\caption{Results of the forward feature selections based on the coefficient \(T\) to identify those variables that best predict AMT and AP and (AMT,AP), respectively;
see Subsection \ref{RDE.Subsect.FS} for details.
The variables selected via MFOCI to predict (AMT,AP) are marked in red color.
}
\label{Fig.Climate3}\label{Fig.Climate2}
\begin{tabular}{c||lc||lc|lc}
Position 
& \makecell[l]{Variables to \\ predict (AMT,AP)}
& $T_n$
& \makecell[l]{Variables to \\ predict AMT}
& $T_n$
& \makecell[l]{Variables to \\ predict AP}
& $T_n$
\\ 
\hline
\hline
1 
& \textcolor{foxred}{MTWaQ}
& 0.64
& \textcolor{foxred}{MTWaQ}
& 0.84
& \textcolor{foxred}{PWeQ}
& 0.80
\\
2 
& \textcolor{foxred}{PWeQ}
& 0.84
& \textcolor{foxred}{MTCQ} 
& 0.97
& \textcolor{foxred}{PDQ}
& 0.92
\\
3 
& \textcolor{foxred}{MTCQ}
& 0.89
& MTWeQ
& 0.98
& PCQ
& 0.93
\\
4 
& \textcolor{foxred}{PDQ}
& 0.91
& 
& 
& 
& 
\\
\end{tabular}
\end{table}

We observe from Table \ref{Fig.Climate2} that $T$ might fulfill some kind of reversed information gain inequality in the response part, i.e., adding response variables might lower predictability.
In general, however, such behaviour cannot be inferred, see Example \ref{Ex.Cube} where $T((Y_1,Y_2),X) = 1/4 > 0 = \max\{T(Y_1,X),T(Y_2,X)\}$.

\bigskip
As a second real-world example, we now evaluate medical data to once again illustrate that MFOCI is plausible in the sense that it chooses a small number of variables that include, in particular, the most important variables for the individual feature selections.

\bigskip\noindent
\textit{Predicting the extent of Parkinson's disease}\label{RDE.Subsect.FS.App}
\\
As illustrative example for feature selection in medicine, we determine the most important variables for predicting two UPDRS scores---the motor as well as the total UPDRS score---which are assessment tools used to evaluate the extent of Parkinson's disease in patients. The data set\footnote{UCI machine learning repository \citep{dua2019}; to download the data use 
\url{https://archive.ics.uci.edu/ml/datasets/Parkinsons\%2BTelemonitoring}. We excluded the data 'test\_time' and rounded the data 'motor\_UPDRS' and 'total\_UPDRS' to whole numbers.}
consists of \(n=5875\) observations including the two response variables (motor and total UPDRS score) as well as the predictor variables age, sex, and several data concerning shimmer and jitter which are related to the voice of the patient. While shimmer measures fluctuations in amplitude, jitter indicates fluctuations in pitch. Common symptoms of Parkinson’s disease include speaking softly and difficulty maintaining pitch. Therefore, measurements of both shimmer and jitter can be used to detect Parkinson's disease and thus the voice data can be useful for predicting the UPDRS scores.
Note that the response variables Motor UPDRS score and Total UPDRS score are strongly dependent---the data yield Spearman's correlation of \(0.95\) and  Kendall's correlation of \(0.85\,.\)

\begin{table}[t!]
\small
\centering
\caption{Results of the forward feature selections based on the coefficient \(T\) to identify those variables that best predict Motor UPDRS score, Total UPDRS score, as well as both Motor UPDRS score and Total UPDRS score, respectively;
see Section \ref{RDE.Subsect.FS.App} for details.
The variables selected via MFOCI to predict (Motor UPDRS score,Total UPDRS score) are marked in red color.
}
\label{Fig.Park}
\begin{tabular}{c||lc||lc|lc}
Position 
& \makecell[l]{Variables to predict \\ Motor and Total \\ UPDRS score}
& $T_n$
& \makecell[l]{Variables to \\ predict Motor \\ UPDRS score}
& $T_n$
& \makecell[l]{Variables to \\ predict Total \\ UPDRS score}
& $T_n$
\\ 
\hline
\hline
1 
& \textcolor{foxred}{age}
& 0.5316
& \textcolor{foxred}{age}
& 0.4935
& \textcolor{foxred}{age}
& 0.5154
\\
2 
& \textcolor{foxred}{sex}
& 0.6759
& \textcolor{foxred}{sex}
& 0.6604
& \textcolor{foxred}{sex}
& 0.6711
\\
3 
& \textcolor{foxred}{DFA}
& 0.7433
& \textcolor{foxred}{DFA}
& 0.7383
& \textcolor{foxred}{DFA}
& 0.7413
\\
4 
& \textcolor{foxred}{Shimmer.APQ11}
& 0.7668
& \textcolor{foxred}{RPDE}
& 0.7581
& \textcolor{foxred}{RPDE}
& 0.7668
\\
5 
& \textcolor{foxred}{RPDE}
& 0.7707
& Shimmer.dB
& 0.7651
& Shimmer.dB 
& 0.7745
\\
6 
& \textcolor{foxred}{Shimmer.DDA}
& 0.7783
& \textcolor{foxred}{Shimmer.DDA}
& 0.7693
& \textcolor{foxred}{Shimmer.DDA}
& 0.7760
\\
7 
& \textcolor{foxred}{NHR}
& 0.7801
& \textcolor{foxred}{Shimmer.APQ5}
& 0.7696
& \textcolor{foxred}{Shimmer.APQ5}
& 0.7780
\\
8 
& \textcolor{foxred}{Shimmer}
& 0.7822
& \textcolor{foxred}{Shimmer.APQ11}
& 0.7703
& \textcolor{foxred}{Jitter.RAP}
& 0.7781
\\
9 
& \textcolor{foxred}{Shimmer.APQ5}
& 0.7834
& 
& 
& 
& 
\\
10
& \textcolor{foxred}{Jitter.Abs.}
& 0.7838
& 
& 
& 
& 
\\
11
& \textcolor{foxred}{Jitter.RAP}
& 0.7840
& 
& 
& 
& 
\end{tabular}
\end{table}

From Table \ref{Fig.Park}, we observe that MFOCI selects 11 variables for predicting both UPDRS scores of Parkinson patients jointly, while 8 variables are selected for predicting each score individually. 
The most important variables for jointly and individually predicting the UPDRS scores of Parkinson patients are age, sex, and DFA (the signal fractal scaling exponent), making MFOCI plausible in this case as well. The order of the seven most important variables are the same when predicting the individual scores. However, our feature selection recommends a different order from the fourth position on when predicting the scores jointly.
Interestingly, from the fourth or fifth variable on, the values of \(T_n\) increase only slightly. Since \(T\) characterizes conditional independence, and a small increase is associated with only slightly greater squared variability of the conditional distribution functions, one could argue that in all cases a total of 4 or 5 of the 18 characteristics are sufficient to predict one or both of the variables. 
If we compare the estimates \(T_n\) for the different scenarios, we find that they are very similar which can be explained by the strong positive dependence of the individual scores.

\subsection{Multivariate Feature Selection Comparison---Real-World Data Examples}\label{RDE.Subsect.FSIII.App}

Complementing the comparison of MFOCI with KFOCI and Lasso given in Section \ref{RDE.Subsect.FSII}, we now present a comparison of MFOCI also with the bivariate vine copula based quantile regression (BVCQR, in short) proposed by \citet[Section 6]{czado2022}, which allows for a dependence-based feature selection.


\bigskip\noindent
\textit{Predicting daily weather variables}
\\
The underlying data consist of the Seoul weather data set\footnote{UCI machine learning repository \citep{dua2019}; to download the data use 
\url{https://archive.ics.uci.edu/ml/datasets/Bias+correction+of+numerical+prediction+model+temperature+forecast}.}
containing daily observations of two response variables, 
\emph{NextMin: daily minimum air temperature} for the next day and \emph{NextMax: daily maximum air temperature} for the next day, and 13 predictor variables from June 30 to August 30 during 2013-2017 of weather station \emph{central Seoul} (sample size \(n=307\)).
All the variables in this data set exhibit quite high dependencies;
for instance, the Pearson correlation between the response variables NextMin and NextMax is \(0.65\).
\\
In order to achieve comparability with the feature ranking computed in \citep{czado2022}, we are compelled to ignore for the moment the temporal dependence between daily measurements.
Then the feature selection procedure via 
BVCQR ends with 11 predictor variables,
KFOCI (kernel \texttt{rbfdot(1)} \& default kernel, both with default number of nearest neighbours) ends with 6 predictor variables,
Lasso ends with 5 predictor variables,
while MFOCI via \(\overline{T}\) and KFOCI (kernel \texttt{rbfdot(1)} with number of nearest neighbours = 1) end with (different) subsets of no more than 4 variables.
For each subset of selected variables we calculate the (cross validated) mean squared prediction error (MSPE) based on a random forest using \texttt{R}-package \emph{MultivariateRandomForest}. Table \ref{Fig.SeoulWeather} depicts the subsets of chosen predictor variables and lists the MSPEs for each response.

\begin{table}[t!]
\small
\centering
\caption{Chosen predictor variables to predict (NextMax, NextMin) selected via MFOCI, BVCQR, KFOCI and Lasso with MSPEs for each response variable; see Section \ref{RDE.Subsect.FSIII.App} for details.
The variables selected via MFOCI to predict (NextMax, NextMin) are marked in red color.}
\label{Fig.SeoulWeather}
\begin{tabular}{c||l|l|l|l|l}
\makecell[l]{Variables \\ to predict \\ (NextMax, \\ NextMin)}
& \makecell[l]{Feature \\ selection \\ via MFOCI}
& \makecell[l]{Feature \\ selection \\ via BVCQR}
& \makecell[l]{Feature \\ selection \\ via KFOCI, \\ kernel \texttt{rbfdot(1)} \\ \& default kernel}
& \makecell[l]{Feature selection \\ via KFOCI, \\ kernel \\ \texttt{rbfdot(1)} \\ (Knn = 1)}
& \makecell[l]{Feature \\ selection \\ via Lasso}
\\
\hline
\hline
& \textcolor{foxred}{LDAPS\_Tmin}
& \textcolor{foxred}{LDAPS\_Tmin}
& \textcolor{foxred}{LDAPS\_Tmin}
& \textcolor{foxred}{LDAPS\_Tmin}
& \textcolor{foxred}{LDAPS\_Tmax}
\\
& \textcolor{foxred}{LDAPS\_Tmax}
& \textcolor{foxred}{LDAPS\_Tmax}
& \textcolor{foxred}{LDAPS\_Tmax}
& \textcolor{foxred}{LDAPS\_Tmax}
& \textcolor{foxred}{LDAPS\_Tmin}
\\
& \textcolor{foxred}{LDAPS\_CC3}
& LDAPS\_RHmax
& \textcolor{foxred}{LDAPS\_CC1}
& LDAPS\_RHmin
& \textcolor{foxred}{LDAPS\_CC1}
\\
& \textcolor{foxred}{LDAPS\_CC1}
& LDAPS\_WS
& LDAPS\_CC2
& LDAPS\_CC2
& LDAPS\_CC2
\\ 
& 
& Present\_Tmin
& LDAPS\_CC4
& 
& Present\_Tmax
\\
& 
& \textcolor{foxred}{LDAPS\_CC1}
& \textcolor{foxred}{LDAPS\_CC3}
& 
& 
\\ 
& 
& Present\_Tmax
& 
&
& 
\\
& 
& LDAPS\_LH
& 
&
& 
\\
& 
& \textcolor{foxred}{LDAPS\_CC3}
& 
&
& 
\\
& 
& LDAPS\_RHmin
& 
&
& 
\\
& 
& LDAPS\_CC4
& 
&
& 
\\
\hline
\hline
\makecell[c]{MSPE \\ NextMax}
& \makecell[c]{2.55} 
& \makecell[c]{2.40} 
& \makecell[c]{2.57} 
& \makecell[c]{2.81} 
& \makecell[c]{2.63} 
\\
\hline
\makecell[c]{MSPE \\ NextMin}
& \makecell[c]{1.09} 
& \makecell[c]{1.07} 
& \makecell[c]{1.03} 
& \makecell[c]{1.05} 
& \makecell[c]{1.03} 
\end{tabular}
\end{table}

\subsection{Identifying Networks}\label{secnetworks}

Since \(T\) is capable of measuring the strength of (directed) dependence between random vectors of different dimensions, there exist numerous ways for identifying and visualizing networks between variables.
A common and very popular option is to group or cluster the variables according to their similarity, even though the quantification of similarity can be very different.
In \citep{sfx2024cluster}, the authors introduce a hierarchical variable clustering method based on the measure \(\overline{T}\) and hence based on the predictive strength between random vectors. Note that recently developed methods for clustering random variables (see, e.g., \citep{Durante2017, Kojadinovic2004, sfx2021diss}) differ from the well-studied clustering of data, see, e.g., \citep{Duran-1974,Gan-2021} for an overview of cluster analyses of data.

{\begin{figure}
\centering
\caption{Interconnectedness of the three largest banks in the US, Europe and Asia and Pacific, as well as connectedness with the banks Citigroup and Deutsche Bank measured by \(T\,;\) see Subsection \ref{secnetworks} for details.
}
\label{figintcon}
\begin{tikzpicture}[scale=1]

\begin{scope}[shift={(0,0)}]
\tikzstyle{every node} = []
\node (cg) at (-1, 1.2) [shape=circle,draw] {CG};
\node (db) at (1, 1.2) [shape=circle,draw] {DB};
\node (us) at +(-4,3) [shape=circle,draw] {\bf US};
\node (eu) at +(4,3) [shape=circle,draw] {\bf EU};
\node (ap) at +(0,-2) [shape=circle,draw] {\bf AP};
\draw [->,>={angle 45}] (cg) -- (us) node[sloped, pos=0.2, above]{\(0.28\)};
\draw [->,>={angle 45}] (us) .. controls (-2.5,1.7) .. (cg) node[sloped, pos=0.8, below]{\(0.58\)};
\draw [->,>={angle 45}] (eu) -- (cg) node[sloped, pos=0.6, above]{\(0.29\)};
\draw [->,>={angle 45}] (cg) .. controls (.5,2.5) ..  (eu) node[sloped, pos=0.5, above]{\(0.14\)};
\draw [->,>={angle 45}] (cg) -- (ap) node[sloped, pos=0.4, above]{\(0.05\)};
\draw [->,>={angle 45}] (ap) .. controls (-0.8,-.4) .. (cg) node[sloped, pos=0.7, below]{\(0.15\)};
\draw [->,>={angle 45}] (us) --  (db) node[sloped, pos=0.6, above]{\(0.29\)};
\draw [->,>={angle 45}] (db) .. controls (-.5,2.5) ..  (us) node[sloped, pos=0.5, above]{\(0.14\)};
\draw [->,>={angle 45}] (eu) .. controls (2.5,1.7) .. (db) node[sloped, pos=0.8, below]{\(0.46\)};
\draw [->,>={angle 45}] (db) -- (eu) node[sloped, pos=0.2, above]{\(0.21\)};
\draw [->,>={angle 45}] (db) -- (ap) node[sloped, pos=0.4, above]{\(0.05\)};
\draw [->,>={angle 45}] (ap) .. controls (0.8,-.4) .. (db) node[sloped, pos=0.7, below]{\(0.12\)};
\draw [->,>={angle 45}] (us) .. controls (-1,4) and (1,4) .. (eu) node[sloped, pos=0.5, above]{\(0.18\)};
\draw [->,>={angle 45}] (eu) .. controls (2.7,-0.2) and (1.5,-1) .. (ap) node[sloped, pos=0.5, below]{\(0.08\)};
\draw [->,>={angle 45}] (us) .. controls (-2.7,-0.2) and (-1.5,-1) .. (ap) node[sloped, pos=0.5, below]{\(0.07\)};
\draw [->,>={angle 45}] (eu) -- (us) node[sloped, pos=0.5, above]{\(0.14\)};
\draw [->,>={angle 45}] (ap) -- (eu) node[sloped, pos=0.5, below]{\(0.10\)};
\draw [->,>={angle 45}] (ap) -- (us) node[sloped, pos=0.5, below]{\(0.05\)};
\end{scope}
\end{tikzpicture}
\end{figure}}

Another very appealing approach for identifying networks is to use directed graphs for visualizing the strength of directed dependence between (groups of) random variables.
As an example from finance, below we analyze and illustrate the interconnectedness of banks.
\\
We consider the interconnectedness of the \(3\) largest banks in each of the U.S. (US), Europe (EU) and Asia and Pacific (AP), and compare further their connectedness with the \(4\)th largest banks in the US and Europe, which are the Citigroup (CG) and Deutsche Bank (DB), respectively.
The three largest banks of the U.S. comprise JP Morgan (JPM), Bank of America (BAC), and Wells Fargo (WFC). The three largest banks of Europe comprise HSBC, BNP Paribas (BNP), and Cr\'{e}dit Agricole (CAG), and the three largest banks of Asia and Pacific comprise the Industrial and Commercial Bank of China (ICBC), the China Construction Bank Corporation (CCB), and the Agricultural Bank of China (ABC).
\\
For revealing the interconnectedness of the banks, we estimate their predictability via \(T\) from a sample of daily log-returns of the Banks' stock data (in USD) over a time period from April 7, 2011 to December 14, 2022. We assume that the log-returns are i.i.d., which is a standard assumption for stock data, see, e.g., \citep{McNeil-2015}. Figure \ref{figintcon} depicts the values of \(T_n\) for the above described interrelations. We oberserve that
\begin{itemize}
    \item The three largest U.S. banks have a greater influence on the other banks than vice versa.
    \item There is little dependence of the largest U.S. and European banks on the largest Asian banks.
    \item The dependence of the fourth largest U.S. bank (CG) on the three largest U.S. banks is relatively large. Similarly, the fourth-largest European bank is quite dependent on the three largest European banks. Further, dependence of individual banks on Asian banks should not be neglected.
\end{itemize}
We mention that there is high pairwise correlation between the log-returns of the largest U.S., the largest European, and the largest Asian banks, respectively.  For example, Spearman's rank correlation of log-returns among the largest U.S. banks ranges from 0.76 to 0.85. 
Recall that our proposed rank-based measure \(T\) can also measure the influence between multiple input and output variables.

\section{Geometric Interpretation} 
\label{Sec.GeoInterp.}

According to the definitions in \eqref{Tuniv} and \eqref{defmdm}, the quantity \(T\) measures the quadratic variability of conditional distributions and thus relate conditional distributions to unconditional distributions in the \(L^2\) sense. Hence, their properties can be elegantly visualized in a Hilbert space setting.
In the first part of this section, 
we present a geometric interpretation of the most important properties of Azadkia \& Chatterjee's rank correlation coefficient \(T\) (\(q=1\)), namely axioms (A\ref{prop1}) to (A\ref{prop3}),
the information gain inequality (P\ref{prop.IGI}) and the characterization of conditional independence (P\ref{prop.CI}). 
In the second part, further insights into the sequential construction principle underlying \(T\) (\(q > 1\)) follow.

\bigskip
As setting, we consider the Hilbert space \(L^2(\Omega)\) of square-integrable random variables with associated norm \(\lVert Z \rVert_2  := \lVert Z \rVert_{L^2}:= \E [Z^2]\,.\) 
Then 
$$
   T(Y,\XX) 
   = \frac{\int_{\R} \Var(\E[\1_{\{Y\geq y\}}|\XX]) \de P^Y(y)}
           {\int_{\R} \Var(\1_{\{Y\geq y\}}) \de P^Y(y)}
   = \frac{\int_{\R} \lVert\E[\1_F\circ Y|\XX]-\E[\1_F\circ Y]\rVert_2^2 \de P^Y(y)}
           {\int_{\R} \lVert\1_F\circ Y-\E[\1_F\circ Y]\rVert_2^2 \de P^Y(y)}
$$
with  \(F = [y,\infty)\), $y \in \R$,
i.e., $T$ relates the squared distance 
$\lVert\E[\1_F\circ Y|\XX]-\E[\1_F\circ Y]\rVert_2^2$ to the squared distance $\lVert\1_F\circ Y-\E[\1_F\circ Y]\rVert_2^2$.

\subsection{The case $p=2$ and $q=1$} \label{Geom.Sub.1}

We choose one response variable \(Y\) and two predictor variables \(\XX=(X_1,X_2).\) 
From Figure \ref{Figgeomint} we observe the following fundamental properties of \(T\,:\)

\begin{figure} 
\begin{tikzpicture}[scale=0.925]
\coordinate (E) at (0,0); 
\coordinate (Y) at (5*1.5,6*1.5); 
\coordinate (Y0) at (5*1.65,6*1.65); 
\coordinate (Y12) at (6.6*1.5,1.5*1.5); 
\coordinate (Y1) at (5*1.5,0); 
\coordinate (Y2) at (3.5*1.5,2*1.5); 
\coordinate (X1) at (9*1.5,0); 
\coordinate (X2) at (8*1.5,4*1.5); 
\coordinate (X12) at (7.5*1.5,3*1.5); 

\fill[fill=gray!10] (E)--(X1)--(X2);
    \draw[ black, thick] (E) -- (X1);
    \draw[ black, thick] (E) -- (Y);
    \draw[ black, thick] (E) -- (X2);
\filldraw[black] (0,0) circle (2pt) node[anchor=north]{\(\E[\1_F \circ Y]\)};
\draw(X1) node[below]{$\cL^2(\sigma(X_1))$};
\draw(X12) node[below]{$\cL^2(\sigma(X_1,X_2))$};
\draw[dashed, black] (Y) -- (Y12);
\draw[dotted, black] (Y12) -- (Y1);
\draw[black] (Y12) -- (E);
\draw[dashed, black] (Y) -- (Y1);
\draw[dashed, black] (E) -- (Y12) node [midway, above, sloped] (TextNode) {\(\E[\1_F \circ Y|(X_1,X_2)]-\E[\1_F \circ Y]\)};
\draw[dashed, black] (E) -- (Y1) node [midway, below, sloped] (TextNode) {\(\E[\1_F \circ Y|X_1]-\E[\1_F \circ Y]\)};
\draw[dashed, black] (E) -- (Y) node [midway, above, sloped] (TextNode) {\(\1_F \circ Y-\E[\1_F \circ Y]\)};

\filldraw[black] (Y) circle (2pt) node[anchor=south]{\(\1_F\circ Y\)};
\filldraw[black] (Y1) circle (2pt) node[anchor=north]{\(\E[\1_F\circ Y|X_1]\)};
\filldraw[black] (Y12) circle (2pt) node[anchor=west]{\(\E[\1_F \circ Y|(X_1,X_2)]\)};

\draw (5*1.4,0) arc[start angle=180, end angle=90, radius=0.5];
\filldraw[black] (5*1.466,5*0.033) circle (0.5pt);
\draw (6.6*1.425,1.5*1.425) arc[start angle=180, end angle=110, radius=0.5];
\filldraw[black] (6.6*1.466,1.5*1.57) circle (0.5pt);
\draw (5*1.6,0) arc[start angle=0, end angle=45, radius=0.5];
\filldraw[black] (5*1.565,5*0.025) circle (0.5pt);
\end{tikzpicture}
    \caption{Projection of the random variable \(\1_F\circ Y\) onto the plane spanned by \(\XX=(X_1,X_2)\) as well as projections of \(\1_F\circ Y\) and \(\E[\1_F\circ Y|\XX]\) onto the line spanned by \(X_1\,,\) see Section \ref{Geom.Sub.1} for interpretations regarding \(T(Y,\XX)\,.\)
    Lengths of vectors are measured w.r.t. the \(\lVert \cdot \rVert_{L^2}\)-norm of the Hilbert space \(L^2(\Omega)\,.\)}
    \label{Figgeomint}
\end{figure}

\begin{enumerate}
\item[(A1)] 
Since $0\leq \lVert\E[\1_F\circ Y|\XX]-\E[\1_F \circ Y]\rVert_2 \leq \lVert \1_F\circ Y-\E[\1_F\circ Y]\rVert_2 $
for all \(F = [y,\infty)\), $y \in \R$, we immediately obtain \(0\leq T(Y,\XX)\leq 1\). 

\item[(A2)]  
\(T(Y,\XX)=0\) if and only if 
$ 0
=\lVert \E[\1_F\circ Y|\XX]-\E[\1_F \circ Y]\rVert_2^2$ 
for all  \(F = [y,\infty)\), $y \in \R$, which is equivalent to 
$\E[\1_F\circ Y|\XX]=\E[\1_F \circ Y]$ almost surely (i.e., the orthogonal projection of \(\1_F\circ Y\) onto the plane spanned by \(\XX\) is \(\E[\1_F \circ Y])\) for all  \(F = [y,\infty)\), $y \in \R$, which means that  $\XX$ and $Y$ are independent.

\item[(A3)]
 \(T(Y,\XX)=1\) if and only if 
$\lVert\E[\1_F\circ Y|\XX]-\E[\1_F\circ Y]\rVert_2 = \lVert\1_F\circ Y-\E[\1_F\circ Y]\rVert_2$
for all  \(F = [y,\infty)\), $y \in \R$, which is, by Pythagorean theorem, equivalent to 
$0 
= \lVert\E[\1_F\circ Y|\XX]-\1_F\circ Y\rVert_2,$
i.e., $\1_F\circ Y \in \cL^2(\sigma(\XX))$ for all  \(F = [y,\infty)\), $y \in \R$, meaning that $Y$ is perfectly dependent on $\XX$.

\item[(P1)]
\emph{Information gain inequality}: We observe from Figure \ref{Figgeomint} that
\begin{align*}
   \lVert\E[\1_F\circ Y|X_1]-\E[\1_F\circ Y]\rVert_2 \leq  \lVert\E[\1_F\circ Y|(X_1,X_2)]-\E[\1_F\circ Y]\rVert_2 
\end{align*}
for all  \(F = [y,\infty)\), $y \in \R$, implying
\begin{eqnarray*}
    T(Y,X_1) 
    &   =  & \frac{\int_{\R} \lVert\E[\1_F\circ Y|X_1]-\E[\1_F\circ Y]\rVert_2^2 \de P^Y(y)}
           {\int_{\R} \lVert\1_F\circ Y-\E[\1_F\circ Y]\rVert_2^2 \de P^Y(y)} 
    \\
    & \leq & \frac{\int_{\R} \lVert\E[\1_F\circ Y|(X_1,X_2)]-\E[\1_F\circ Y]\rVert_2^2 \de P^Y(y)}
           {\int_{\R} \lVert\1_F\circ Y-\E[\1_F\circ Y]\rVert_2^2 \de P^Y(y)} 
    \\
    &   =  & T(Y,(X_1,X_2))\,.
\end{eqnarray*}

\item[(P2)] 
\emph{Characterization of conditional independence}: 
It holds that 
$T(Y,X_1) = T(Y,(X_1,X_2))$ if and only if 
$\lVert\E[\1_F\circ Y|X_1]-\E[\1_F\circ Y]\rVert_2 = \lVert\E[\1_F\circ Y|(X_1,X_2)]-\E[\1_F\circ Y]\rVert_2$
for all  \(F = [y,\infty)\), $y \in \R$, which is, by Pythagorean theorem for conditional expectations, equivalent to 
$0=\lVert\E[\1_F\circ Y|X_1]-\E[\1_F\circ Y|(X_1,X_2)]\rVert_2$,
i.e., $\E[\1_F\circ Y|(X_1,X_2)] = \E[\1_F\circ Y|X_1]$ almost surely for all \(F = [y,\infty)\), $y \in \R$, 
meaning that $Y$ and $X_2$ are conditionally independent given $X_1$.
\end{enumerate}

\subsection{The case $p=1$ and $q=2$} \label{Geom.Sub.2}

For a geometric interpretation of the individual summands of the multivariate measure of predictability \(T\), we choose two response variables \(\YY=(Y_1,Y_2)\) and one predictor variable \(X\).
Then
\begin{align}\label{eqtqg2}
  T (\YY,X)
  = \frac{T(Y_1,X)}{2 - T(Y_2,Y_{1})}
    + \frac{T(Y_2,(X,Y_1)) - T(Y_2,Y_{1})}{2 - T(Y_2,Y_{1})} \,.
\end{align}
From Figure \ref{Figgeomint2} and the Pythagorean theorem for conditional expectations, the nominator of the second term transforms to
\begin{eqnarray*}
  \lefteqn{T(Y_2,(X,Y_1)) - T(Y_2,Y_{1})}
  \\
  & = & \frac{\int_{\R} \Var(\E[\1_{F} \circ Y_2|(X,Y_1)]) - \Var(\E[\1_{F} \circ Y_2|Y_1]) \de P^{Y_2}(y)}
           {\int_{\R} \Var(\1_{F} \circ Y_2) \de P^{Y_2}(y)}
  \\
  & = & \frac{\int_{\R} \lVert\E[\1_F\circ Y_2|(X,Y_1)]-\E[\1_F\circ Y_2]\rVert_2^2 
             - \lVert\E[\1_F\circ Y_2|Y_1]-\E[\1_F\circ Y_2]\rVert_2^2  \de P^{Y_2}(y)}
           {\int_{\R} \lVert\1_F\circ Y_2-\E[\1_F\circ Y_2]\rVert_2^2 \de P^{Y_2}(y)}
  \\
  & = & \frac{\int_{\R} \lVert\E[\1_F\circ Y_2|(X,Y_1)]-\E[\1_F\circ Y_2|Y_1]\rVert_2^2  \de P^{Y_2}(y)}
           {\int_{\R} \lVert\1_F\circ Y_2-\E[\1_F\circ Y_2]\rVert_2^2 \de P^{Y_2}(y)}
\end{eqnarray*}
for \(F = [y,\infty)\,.\) 
Thus, $T (\YY,X)$ is a combination of the value $T(Y_1,X)$ and the averaged and normalized squared distance between $\E[\1_F\circ Y_2|(X,Y_1)]$ and $\E[\1_F\circ Y_2|Y_1]$.

\begin{figure} 
\begin{tikzpicture}[scale=0.925]
\coordinate (E) at (0,0); 
\coordinate (Y) at (5*1.5,6*1.5); 
\coordinate (Y0) at (5*1.65,6*1.65); 
\coordinate (Y12) at (6.6*1.5,1.5*1.5); 
\coordinate (Y1) at (5*1.5,0); 
\coordinate (Y2) at (3.5*1.5,2*1.5); 
\coordinate (X1) at (9*1.5,0); 
\coordinate (X2) at (8*1.5,4*1.5); 
\coordinate (X12) at (7.5*1.5,3*1.5); 

\fill[fill=gray!10] (E)--(X1)--(X2);
    \draw[ black, thick] (E) -- (X1);
    \draw[ black, thick] (E) -- (Y);
    \draw[ black, thick] (E) -- (X2);
\filldraw[black] (0,0) circle (2pt) node[anchor=north]{\(\E[\1_F \circ Y_2]\)};
\draw(X1) node[below]{$\cL^2(\sigma(Y_1))$};
\draw(X12) node[below]{$\cL^2(\sigma(X,Y_1))$};
\draw[dashed, black] (Y) -- (Y12);
\draw[dotted, black] (Y12) -- (Y1);

\draw[black] (Y12) -- (E);
\draw[dashed, black] (Y) -- (Y1);

\draw[dashed, black] (E) -- (Y12) node [midway, above, sloped] (TextNode) {\(\E[\1_F \circ Y_2|(X,Y_1)]-\E[\1_F \circ Y_2]\)};
\draw[dashed, black] (E) -- (Y1) node [midway, below, sloped] (TextNode) {\(\E[\1_F \circ Y_2|Y_1]-\E[\1_F \circ Y_2]\)};
\draw[dashed, black] (E) -- (Y) node [midway, above, sloped] (TextNode) {\(\1_F \circ Y_2-\E[\1_F \circ Y_2]\)};

\filldraw[black] (Y) circle (2pt) node[anchor=south]{\(\1_F\circ Y_2\)};
\filldraw[black] (Y1) circle (2pt) node[anchor=north]{\(\E[\1_F\circ Y_2|Y_1]\)};
\filldraw[black] (Y12) circle (2pt) node[anchor=west]{\(\E[\1_F \circ Y_2|(X,Y_1)]\)};

\draw (5*1.4,0) arc[start angle=180, end angle=90, radius=0.5];
\filldraw[black] (5*1.466,5*0.033) circle (0.5pt);
\draw (6.6*1.425,1.5*1.425) arc[start angle=180, end angle=110, radius=0.5];
\filldraw[black] (6.6*1.466,1.5*1.57) circle (0.5pt);
\draw (5*1.6,0) arc[start angle=0, end angle=45, radius=0.5];
\filldraw[black] (5*1.565,5*0.025) circle (0.5pt);
\end{tikzpicture}
    \caption{Illustration of the value $T(Y_2,(X,Y_1)) - T(Y_2,Y_{1})$ in \eqref{eqtqg2} for interpreting \(T((Y_1,Y_2),X)\,,\) see Subsection \ref{Geom.Sub.2}. 
    Lengths of vectors are measured w.r.t. the \(\lVert \cdot \rVert_{L^2}\)-norm of the Hilbert space \(L^2(\Omega)\,.\)}
    \label{Figgeomint2}
\end{figure}

\end{appendices}

\end{document}